\documentclass[a4paper,12pt]{article}
\usepackage[left=30mm, top=30mm, right=30mm, bottom=30mm, nohead, nofoot]{geometry}
	
    \usepackage{float}
	\usepackage{amsmath,amsfonts,amssymb,amsthm,mathtools} 
	\usepackage{icomma} 

	\usepackage{wrapfig} 
	\usepackage{rotating}
    \usepackage{subfigure}
	\usepackage{hhline}
	\usepackage{lscape}
	\usepackage[usenames,dvipsnames,svgnames,table,rgb]{xcolor}
	\usepackage{dynkin-diagrams}	
	\usepackage{enumitem}
	\setlist{nolistsep, leftmargin=5mm}
	
	\usepackage{array,tabularx,tabulary,booktabs} 
	\usepackage{longtable} 
	\usepackage{multirow} 
	
	\usepackage{multicol} 
	
	\usepackage{extsizes} 
	\usepackage{geometry} 
    \usepackage{hyperref}
	
	\usepackage{setspace} 
	\usepackage{lastpage} 
	\usepackage{soul} 
	\usepackage{bbding}
	\usepackage{hyperref}
	\usepackage[usenames,dvipsnames,svgnames,table,rgb]{xcolor}
	\hypersetup{ 
	colorlinks=true, 
	linkcolor=black, 
	citecolor=black, 
	filecolor=black, 
	urlcolor=ForestGreen 
	}
	
	\usepackage{environ}
	\makeatletter
	\newsavebox{\measure@tikzpicture}
	\NewEnviron{scaletikzpicturetowidth}[1]{%
	\def\tikz@width{#1}%
	\begin{lrbox}{\measure@tikzpicture}%
	\BODY
	\end{lrbox}%
	\pgfmathparse{#1/\wd\measure@tikzpicture}%
	\BODY
	}
	\makeatother
	
    \usepackage{verbatim}
	
	\usepackage{attachfile2}
	\attachfilesetup{appearance=true,
	color=0 0 0
	}
	
	\usepackage{forest} 
	\usepackage{vowel} 
	\usepackage{philex} 

	\usepackage{sectsty}
	\sectionfont{\normalsize}
	\subsectionfont{\normalsize}
	\usepackage{titlesec}
	\newlength{\bibitemsep}
	\newlength{\bibparskip}
	\let\oldthebibliography\thebibliography
	\renewcommand\thebibliography[1]{%
	\oldthebibliography{#1}%
	\setlength{\parskip}{\bibitemsep}%
	\setlength{\itemsep}{\bibparskip}%
	}
\usepackage{tikz}
\usepackage{pstricks}
\usetikzlibrary{arrows,positioning,shapes.geometric}
\usepackage{alltt}
\usepackage{amsthm}
\usepackage{amssymb}
\usepackage{amsmath}
\usepackage{tikz-cd}
\usepackage{hyperref}
\usepackage{xcolor}
\newtheorem*{Proposition}{Proposition}
\newtheorem{theorem}{Theorem}[section]
\newtheorem*{TheoremA}{Theorem A}
\newtheorem*{TheoremB}{Theorem B}
\newtheorem*{TheoremC}{Theorem C}
\newtheorem*{Theorem}{Theorem}
\newtheorem*{Corollary}{Corollary}
\newtheorem*{Example}{Example}
\newtheorem{lemma}[theorem]{Lemma}
\newtheorem{corollary}[theorem]{Corollary}
\newtheorem{proposition}[theorem]{Proposition}

\newtheorem{remark}[theorem]{Remark}
\newtheorem{definition}[theorem]{Definition}
\newtheorem*{notation}{Notation}

\newtheorem{example}[theorem]{Example}

\newcommand{\Cr}[1]{\mathcal{#1}} 
\newcommand{\tens}{\otimes} 
\newcommand{\Sing}{\operatorname{Sing}} 
\newcommand{\Aut}{\operatorname{Aut}} 

\begin{document}
\title{On the action of Bender-Knuth generators of cactus group on the set of short semi-standard Young tableaux}
\author{Igor Svyatnyy\footnote{\textsc{Igor \, K. \, Svyatnyy: \, National\, Research\, University\, Higher\, School\, of\, \,Economics, \,Usacheva\, Street 6,\, 119048,\, Moscow,\, Russia.}\, \texttt{igor.svyatnyy@outlook.com}}}

\definecolor{zzttqq}{rgb}{0.26666666666666666,0.26666666666666666,0.26666666666666666}
\definecolor{cqcqcq}{rgb}{0.7529411764705882,0.7529411764705882,0.7529411764705882}
\maketitle

\begin{abstract}
    In the article by Michael Chmutov, Max Glick and Pavel Pylyavskii \cite{Chmutov} the action of the cactus group $C_N$ on the set of semi-standard Young tableaux filled with the numbers from $1$ to $N$ was defined. Namely, they constructed the set of generators (we rightfully call them Bender-Knuth generators) of the cactus group and a group homomorphism from $C_N$ to Berenstein-Kirillov group $BK_N$ (cf. \cite{Berenstein_Kirillov}), which sends these generators to the Bender-Knuth involutions on the set of semi-standard Young tableaux. 
    
    In \cite{Henriques_Kamnitzer} Andre Henriques and Joel Kamnitzer defined a natural action of cactus group $C_N$ on the tensor product of $N$ normal crystals via commutors. The term \textit{crystal} was introduced by Kashiwara in \cite{Kashiwara_1993} and later appeared in his review article \cite{Kashiwara_1995}. A normal crystal can be thought of as a combinatorial model of an integrable representation of quantized universal enveloping algebra $U_q(\mathfrak{g})$, where $\mathfrak{g}$ is a semisimple Lie algebra (or even Kac-Moody Lie algebra).  
    
    By applying the result of Henriques and Kamnitzer to the $N$-th tensor power of a spinor crystal $\Cr{B}_S$ I defined the action of cactus group $C_N$ on the set of short semi-standard Young tableaux filled with the numbers $1, 2, \ldots, N$ in \cite{Svyatnyy}. A semi-standard Young tableau is called \textit{short} if the number of cells in the first two columns with the numbers $\leqslant N$ is less than or equal to $N$. Howe duality between the Lie algebra $\mathfrak{o}_{2n}$ and $O_N$ (cf. \cite{Howe}) gives the natural bijection between the set of short semi-standard Young tableaux of a fixed shape and the set $T_\lambda^N$ defined in \cite{Svyatnyy}, which indexes the connected components of $\Cr{B}_S^{\tens N}$ isomorphic to $\Cr{B}_\lambda$.
    The set of short semi-standard Young tableaux obviously forms a subset inside the set of semi-standard Young tableaux. The purpose of this paper is to explicitly compute the action of Bender-Knuth generators of cactus group $C_N$ on the set of short semi-standard Young tableaux defined in \cite{Svyatnyy} and compare it with their action on the set of semi-standard Young tableaux defined in \cite{Chmutov}.
    

\end{abstract}

\tableofcontents

\section*{Introduction.}
\addcontentsline{toc}{section}{Introduction.}
\subsection*{The $\mathfrak{g}$-Crystals category.}
\addcontentsline{toc}{subsection}{The \textbf{\underline{$\mathfrak{g}$-Crystals}} category}

Let $\mathfrak{g}$ be a finite-dimensional complex semisimple Lie algebra. Crystals were introduced by Kashiwara \cite{Kashiwara_1993}, \cite{Kashiwara_1995} as a combinatorial structure arising from the representations of the quantum group $U_q(\mathfrak{g})$. Informally, a crystal is a directed graph with the edges labeled by simple roots of $\mathfrak{g}$ and the vertices labeled by the weight lattice of $\mathfrak{g}$.  Crystals can be tensored together to produce a crystal whose underlying set (i.e., the set of the vertices of the corresponding graph) is the product of underlying sets of the multiples. Tensor product of crystals is associative but is not symmetric as the map $\operatorname{flip}: \Cr{B}_1 \tens \Cr{B}_2 \rightarrow \Cr{B}_2 \tens \Cr{B}_1$ which sends $(b_1, b_2)$ to $(b_2, b_1)$ is not even a morphism of crystals. 

In \cite{Henriques_Kamnitzer} Andre Henriques and Joel Kamnitzer defined the \textbf{\underline{$\mathfrak{g}$-Crystals}} category.  Its objects are the crystals arising from the integrable $U_q(\mathfrak{g})$-modules that belong to the semisimple category $\Cr{O}_{int}(\mathfrak{g})$ (cf. Definition $2.3$ in \cite{Kashiwara_1995}). This category is closed under the tensor product operation for crystals. One of the main results of \cite{Henriques_Kamnitzer} is the construction of a family of functorial isomorphisms $\sigma_{\mathcal{B}_1, \mathcal{B}_2}: \mathcal{B}_1 \otimes \mathcal{B}_2 \rightarrow \mathcal{B}_2 \otimes \mathcal{B}_1$ for $\Cr{B}_1, \Cr{B}_2 \in \operatorname{Ob}(\textbf{\underline{$\mathfrak{g}$-Crystals}})$. These functorial isomorphisms are referred to as \textbf{\textit{commutors}}. 

In this paper, for the most part we use the term \textit{\textbf{crystal}} for an object of the \textbf{\underline{$\mathfrak{g}$-Crystals}} category.  By definition, the category \textbf{\underline{$\mathfrak{g}$-Crystals}} is semisimple. Simple objects of this category are \textbf{\textit{the highest weight crystals}} (i.e., crystals with a root vertex). They are labeled by the set of dominant integral weights of $\mathfrak{g}$, denoted by $P_+$. The family of simple objects of the category \textbf{\underline{$\mathfrak{g}$-Crystals}} is denoted by $\mathbb{B} = \{\mathcal{B}_\lambda, \lambda \in P_+\}$. Naturally, the crystal $\Cr{B}_\lambda$ arises from the irreducible integrable $U_q(\mathfrak{g})$-module $V_\lambda \in \operatorname{Ob}(\Cr{O}_{int}(\mathfrak{g}))$ (cf. Theorem $2.1$ in \cite{Kashiwara_1995}). 

There are no non-zero morphisms between the highest weight crystals $\mathcal{B}_\mu$ and $\mathcal{B}_\lambda$ if $\lambda \neq \mu$ and only the identity morphism otherwise. This means that the automorphisms of a given crystal can only permute its highest weight components of the same highest weight. It will be very important for us. 

\subsection*{The spinor crystal.}
\addcontentsline{toc}{subsection}{The spinor crystal.}
We set $\mathfrak{g}$ to be equal to the classical Lie algebra $D_n$. Let $(\epsilon_1, \epsilon_2, \ldots, \epsilon_n)$ be the orthonormal basis of the dual of the Cartan subalgebra of $D_n$. Set $\Lambda_{n-1} = \frac{1}{2}(\epsilon_1 + \ldots + \epsilon_{n-1} - \epsilon_n)$ and $\Lambda_n = \frac{1}{2}(\epsilon_1 + \ldots + \epsilon_{n-1} + \epsilon_n)$. The \textbf{\textit{spinor crystal}}, denoted by $\mathcal{B}_S$, is an object of the \underline{\textbf{$D_n$-crystals}} category equal to the direct sum of two highest weight crystals: $\mathcal{B}_S = \mathcal{B}_{\Lambda_{n-1}} \oplus \mathcal{B}_{\Lambda_n}$. We denote the set of weights of the spinor crystal by $P[S]$. 

Consider the $N$-th tensor power of the spinor crystal $\mathcal{B}_S^{\otimes N}$. In order to decompose it into the sum of the highest weight crystals we need the following key proposition. 

\begin{Proposition}[Proposition \ref{prop_3.2}]
    Let $\mathcal{B}_{\lambda}$ be the highest weight crystal of the highest weight $\lambda \in P_+$. Then the following decomposition formula holds: 
    \begin{equation}
       \mathcal{B}_\lambda  \otimes \mathcal{B}_S  = \bigoplus_{\substack{\mu \in P[S],  \\ \lambda + \mu \in P_+}} \mathcal{B}_{\lambda + \mu}.
    \end{equation}

    Moreover, the highest weight element of the connected component $\mathcal{B}_{\lambda + \mu}$  equals $b_{\lambda} \otimes b_{\mu}$, where $b_\lambda$ is the highest weight element of the crystal $\mathcal{B}_{\lambda}$ and $b_\mu$ is the element of the weight $\mu$ in the spinor crystal $\mathcal{B}_S$. 

\end{Proposition}

Denote by $\Delta^N \subset P_+$ the set of the highest weights of the highest weight crystals in the decomposition of $\mathcal{B}_S^{\otimes N}$ into the sum of simple crystals: $$\mathcal{B}_S^{\otimes N} = \bigoplus_{\lambda \in \Delta^N} \mathcal{B}_\lambda^{m_{\lambda}},$$ where $m_\lambda \in \mathbb{Z}_{>0}$. 

Using the proposition,  by induction one can obtain the following decomposition formula.

\begin{Corollary}[Corollary \ref{cor_3.1}]
    \begin{equation} {\label{eq_1.1}}
        \mathcal{B}_{S}^{\otimes N} = \bigoplus_{(\mu_1, \mu_{2}, \ldots, \mu_N) \in T^{N}} \mathcal{B}_{\mu_1 + \mu_{2} + \ldots + \mu_N},
    \end{equation}
where 
\begin{equation}
    T^N := \Big\{ (\mu_1, \mu_{2}, \ldots, \mu_N) \hspace{1mm} | \hspace{1mm} \mu_k \in P[S], \sum_{i=1}^{k} \mu_i \in P_+, \forall k \in \{1, 2, \ldots, N\} \Big\}
\end{equation}
\end{Corollary}
It is clear that \begin{gather} \label{delta} \Delta^N = \Big\{ \sum_{i=1}^{N} \mu_i  \hspace{1mm} | \hspace{1mm} (\mu_1, \mu_{2}, \ldots, \mu_N) \in T^N \Big\}\end{gather}
We present $T^N$ as the disjoint union: 
\begin{gather} \label{tlambdaN}
T^{N} = \coprod_{\lambda \in \Delta^{N}} T_\lambda^{N}, \hspace{1mm} \text{where}\\
T_\lambda^N := \Big\{(\mu_1,\mu_{2}, \ldots, \mu_N) \in T^N \hspace{1mm} | \sum_{i=1}^{N} \mu_i = \lambda \Big\} \subset T^N.
\end{gather}
Notice that the set $T_\lambda^N$ indexes the summands on the right-hand side of the formula \ref{eq_1.1} isomorphic to $\mathcal{B}_{\lambda}$.

\subsection*{Cell tables and short semi-standard Young tableaux.}
\addcontentsline{toc}{subsection}{Cell tables and short semi-standard young tableaux.}
It turns out the sets $\Delta^N$ and $T_\lambda^N$ have nice combinatorial interpretations. This matter was discussed in detail in my previous study \cite{Svyatnyy}. This paper uses some of the key results of \cite{Svyatnyy}.

The set $\Delta^N$ can be interpreted as the set of regular cell diagrams of the height $n$ and length $N$, denoted by $\mathfrak{D}(N,n)$. 

A \textbf{\textit{Regular cell diagram}} denoted by $D(\textbf{l}, \textbf{r}) \in \mathfrak{D}(N,n)$ is a union of cells and a vertical line called the \text{axis} of the cell diagram constructed by two sets of nonnegative integers $\textbf{l} = (l_1, l_2, \ldots l_n)$, $\textbf{r} = (r_1,r_2, \ldots r_n)$, which satisfy the following conditions:

\begin{itemize}
 \item $r_i + l_i = N, \hspace{1mm} \forall i \in \{1, \ldots n\}$
 \item $r_1 \geqslant r_2 \geqslant \ldots \geqslant r_n$
 \item $r_{n-1} \geqslant l_{n}$
\end{itemize}

The construction looks as follows. In the $i$-th
row to the left (resp. right) of the axis there are $l_i$ (resp. $r_i$) cells. 

\begin{Example}[Example \ref{ex_5.1}]
    As an example we draw the regular cell diagram $D(\textbf{l}, \textbf{r}) \in \mathfrak{D}(7,4)$, where $\textbf{r} =  (5, 4, 4, 3), \hspace{2mm} \textbf{l} = (2, 3, 3, 4)$.

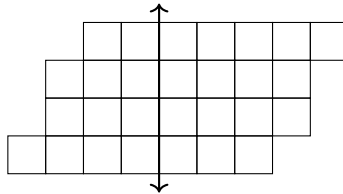
\centering 
\begin{tikzpicture}[scale = 0.5]
\draw[thick,<->] (6,-0.5) -- (6, 4.5) node[anchor=north west] {};
\draw (2,0) rectangle (3,1);
\draw (3,0) rectangle (4,1);
\draw (4,0) rectangle (5,1);
\draw (5,0) rectangle (6,1);
\draw (6,0) rectangle (7,1);
\draw (7,0) rectangle (8,1);
\draw (8,0) rectangle (9,1);
\draw (3,1) rectangle (4,2);
\draw (4,1) rectangle (5,2);
\draw (5,1) rectangle (6,2);
\draw (6,1) rectangle (7,2);
\draw (7,1) rectangle (8,2);
\draw (8,1) rectangle (9,2);
\draw (9,1) rectangle (10,2);
\draw (3,2) rectangle (4,3);
\draw (4,2) rectangle (5,3);
\draw (5,2) rectangle (6,3);
\draw (6,2) rectangle (7,3);
\draw (7,2) rectangle (8,3);
\draw (8,2) rectangle (9,3);
\draw (9,2) rectangle (10,3);
\draw (5,3) rectangle (4,4);
\draw (6,3) rectangle (5,4);
\draw (7,3) rectangle (6,4);
\draw (8,3) rectangle (7,4);
\draw (9,3) rectangle (8,4);
\draw (10,3) rectangle (9,4);
\draw (11,3) rectangle (10,4);
\end{tikzpicture}
\end{Example}

\begin{Proposition}[Proposition \ref{prop_5.2}] 
There exists a natural bijection, denoted by $$\mathcal{K}_N: \Delta^N \longrightarrow \mathfrak{D}(N,n),$$ defined by the following formula 
\begin{equation}\mathcal{K}_N(\lambda) = D(\textbf{l}(\lambda, N), \textbf{r}(\lambda, N)),
\end{equation}
where
\begin{gather}
\textbf{r}(\lambda, N) = \Big(\frac{N}{2} + \lambda_i \Big)_{1 \leqslant i \leqslant n} \hspace{1mm}, \\
\textbf{l}(\lambda, N) = \Big(\frac{N}{2} - \lambda_i \Big)_{1 \leqslant i \leqslant n} \hspace{1mm}.
\end{gather}   
\end{Proposition}
For brevity we denote the image of the element $\lambda \in \Delta^N$ under the map $\mathcal{K}_N$ by $D_\lambda^N$. The set $T_\lambda^N$ can be interpreted as the set of regular cell tableaux of the shape $D^{N}_\lambda \in \mathfrak{D}(N,n)$ denoted by $\mathfrak{Ctab}(D_\lambda^{N})$.

A \textbf{\textit{Regular cell tableau}} of the shape $D^{(N)} \in \mathfrak{D}(N,n)$ is a sequence of nested regular cell diagrams $(D^{(1)}, D^{(2)}, \ldots ,D^{(N)})$ such that $D^{(i)} \in \mathfrak{D}(i, n)$ for all $i \in \{1, 2, \ldots N\}$.

\begin{Proposition}[Proposition \ref{prop_5.4}]
\label{bj_t&ct}
There exists a natural bijection  
$$\mathcal{I}_\lambda: T_\lambda^{   N}\longrightarrow \mathfrak{Ctab}(D_\lambda^{N}), $$
given by the formula
\begin{gather} 
\label{bijection_t_cell_tables}\mathcal{I}_{\lambda}:  (\mu_1, \mu_{2}, \ldots \mu_{N}) \mapsto (\mathcal{K}_1(\mu_1), \mathcal{K}_2(\mu_1 + \mu_2), \ldots, \mathcal{K}_i\Big(\sum_{j=1}^{i} \mu_j\Big), \ldots, \mathcal{K}_N (\lambda)).\end{gather}
\end{Proposition}

It was also observed in \cite{Svyatnyy} that the notion of a regular cell diagram is connected with the notion of a short Young diagram. 

A \textbf{ \textit{short Young diagram}} is a Young diagram with the total length of the first two columns not greater than $N$. We denote the set of short Young diagrams with no more than $n$ columns by $\operatorname{SYD}(N, n)$.

The following proposition was proved in \cite{Svyatnyy}.

\begin{Proposition}[Theorem \ref{th_5.1}]
    Consider a map 
    $$\mathcal{F}_N: \mathfrak{D}(N, n) \rightarrow \operatorname{SYD}(N, n),$$
    defined in the following way. If $n$ is even, then \begin{gather}
\mathcal{F}_N (D(\textbf{l}, \textbf{r})) =  (l_n, l_{n-1}, \ldots, l_1)^{t},
\end{gather}
otherwise
\begin{gather}
\mathcal{F}_N (D(\textbf{l}, \textbf{r})) =  (r_n, l_{n-1}, \ldots, l_1)^{t},
\end{gather} 

The map $\mathcal{F}_N$ is a bijection. 

\end{Proposition}

The bijection $\mathcal{F}_N$ comes from the Howe duality between Lie algebra $D_n$ and orthogonal Lie group $O_N$. One can take the $N$-th tensor power of the spinor representation  $(V_{\Lambda_{n-1}} \oplus V_{\Lambda_n})^{\otimes N}$ of Lie algebra $D_n$ and decompose it into the direct sum of simple $D_n$-modules: $$(V_{\Lambda_{n-1}} \oplus V_{\Lambda_n})^{\otimes N} = \bigoplus_\lambda U_\lambda \tens L_\lambda.$$ Here we denoted by $V_\lambda$ the highest weight module of highest weight $\lambda$. Then by Howe duality \cite{Howe} the multiplicity space $U_\lambda$ of the summand $V_\lambda, \hspace{1mm} \lambda \in \Delta^N$ has the natural structure of an irreducible representation of the group $O_N$. It is well known that such representations are indexed by the short Young diagrams \cite{Fulton}. Bijection $\mathcal{F}_N$ sends the cell diagram $D_\lambda^N \in \mathfrak{D}(N,n)$, corresponding to the highest weight $\lambda \in \Delta^N$ to the short Young diagram corresponding to the isomorphism class of the multiplicity space $U_\lambda$ as an $O_N$-module. The latter was also proved in \cite{Svyatnyy}.

Using the previous proposition, we can present another combinatorial interpretation of the set $T_\lambda^N$.  

A \textbf{\textit{short semi-standard Young tableau}} of the shape $\nu^{(N)}$ is a sequence of nested short Young diagrams $(\nu^{(1)},\nu^{(2)}, \ldots, \nu^{(N)})$, such that $\nu^{(i)} \in \operatorname{SYD}(i, n)$ and 
$\nu^{(i+1)} - \nu^{(i)}$ is a horizontal strip for all $1 \leqslant i < N$.

We denote by $\operatorname{SSSYT}(\nu, N)$ the set of short semi-standard Young tableaux of the shape $\nu$. 

\begin{Proposition}[Proposition \ref{prop_5.5}]
Set $\nu = \mathcal{F}_N(D_\lambda^N)$. Then the map
$$\mathcal{Y}_\lambda: \mathfrak{Ctab}(D_\lambda^{N}) \longrightarrow \operatorname{SSSYT}(\nu, N),$$ defined by the following formula
\begin{gather}
    \mathcal{Y}_\lambda: (D^{(1)}, D^{(2)}, \ldots, D^{(N)} = D_\lambda^N) \mapsto (\mathcal{F}_1(D^{(1)}), \mathcal{F}_2(D^{(2)}), \ldots, \mathcal{F}_N(D_{\lambda}^{N}) = \nu).
\end{gather}
is a bijection between the set of regular cell tableaux of the shape $D_\lambda^{N}$ and the set of short semi-standard Young tableaux of the shape $\nu = \mathcal{F}_N(D_\lambda^N)$.
\end{Proposition}

    The latter bijection comes from the identification of two bases of the multiplicity space $U_\lambda$. One of them is the Gelfand-Tsetlin basis with respect to the chain of nested orthogonal groups. This basis is naturally indexed by the set $\operatorname{SSSYT}(\nu, N)$ of  short semi-standard Young tableaux of the shape $\nu = \mathcal{F}_N(D_\lambda^N)$. The other basis of the multiplicity space consists of the highest weight vectors of the summands $V_\lambda$ that appear if we consecutively (by taking tensor product of the previous decomposition with  $V_{\Lambda_{n-1}} \oplus V_{\Lambda_{n}}$) decompose $(V_{\Lambda_{n-1}} \oplus V_{\Lambda_{n}})^{\otimes N}$ into the sum of simple modules. This basis is naturally indexed by the set $T_\lambda^N$ or $\mathfrak{Ctab}(D_\lambda^N)$. For details see my previous paper \cite{Svyatnyy}.

\subsection*{Cactus group.}
\addcontentsline{toc}{subsection}{Cactus group.}
Let $C_N$ be a group with the generators $ \mathbf{s}_{p, q}$, $1 \leqslant p < q \leqslant N$ and relations
$$  \mathbf{s}_{p,q}^{2} = e;$$
$$ \mathbf{s}_{p_1, q_1}  \mathbf{s}_{p_2, q_2} =  \mathbf{s}_{p_2, q_2}  \mathbf{s}_{p_1, q_1}, \hspace{1mm} \text{if} \hspace{2mm} q_1 < p_2;$$
$$ \mathbf{s}_{p_1, q_1}  \mathbf{s}_{p_2, q_2}  \mathbf{s}_{p_1, q_1} =  \mathbf{s}_{p_1 + q_1 - q_2, p_1 + q_1 - p_2}, \hspace{1mm} \text{if} \hspace{2mm} p_1 \leqslant p_2 < q_2 \leqslant q_1.$$

In the work of Andre Henriques and Joel Kamnitzer \cite{Henriques_Kamnitzer} the group $C_N$ is called the \textbf{\textit{cactus group}}. Henriques and Kamnitzer showed that the cactus group naturally appears in the coboundary categories. 
\textbf{\textit{Coboundary category}} is a tensor category equipped with the set of functorial isomorphisms $\sigma_{X,Y}: X \otimes Y  \rightarrow  Y \otimes X$ , satisfying certain axioms. The category \textbf{\underline{$\mathfrak{g}$-Crystals}} equipped with the family of functorial isomorphisms $\{\sigma_{\Cr{B}_1,\Cr{B}_2} |\hspace{1mm} \Cr{B}_1, \Cr{B}_2 \in \operatorname{Ob}( \textbf{\underline{$\mathfrak{g}$-Crystals}})\}$ defined in \cite{Henriques_Kamnitzer} is coboundary. 

Let $\mathcal{B}$ be an arbitrary crystal. In \cite{Henriques_Kamnitzer}, an action of the cactus group $C_N$ on the tensor product $\mathcal{B}^{\otimes N}$ was introduced. The generators $\mathbf{s}_{p,q}$ of $C_N$ act via suitable compositions of commutors, which by construction are the isomorphisms of crystals. Consequently, the cactus group acts on $\mathcal{B}^{\otimes N}$ by crystal automorphisms. These automorphisms permute the connected (irreducible) components of $\mathcal{B}^{\otimes N}$ of the same highest weight.

Hence, by setting $\mathfrak{g} = D_n$ and $\Cr{B} = \Cr{B}_S$  we obtain an action of the cactus group $C_N$ on the sets $T_\lambda^N$, $\mathfrak{Ctab}(D_\lambda^N)$, $\operatorname{SSSYT}(\nu, N)$ for each $\lambda \in \Delta^N$ and $ \nu =\mathcal{F}_N(D_\lambda^N)$ by permutations, as each of these sets indexes the connected (irreducible) components of $\Cr{B}_S^{\tens N}$ of the highest weight $\lambda$. Computing this action explicitly is the main goal of this work.

\subsection*{Connection with the work of Chmutov, Glick and Pylyavskii}
\addcontentsline{toc}{subsection}{Connection with the Chmutov's work}

In \cite{Chmutov} Michael Chmutov, Max Glick and Pavel Pylyavskii constructed the homomorphism from the cactus group $C_N$ to Berenstein-Kirillov group $BK_N$ (cf.\cite{Berenstein_Kirillov}) and by doing so defined the action of the cactus group on the set of semi-standard Young tableaux filled with the numbers from $1$ to $N$. 

A short semi-standard Young tableau is obviously a semi-standard Young tableau and hence it is natural to ask how our action of the cactus group on the set of short semi-standard Young tableaux is connected with the action of the cactus group on the set of semi-standard Young tableaux defined in \cite{Chmutov}. 

It turns out that the action of Chmutov-Glick-Pylyavskii on semi-standard Young tableaux can not be restricted to the subset of the short semi-standard Young tableaux. Nonetheless, there are some similarities between these two actions and we are going to observe them.  

We need the alternative description of the cactus group acquired by Chmutov, Glick and Pylyavskii in the same work \cite{Chmutov}. 

\begin{Theorem}[Theorem \ref{th_6.1}]
Let $G_N$ be the group with generators $\mathbf{t}_i$, $i \in \{1, 2, \ldots N-1\}$ and relations \begin{gather}
\mathbf{t}_i^{2} = 1 \\
\mathbf{t}_i\mathbf{t}_j =\mathbf{t}_j\mathbf{t}_i, \hspace{1mm} \text{if} \hspace{2mm}  |i - j | > 1 \\
(\mathbf{t}_i\mathbf{s}_{k-1}\mathbf{s}_{k-j}\mathbf{s}_{k-1})^{2} = 1, \hspace{1mm} \text{if} \hspace{2mm} i+1 < j < k , \hspace{1mm} \text{where}\\ 
\mathbf{s}_i = \mathbf{t}_1(\mathbf{t}_2\mathbf{t}_1)\ldots (\mathbf{t}_i\mathbf{t}_{i-1}\ldots \mathbf{t}_1)
\end{gather}
Then, $G_N$ is isomorphic to the cactus group $C_N$.
Isomorphism is given by  
$$\mathbf{t}_1 \mapsto \mathbf{s}_{1,2}, \hspace{1mm} \mathbf{t}_2 \mapsto \mathbf{s}_{1,2}\mathbf{s}_{1,3}\mathbf{s}_{1,2}, \hspace{1mm} \mathbf{t}_i \mapsto 
\mathbf{s}_{1,i}\mathbf{s}_{1,i+1}\mathbf{s}_{1,i}\mathbf{s}_{1, i-1}, \hspace{1mm} \text{if} \hspace{1mm} i> 2$$ in one direction
$$\mathbf{s}_{i,j} \mapsto \mathbf{s}_{j-1}\mathbf{s}_{j-i}\mathbf{s}_{j-1}, \hspace{1mm} \text{where} \hspace{1mm} \mathbf{s}_i = \mathbf{t}_1(\mathbf{t}_2\mathbf{t}_1)\ldots (\mathbf{t}_i\mathbf{t}_{i-1}\ldots \mathbf{t}_1)$$
in the other one.
\end{Theorem}

It turns out that it is much easier to compute the action of the generators $ t_i$ on the sets $T_\lambda^N$, $\mathfrak{Ctab}(D_\lambda^N)$, $\operatorname{SSSYT}(\nu, N)$ rather than the action of the generators $\mathbf{s}_{p,q}$. 

For instance, in Chmutov's case the generators $\mathbf{t}_i$ act on the set of semi-standard Young tableaux by Bender-Knuth involutions. Hence, we are going to call the generators $\mathbf{t}_i$ \textbf{\textit{Bender-Knuth generators}} of the cactus group $C_N$.

\subsection*{Main results.}
\addcontentsline{toc}{subsection}{Main results.}

The main result of this work is the explicit description of the action of the cactus group $C_N$ on the set $\operatorname{SSSYT}(\nu, N)$ of short semi-standard Young tableaux of fixed shape $\nu \in \operatorname{SYD}(N,n)$, which comes from the natural action of the cactus group on $\Cr{B}_S^{\tens N}$ by commutors. As was pointed out earlier, this set is in bijection with the sets $\mathfrak{Ctab}(D_\lambda^N)$ and $T_\lambda^N$, where $\nu = \mathcal{F}_N(D_\lambda^N)$. 

We are going to explicitly describe the action of the Bender-Knuth generators $\mathbf{t}_i \in  C_N, \hspace{1mm} \forall \hspace{1mm} i \in \{1, \ldots, N-1\}$ on the set $\operatorname{SSSYT}(\nu, N)$. The main result of this work is the theorems formulated below. Before stating the main result, we introduce the following notations.

\begin{notation}
Let $\tilde\psi_{\nu}: C_N \rightarrow \Aut(\operatorname{SSSYT}(\nu, N))$ be the map defining the action of the cactus group $C_N$ on the set $\operatorname{SSSYT}(\nu, N)$.  Denote by $\tau_i$ the image of the Bender-Knuth generator $\mathbf{t_i} \in C_N$ under the map $\tilde \psi_{\nu}$.
\end{notation}

\begin{notation}
Denote by $\tau^{bk}_i$ the $i$-th Bender-Knuth involution on the set of semi-standard Young tableaux. 
\end{notation}

Consider the skew semi-standard Young tableau filled with the numbers $i, i+1$. We will call a cell of this tableau \textbf{\textit{free}} if there are no other cells in the same column with it. 

The next three theorems explicitly describe the action of $\tau_i$ on the set $\operatorname{SSSYT}(\nu, N)$.  

\begin{TheoremA} [Theorem \ref{th_7.1}]
\label{main_t}
Let $x = (\nu^{(1)}, \nu^{(2)}, \ldots, \nu^{(i-1)},\nu^{(i)}, \nu^{(i+1)}, \ldots, \nu^{(N)} = \nu) \in \operatorname{SSSYT}(\nu, N)$. Consider $$\tau^{bk}_i(x) = (\nu^{(1)}, \nu^{(2)}, \ldots, \nu^{(i-1)},\tilde\nu^{(i)}, \nu^{(i+1)}, \ldots \nu^{(N)}).$$ If $\tau^{bk}_i(x) \in \operatorname{SSSYT}(\nu, N) \Leftrightarrow \tilde\nu^{(i)} \in \operatorname{SYD}(i, n)$, then \begin{gather}\tau_i(x) = \tau^{bk}_i(x).
\end{gather} 
\end{TheoremA}
So, if Bender-Knuth involution $\tau_i^{bk}$ sends $x \in \operatorname{SSSYT}(\nu, N)$ to a short semi-standard Young tableau, then $\tau_i$ acts on $x$ as $\tau_i^{bk}$.  

Now, consider a tableau $x \in \operatorname{SSSYT}(\nu, N)$, such that $\tau^{bk}_i(x) \notin \operatorname{SSSYT}(\nu, N) \Leftrightarrow \tilde\nu^{(i)} \notin \operatorname{SYD}(i,n)$. The following proposition explains the classification of such tableaux. 

\begin{Proposition}[Proposition \ref{prop_7.1}]
    Let $x = (\nu^{(1)}, \nu^{(2)}, \ldots, \nu^{(i-1)},\nu^{(i)}, \nu^{(i+1)}, \ldots, \nu^{(N)} = \nu) \in \operatorname{SSSYT}(\nu, N)$. If $\tau_i^{bk}(x) \notin \operatorname{SSSYT}(\nu, N)$, then there are $2$ possibilities for the tableau $x$: 
    \\
    \textbf{Possibility 1}. Tableau $x$ satisfies the following conditions: 
\begin{itemize}
    \item There is only one cell  in the last row of the skew Young tableau $\nu^{(i+1)} - \nu^{(i-1)}$. Moreover, this cell is free and has the number $i+1$ in it. 
    \item There is a  free cell with the number $i+1$ in it in the penultimate row of the skew tableau $\nu^{(i+1)} - \nu^{(i-1)}$.
    \item In the first two columns of the tableau $\nu^{(i+1)}$ there are exactly $i+1$ cells.
\end{itemize}
Tableau $x$ that satisfies the conditions above is said to be of  \textbf{type $1$}. 
\\
\textbf{\textit{Possibility 2}}. Tableau $x$ satisfies the following conditions: 
\begin{itemize}
\item In the last row of the skew tableau $\nu^{(i+1)} - \nu^{(i-1)}$ there are at least two  free cells with the number $i+1$ and less than two cells with the number $i$.
\item In the first two columns of the tableau $\nu^{(i+1)}$ there are exactly $i+1$ cells.
\end{itemize}
Tableau $x$ that satisfies the conditions above is said to be of  \textbf{type $2$}. 
\end{Proposition}

\begin{TheoremB}[Theorem \ref{th_7.2}]
    If  $x \in \operatorname{SSSYT}(\nu, N)$ is a tableau of \textbf{type }$1$, then $\tau_i$ operates on $x$ the following way: 
    \begin{itemize}
        \item All rows of the skew tableau $\nu^{(i+1)} - \nu^{(i-1)}$, aside from the last two rows, convert the same way they would under the Bender-Knuth involution.  
        \item The last row of  $\nu^{(i+1)} - \nu^{(i-1)}$ stays the same (it only has one cell with $i+1$ in it).
        \item The number inside the first free cell with the number $i+1$ in the penultimate row is being replaced with $i$ and then the usual Bender-Knuth involution is applied.   
    \end{itemize}
    
\end{TheoremB}

\begin{Example}[Example \ref{ex_7.2}]
Suppose the first two columns of the tableau below (left) contain exactly $i+1$ cells. Then the automorphism ${\tau_i}$ maps it to the tableau on the right.
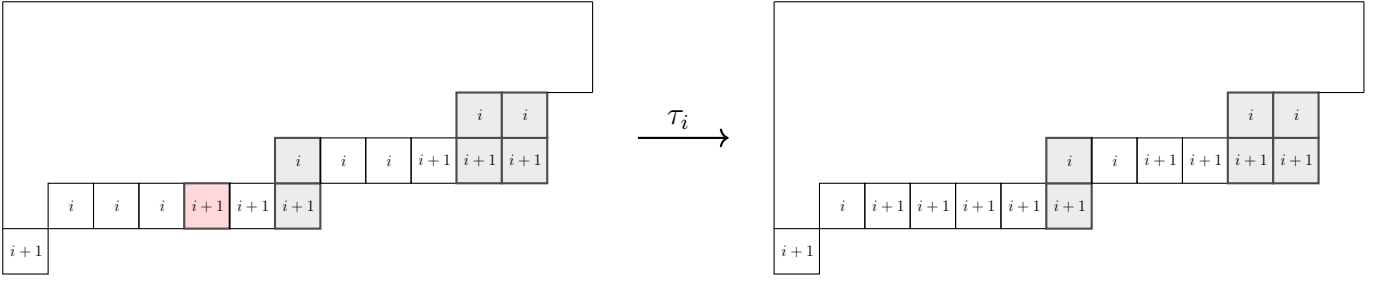
\vspace{2mm}

\begin{center}
\begin{tikzpicture}[scale = 0.6]
\draw (2,0) rectangle (3,1);
\draw (3, 0.5) circle (0pt)  node[scale = 0.5, anchor=east]{$i+1$};

\draw (3,1) rectangle (4,2);
\draw (3.7, 1.5) circle (0pt)  node[scale = 0.5, anchor=east]{$i$};
\draw (4,1) rectangle (5,2);
\draw (4.7, 1.5) circle (0pt)  node[scale = 0.5, anchor=east]{$i$};
\draw (5,1) rectangle (6,2);
\draw (5.7, 1.5) circle (0pt)  node[scale = 0.5, anchor=east]{$i$};
\filldraw[color=black!70, fill=red!15,  thick] (6,1) rectangle (7,2);
\draw (7, 1.5) circle (0pt)  node[scale = 0.5, anchor=east]{$i+1$};
\draw (7,1) rectangle (8,2);
\draw (8, 1.5) circle (0pt)  node[scale = 0.5, anchor=east]{$i+1$};
\filldraw[color=black!70, fill=gray!15,  thick] (8,1) rectangle (9,2);
\draw (9, 1.5) circle (0pt)  node[scale = 0.5, anchor=east]{$i+1$};
\filldraw[color=black!70, fill=gray!15,  thick] (8,2) rectangle (9,3);
\draw (8.7, 2.5) circle (0pt)  node[scale = 0.5, anchor=east]{$i$};
\draw (9,2) rectangle (10,3);
\draw (9.7, 2.5) circle (0pt)  node[scale = 0.5, anchor=east]{$i$};
\draw (10,2) rectangle (11,3);
\draw (10.7, 2.5) circle (0pt)  node[scale = 0.5, anchor=east]{$i$};
\draw (11,2) rectangle (12,3);
\draw (12, 2.5) circle (0pt)  node[scale = 0.5, anchor=east]{$i+1$};
\filldraw[color=black!70, fill=gray!15,  thick] (12,2) rectangle (13,3);
\draw (13, 2.5) circle (0pt)  node[scale = 0.5, anchor=east]{$i+1$};
\filldraw[color=black!70, fill=gray!15,  thick] (13,2) rectangle (14,3);
\draw (14, 2.5) circle (0pt)  node[scale = 0.5, anchor=east]{$i+1$};
\filldraw[color=black!70, fill=gray!15,  thick] (12,3) rectangle (13,4);
\draw (12.7, 3.5) circle (0pt)  node[scale = 0.5, anchor=east]{$i$};
\filldraw[color=black!70, fill=gray!15,  thick] (13,3) rectangle (14,4);
\draw (13.7, 3.5) circle (0pt)  node[scale = 0.5, anchor=east]{$i$};
\draw (2,1) -- (2,6);
\draw (2,6) -- (15,6);
\draw (15,6) -- (15,4);
\draw (15,4) -- (14,4);

\draw[thick,->] (16,3) -- (18,3);
\filldraw[black] (16.40,3.40) circle (0pt) node[anchor=west]{$\tau_i$};


\draw (19,0) rectangle (20,1);
\draw (20, 0.5) circle (0pt)  node[scale = 0.5, anchor=east]{$i+1$};

\draw (20,1) rectangle (21,2);
\draw (20.7, 1.5) circle (0pt)  node[scale = 0.5, anchor=east]{$i$};
\draw (21,1) rectangle (22,2);
\draw (22, 1.5) circle (0pt)  node[scale = 0.5, anchor=east]{$i+1$};
\draw (22,1) rectangle (23,2);
\draw (23, 1.5) circle (0pt)  node[scale = 0.5, anchor=east]{$i+1$};
\draw (23,1) rectangle (24,2);
\draw (24, 1.5) circle (0pt)  node[scale = 0.5, anchor=east]{$i+1$};
\draw (24,1) rectangle (25,2);
\draw (25, 1.5) circle (0pt)  node[scale = 0.5, anchor=east]{$i+1$};
\filldraw[color=black!70, fill=gray!15,  thick] (25,1) rectangle (26,2);
\draw (26, 1.5) circle (0pt)  node[scale = 0.5, anchor=east]{$i+1$};
\filldraw[color=black!70, fill=gray!15,  thick] (25,2) rectangle (26,3);
\draw (25.7, 2.5) circle (0pt)  node[scale = 0.5, anchor=east]{$i$};
\draw (26,2) rectangle (27,3);
\draw (26.7, 2.5) circle (0pt)  node[scale = 0.5, anchor=east]{$i$};
\draw (27,2) rectangle (28,3);
\draw (28, 2.5) circle (0pt)  node[scale = 0.5, anchor=east]{$i+1$};
\draw (28,2) rectangle (29,3);
\draw (29, 2.5) circle (0pt)  node[scale = 0.5, anchor=east]{$i+1$};
\filldraw[color=black!70, fill=gray!15,  thick] (29,2) rectangle (30,3);
\draw (30, 2.5) circle (0pt)  node[scale = 0.5, anchor=east]{$i+1$};
\filldraw[color=black!70, fill=gray!15,  thick] (30,2) rectangle (31,3);
\draw (31, 2.5) circle (0pt)  node[scale = 0.5, anchor=east]{$i+1$};
\filldraw[color=black!70, fill=gray!15,  thick] (29,3) rectangle (30,4);
\draw (29.7, 3.5) circle (0pt)  node[scale = 0.5, anchor=east]{$i$};
\filldraw[color=black!70, fill=gray!15,  thick](30,3) rectangle (31,4);
\draw (30.7, 3.5) circle (0pt)  node[scale = 0.5, anchor=east]{$i$};
\draw (19,1) -- (19,6);
\draw (19,6) -- (32,6);
\draw (32,6) -- (32,4);
\draw (32,4) -- (31,4);
\end{tikzpicture}
\end{center}

Non-free cells are highlighted in gray, and the first free cell in the penultimate row of the left skew Young tableau with the number $i+1$ inside is highlighted in red. This example illustrates the action of ${\tau_i}$ on \textbf{type }$1$  tableaux. 

\end{Example}

\begin{TheoremC}[Theorem \ref{th_7.3}]
   If  $x \in \operatorname{SSSYT}(\nu, N)$ is a tableau of \textbf{type} $2$, then $\tau_i$ operates on $x$ the following way: 
\begin{itemize}
    \item All rows of the skew tableau $\nu^{(i+1)} - \nu^{(i-1)}$, aside from the last row, convert the same way they would under the Bender-Knuth involution.
    \item The first cell in the last row of the skew tableau $\nu^{(i+1)} - \nu^{(i-1)}$ (it is always free) changes its value from $i$ to $i+1$ and vice versa if the number of free cells in the last row of $\nu^{(i+1)} - \nu^{(i-1)}$ is odd.
\end{itemize}
\end{TheoremC}

\begin{Example}[Example \ref{ex_7.3}]
Suppose the first two columns of the tableau below (left) contain exactly $i+1$ cells. Then the automorphism ${\tau_i}$ maps it to the tableau on the right.
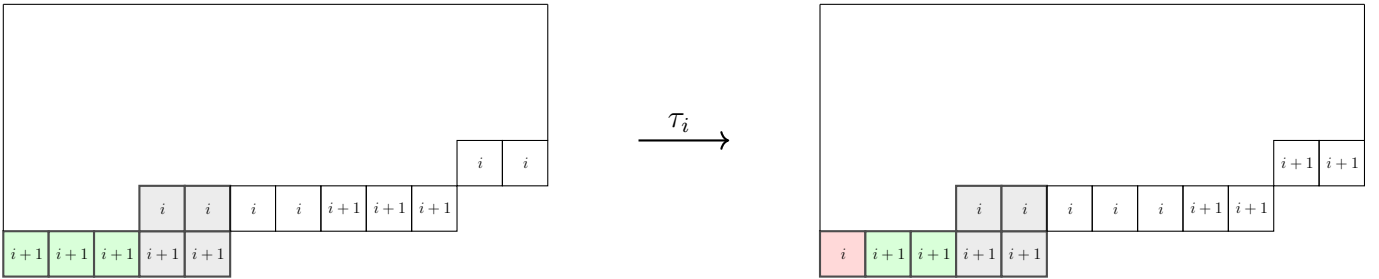
\vspace{2mm}  
\begin{center}
\begin{tikzpicture}[scale = 0.6]
\filldraw[color=black!70, fill=green!15,  thick] (2,0) rectangle (3,1);
\draw (3, 0.5) circle (0pt)  node[scale = 0.5, anchor=east]{$i+1$};
\filldraw[color=black!70, fill=green!15,  thick] (3,0) rectangle (4,1);
\draw (4, 0.5) circle (0pt)  node[scale = 0.5, anchor=east]{$i+1$};
\filldraw[color=black!70, fill=green!15,  thick] (4,0) rectangle (5,1);
\draw (5, 0.5) circle (0pt)  node[scale = 0.5, anchor=east]{$i+1$};
\filldraw[color=black!70, fill=gray!15,  thick] (5,0) rectangle (6,1);
\draw (6, 0.5) circle (0pt)  node[scale = 0.5, anchor=east]{$i+1$};
\filldraw[color=black!70, fill=gray!15,  thick](6,0) rectangle (7,1);
\draw (7, 0.5) circle (0pt)  node[scale = 0.5, anchor=east]{$i+1$};
\filldraw[color=black!70, fill=gray!15,  thick] (5,1) rectangle (6,2);
\draw (5.7, 1.5) circle (0pt)  node[scale = 0.5, anchor=east]{$i$};
\filldraw[color=black!70, fill=gray!15,  thick] (6,1) rectangle (7,2);
\draw (6.7, 1.5) circle (0pt)  node[scale = 0.5, anchor=east]{$i$};
\draw (7,1) rectangle (8,2);
\draw (7.7, 1.5) circle (0pt)  node[scale = 0.5, anchor=east]{$i$};
\draw (8,1) rectangle (9,2);
\draw (8.7, 1.5) circle (0pt)  node[scale = 0.5, anchor=east]{$i$};
\draw (9,1) rectangle (10,2);
\draw (10, 1.5) circle (0pt)  node[scale = 0.5, anchor=east]{$i+1$};
\draw (10,1) rectangle (11,2);
\draw (11, 1.5) circle (0pt)  node[scale = 0.5, anchor=east]{$i+1$};
\draw (11,1) rectangle (12,2);
\draw (12, 1.5) circle (0pt)  node[scale = 0.5, anchor=east]{$i+1$};

\draw (12,2) rectangle (13,3);
\draw (12.7, 2.5) circle (0pt)  node[scale = 0.5, anchor=east]{$i$};
\draw (13,2) rectangle (14,3);
\draw (13.7, 2.5) circle (0pt)  node[scale = 0.5, anchor=east]{$i$};

\draw (2,1) -- (2,6);
\draw (2,6) -- (14,6);
\draw (14,6) -- (14,3);

\draw[thick,->] (16,3) -- (18,3);
\filldraw[black] (16.40,3.40) circle (0pt) node[anchor=west]{$\tau_i$};

\filldraw[color=black!70, fill=red!15,  thick] (20,0) rectangle (21,1);
\draw (20.7, 0.5) circle (0pt)  node[scale = 0.5, anchor=east]{$i$};
\filldraw[color=black!70, fill=green!15,  thick] (21,0) rectangle (22,1);
\draw (22, 0.5) circle (0pt)  node[scale = 0.5, anchor=east]{$i+1$};
\filldraw[color=black!70, fill=green!15,  thick] (22,0) rectangle (23,1);
\draw (23, 0.5) circle (0pt)  node[scale = 0.5, anchor=east]{$i+1$};
\filldraw[color=black!70, fill=gray!15,  thick] (23,0) rectangle (24,1);
\draw (24, 0.5) circle (0pt)  node[scale = 0.5, anchor=east]{$i+1$};
\filldraw[color=black!70, fill=gray!15,  thick](24,0) rectangle (25,1);
\draw (25, 0.5) circle (0pt)  node[scale = 0.5, anchor=east]{$i+1$};
\filldraw[color=black!70, fill=gray!15,  thick] (23,1) rectangle (24,2);
\draw (23.7, 1.5) circle (0pt)  node[scale = 0.5, anchor=east]{$i$};
\filldraw[color=black!70, fill=gray!15,  thick] (24,1) rectangle (25,2);
\draw (24.7, 1.5) circle (0pt)  node[scale = 0.5, anchor=east]{$i$};
\draw (25,1) rectangle (26,2);
\draw (25.7, 1.5) circle (0pt)  node[scale = 0.5, anchor=east]{$i$};
\draw (26,1) rectangle (27,2);
\draw (26.7, 1.5) circle (0pt)  node[scale = 0.5, anchor=east]{$i$};
\draw (27,1) rectangle (28,2);
\draw (27.7, 1.5) circle (0pt)  node[scale = 0.5, anchor=east]{$i$};
\draw (28,1) rectangle (29,2);
\draw (29, 1.5) circle (0pt)  node[scale = 0.5, anchor=east]{$i+1$};
\draw (29,1) rectangle (30,2);
\draw (30, 1.5) circle (0pt)  node[scale = 0.5, anchor=east]{$i+1$};

\draw (30,2) rectangle (31,3);
\draw (31, 2.5) circle (0pt)  node[scale = 0.5, anchor=east]{$i+1$};
\draw (31,2) rectangle (32,3);
\draw (32, 2.5) circle (0pt)  node[scale = 0.5, anchor=east]{$i+1$};

\draw (20,1) -- (20,6);
\draw (20,6) -- (32,6);
\draw (32,6) -- (32,3);

\end{tikzpicture}
\end{center}

Non-free cells are highlighted in gray, free cells of the last row are highlighted in green. The first cell of the last row of the right tableau, where the number changed, is highlighted in red. This example illustrates the action of ${\tau_i}$ on \textbf{type 2} tableaux. 
\end{Example}

\subsection*{This paper is organized as follows}

The first chapter \ref{crystals} is dedicated to the notion of a crystal. In \ref{general_info},  following the articles of Henriques and Kamnitzer \cite{Henriques_Kamnitzer} and Joseph \cite{Joseph}, we give all the necessary definitions regarding crystals. In \ref{crystal_bases}, following Kashiwara's review article \cite{Kashiwara_1995} we construct the closed family of normal crystals using integrable $U_q(\mathfrak{g})$-modules from the $\mathcal{O}_{int}(\mathfrak{g})$ category.  In\
\ref{spinor_crystal} we study the structure of the spinor crystal. The main results of  subsection \ref{spinor_crystal} are Proposition \ref{prop_3.2} and  Corollary \ref{cor_3.1}. In \ref{b_infty&Kas_involution} we explore the action of  Kashiwara's involution on the crystal $B_\infty$. The main result of the subsection \ref{b_infty&Kas_involution} is Corollary \ref{cor_10.1}. Corollary \ref{cor_10.1} plays a key role in the proof of the main result of the subsection \ref{technical_results}, namely  Theorem \ref{th_4.1}.

The second chapter \label{cell_tables_&_sssyt} follows my previous study \cite{Svyatnyy} and focuses on combinatorial interpretations of the sets $\Delta^N$ (see \ref{delta}) and $T_\lambda^N $ (see \ref{tlambdaN}).  Key propositions of this chapter are \ref{prop_5.2}, \ref{prop_5.4}, \ref{prop_5.5}. 

The third and final chapter \ref{final_chapter} is dedicated to cactus group and its action on the set $T_\lambda^N$ and its combinatorial interpretations discussed in \ref{prop_5.4} and \ref{prop_5.5}. In subsection \ref{def_of_action_of_cactus_group} following \cite{Henriques_Kamnitzer} we define the action of cactus group $C_N$ on the $N$-th tensor power of a crystal and by doing so also define the action of $C_N$ on the set $T_\lambda^N$. We also present an alternative description of the cactus group given by  Michael Chmutov, Max Glick and Pavlo Pylyavskii in \cite{Chmutov}, see Theorem \ref{th_6.1}.  
In subsection \ref{final_final} we describe explicitly the action of Bender-Knuth generators of $C_N$ (these are the generators used in the alternative description of cactus group by Chmutov-Glick-Pylyavskii) on the set $\operatorname{SSSYT}(\nu, N)$ , which is a combinatorial interpretation of the set $T_\lambda^N $ given by Proposition \ref{prop_5.5}, where $\nu = \mathcal{F}_N(\lambda)$, see  Theorem \ref{th_5.1}.  Main results of this chapter and of the article as a whole are Theorems \ref{th_7.1}, \ref{th_7.2},  \ref{th_7.3}. Examples \ref{ex_7.1}, \ref{ex_7.2}, \ref{ex_7.3} visualize the statements of these theorems. 

\subsection*{Acknowledgments}

I am grateful to my scientific advisor Leonid Rybnikov for drawing my attention to the subject and for the extremely useful discussions and references. Also, I want to thank him for pointing out the relevance of Kashiwara's involution in  Theorem \ref{th_4.1}. My initial proof of this theorem was far less elegant. 

I am happy to thank Grigory Olshansky for pointing out the relevance of Howe duality (see Proposition \ref{prop_5.3}) to my research, which helped me to construct a great combinatorial interpretation of the set $T_\lambda^N$, namely the set of short semi-standard Young tableaux. I am also thankful for all the references and the discussions we had. 

\section{Crystals}
\label{crystals}

\subsection{General information on crystals}
\label{general_info}
This subsection is dedicated to general knowledge on crystals. We start with the definition of a crystal. Informally,  a crystal is a combinatorial model for a representation of a Lie algebra $\mathfrak{g}$. It can be thought of as a graph with the vertices labeled by the weight lattice of $\mathfrak{g}$ and directed edges labeled by the simple roots of $\mathfrak{g}$. This subsection mostly follows the article by Andre Henriques and Joel Kamnitzer \cite{Henriques_Kamnitzer}, except that they only consider normal crystals (they call them $\mathfrak{g}$-crystals). 

Let $\mathfrak{g}$ be a complex semisimple Lie algebra. We denote by $P$ its weight lattice, by $P_+$ the set of dominant weights, $I$ denote the set of vertices of the corresponding Dynkin diagram, $\{\alpha_i\}_{i \in I}$ denote its simple roots and $\{\alpha_i^{\vee}\}_{i \in I}$ denote its simple co-roots. 

We follow conventions in Joseph \cite{Joseph} in defining crystals.

\begin{definition}
\label{def_1.1}
A \textbf{crystal} is a set $\mathcal{B}$ with the maps 
    \begin{gather*}
        \operatorname{wt}: \mathcal{B} \rightarrow P, \\
        \varepsilon_{i}, \varphi_{i}: \mathcal{B} \rightarrow \mathbb{Z}  \sqcup \{ - \infty\} \\
        \end{gather*}
and operators,
    \begin{gather*}
        \tilde e_i, \tilde f_i : \mathcal{B} \rightarrow \mathcal{B} \sqcup \{0\}, \hspace{1mm}  \\
        \tilde e_i 0 = 0, \hspace{1mm} \tilde f_i 0= 0,
    \end{gather*}
for all $i \in I$. Here $0$ is the ideal element that does not belong to $\Cr{B}$. They are subject to the following axioms: 
\begin{enumerate}
    \item for all $b \in \mathcal{B}$ we have $\varphi_i(b) - \varepsilon_i(b) = \langle\text{wt}(b), \alpha_i^{\vee}  \rangle$.
    \item if $b \in \Cr{B}$ and $\tilde e_ib  \in \Cr{B}$, then  $\operatorname{wt}(\tilde e_ib) = \operatorname{wt}(b) + \alpha_i, \hspace{1mm} \varepsilon_i(\tilde e_ib) = \varepsilon_i(b) - 1$ and $\varphi_i(\tilde e_i b) = \varphi_i(b) + 1$; 
    \item if $b \in \Cr{B}$ and $\tilde f_ib \in \Cr{B}$, then $\operatorname{wt}(\tilde f_ib) = \operatorname{wt}(b) - \alpha_i, \hspace{1mm} \varepsilon_i(\tilde f_ib) = \varepsilon_i(b) + 1$ and $\varphi_i(\tilde f_i b) = \varphi_i(b) - 1$; 
    \item for all $b, b^{'} \in \mathcal{B}$ we have $b^{'} = \tilde e_i \cdot b \Leftrightarrow b = \tilde f_i \cdot b^{'}$.
    \item for $b \in \Cr{B}$,  if $\varphi_i(b) = - \infty$, then $\tilde e_ib = \tilde f_i b = 0$.
\end{enumerate}
\end{definition}

\begin{definition}
For two crystals $\Cr{A}$ and  $\Cr{B}$ a \textbf{morphism} $\psi$ from $\Cr{A}$ to $\Cr{B}$ is a map $\Cr{A} \sqcup \{0\} \rightarrow \Cr{B} \sqcup \{0\}$ that satisfies the following conditions: 
    
\begin{enumerate}
    \item $\psi(0) = 0$,
    \item If $b \in \Cr{A}$ and $\psi(b) \in \Cr{B}$, then \begin{itemize}
        \item $\operatorname{wt}(\psi(b)) = \operatorname{wt}(b)$,
        \item $\varepsilon_i(\psi(b)) = \varepsilon_i(b)$
        \item $\varphi_i(\psi(b)) = \varphi_i(b)$
    \end{itemize}
    \item For $b \in \Cr{A} $, we have $\psi(\tilde e_ib) = \tilde e_i \psi(b)$ provided $\tilde e_ib \in \Cr{A}$. 
    \item For $b \in \Cr{A},$ we have $\psi(\tilde f_ib) = \tilde f_i \psi(b)$ provided $\tilde f_ib \in \Cr{A}$. 
\end{enumerate}

A morphism is called an \textbf{embedding} if the corresponding map is injective. 
\end{definition}

\begin{definition}
 A morphism $\psi: \Cr{A} \rightarrow \Cr{B}$  is called strict if it commutes with $\tilde e_i$, $\tilde f_i$ , for all $i \in I$.
\end{definition}

\begin{definition}
A crystal $\Cr{B}$ is called \textbf{upper normal} (resp. \textbf{lower normal})  if for any $b \in \Cr{B}$ and $i \in I$, $\varepsilon_i(b) \in \mathbb{Z}$ (resp. $\varphi_i(b) \in \mathbb{Z}$) and \begin{gather*} \varepsilon_i(b) = \text{max} \{n: \tilde e_i^{n} \cdot b \neq 0 \} \\
(\text{resp.} \hspace{1mm} \varphi_i(b) = \text{max} \{n : \tilde{ f}_i^{n}  \cdot b \neq 0 \}) 
.\end{gather*}
A crystal $\Cr{B}$ is called \textbf{normal} if it is both \textbf{upper} and \textbf{lower normal}. 
\end{definition}
\begin{remark}
In the later papers of Kashiwara (e.g.\cite{Kashiwara_1995}) the term "\textbf{normal crystal}" is used for crystals that arise from integrable $U_q(\mathfrak{g})$-modules. For crystals such that  \begin{gather*} \varepsilon_i(b) = \text{max} \{n: \tilde e_i^{n} \cdot b \neq 0 \} \\
(\text{resp.} \hspace{1mm} \varphi_i(b) = \text{max} \{n : \tilde{ f}_i^{n}  \cdot b \neq 0 \}) 
.\end{gather*}
 the term "\textbf{semi-normal crystal}"  is used instead. This difference in terminology isn't important for us, since we are not going to consider semi-normal crystals, which are not normal. 
\end{remark}
The following lemma is obvious.

\begin{lemma}
If $\Cr{A}$ and $\Cr{B}$ are normal, then any morphism $\psi: \Cr{A} \rightarrow \Cr{B}$ is strict. 
\end{lemma}

In this paper we mostly work with crystals arising from the crystal bases of integrable $U_q(\mathfrak{g})$-modules. Those crystals are normal, so all morphisms between them are strict. 

\begin{definition}
\label{def_1.3}
For two crystals $\Cr{A}$ and $\Cr{B}$ we define the \textbf{direct sum} $\Cr{A} \oplus \Cr{B}$ whose underlying set is $\Cr{A} \sqcup \Cr{B}$ with the obvious actions. 
\end{definition}

Next we define the tensor product of crystals. 

\begin{definition}
\label{def_1.2}
For two crystals $\mathcal{A}$ and $\mathcal{B}$ we define their \textbf{tensor product} $\mathcal{A} \otimes \mathcal{B}$ as follows: 
$$\Cr{A} \tens \Cr{B} = \{a \tens b; \hspace{1mm} a \in \Cr{A}, b \in \Cr{B}\},$$
$$\text{wt}(a \tens b) = \text{wt}(a) + \text{wt}(b).$$
For $i \in I$ set  $$\operatorname{wt}_i(b) = \langle \operatorname{wt}(b), \alpha_i^{\vee}\rangle.$$
Then
\begin{gather*}
    \varepsilon_i(a\tens b) = \max (\varepsilon_i(a),  \varepsilon_i(b) - \operatorname{wt}_i(a)), \\
    \varphi_i(a \tens b) = \max (\varphi_i(b), \varphi_i(a) + \operatorname{wt}_i(b)).
\end{gather*}

The actions of $\tilde{e}_i$, $\tilde{f}_i$ are defined by : 

\begin{center}
$\tilde e_i (a \tens b) = \begin{cases}
a \tens \tilde e_i  b, \hspace{1mm} \text{if} \hspace{1mm} \varepsilon_i(b) > \varphi_i(a)
\\
\tilde e_i a \tens b, \hspace{1mm} \text{otherwise}
\end{cases}
$

$\tilde f_i  (a \tens b) = \begin{cases}
a\tens\tilde  f_i b, \hspace{1mm} \text{if} \hspace{1mm} \varepsilon_i(b) \geqslant \varphi_i(a)
\\
\tilde f_i a \tens b, \hspace{1mm} \text{otherwise}
\end{cases}
$
\end{center}

Here we assume $a \tens 0 = 0 \tens b = 0$.
\end{definition}

One can check directly that the tensor product of crystals is indeed a crystal. Moreover, it can be checked that the tensor product of normal crystals is also normal. Both computations can be found in \cite{Joseph} in subsections $5.2.4$ and $5.2.6$  (pages $138-140$). 

Tensor product of crystals is associative: i.e. if $\mathcal{B}_1, \mathcal{B}_2, \mathcal{B}_3$ are crystals, then the map
$$\alpha_{\mathcal{B}_1, \mathcal{B}_2, \mathcal{B}_3}: (\mathcal{B}_1 \otimes \mathcal{B}_2) \otimes \mathcal{B}_3 \rightarrow \mathcal{B}_1 \otimes (\mathcal{B}_2 \otimes \mathcal{B}_3) $$defined by $$(a \tens b)\tens c \mapsto a \tens (b \tens c)$$
is an isomorphism of crystals (see Proposition 1.3.1 in \cite{Kashiwara_1993}). Hence, the category of crystals equipped with a bifunctor $\tens$ is a tensor category and normal crystals form a tensor subcategory inside it.

\begin{definition}
\label{def_1.1.1}
    The \textbf{crystal graph} of a crystal $\mathcal{B}$ is a colored and oriented graph with the vertices in $\mathcal{B}$, with the arrows
    \begin{gather}
        u \overset{i}{\longrightarrow} v \hspace{2mm} \text{iff} \hspace{2mm} u, v \in \Cr{B}, \hspace{1mm}  v = \tilde f_i u
    \end{gather}
\end{definition}

\begin{definition}
\label{def_1.5}
A crystal $\mathcal{B}$ is called \textbf{connected} if the crystal graph of $\mathcal{B}$ is connected. 

A  crystal $\mathcal{B}$ is called a \textbf{highest weight crystal of highest weight} $\lambda \in P_+$, if there exists an element $b_\lambda$ (called a \textbf{highest weight element}) such that $\text{wt}(b_\lambda) = \lambda$,  $\tilde e_i  b_\lambda = 0$ for all $i \in I$ and the whole $\mathcal{B}$ is generated by $\tilde f_i$ acting on $b_\lambda$. 
\end{definition}

Clearly, every highest weight crystal is connected. The converse is not true. Moreover, there are non-isomorphic highest weight crystals of the same highest weight. 


Let $\mathbb{B} = \{\mathcal{B}_{\lambda} : \lambda \in P_+\}$ be a family of crystals such that $\mathcal{B}_\lambda$ is a highest weight crystal of highest weight $\lambda$ . We call the family $(\mathbb{B}, \iota)$ \textbf{closed} if $\iota_{\lambda, \nu}: \mathcal{B}_{\lambda + \nu} \rightarrow \mathcal{B}_{\lambda} \otimes \mathcal{B}_{\nu}$, defined by $b_{\lambda+ \nu} \mapsto b_\lambda \tens b_\nu$ is an inclusion of crystals. 

\begin{theorem}[Joseph, \cite{Joseph}] \label{th_1.1}
There exists a unique closed family of normal crystals $(\mathbb{B}, \iota)$. 
\end{theorem}

There are many ways to construct the closed family of normal crystals $(\mathbb{B}, \iota)$. One of the ways to do it is by taking crystal bases of the irreducible highest weight $U_q(\mathfrak{g})$-modules $V_\lambda$ with highest weights $\lambda \in P_+$, see paragraph \ref{crystal_bases} for details.

\begin{definition}
\label{def_1.6}
The category \underline{\textbf{$\mathfrak{g}$-$\text{Crystals}$}} is the category whose objects are crystals $\mathcal{B}$, such that every connected component of $\mathcal{B}$ is isomorphic to some $\mathcal{B}_\lambda$ from the closed family $\mathbb{B}$. Morphisms in this category are the usual morphisms of crystals. 
\end{definition}

In other words the category \underline{\textbf{$\mathfrak{g}$-$\text{Crystals}$}} consists of crystals isomorphic to the direct sum of highest weight crystals from the closed family $\mathbb{B}$, so any object of this category is automatically a normal crystal. By nature the category \underline{\textbf{$\mathfrak{g}$-$\text{Crystals}$}} is semisimple with simple objects being highest weight crystals $\Cr{B}_\lambda$ from the closed family.  Moreover,  the category \underline{\textbf{$\mathfrak{g}$-$\text{Crystals}$}} is closed under the tensor product operation defined in \ref{def_1.2}, which makes it a tensor subcategory of the category of normal crystals.  The latter will be explained in paragraph \ref{crystal_bases} after the explicit construction of the closed family $(\mathbb{B}, \iota)$. 

For the rest of this paper, unless specifically said otherwise, crystal means an object in \underline{\textbf{$\mathfrak{g}$-$\text{Crystals}$}} category.  We also reserve the notation $\Cr{B}_\lambda$ for the highest weight crystal of highest weight $\lambda \in P_+$ from the unique closed family of normal crystals $
(\mathbb{B}, \iota)$ . For the category \underline{\textbf{$\mathfrak{g}$-$\text{Crystals}$}} we have the following version of Schur's Lemma. 

\begin{proposition}[Schur's lemma]
\label{prop_1.2}
$Hom(\mathcal{B}_\lambda, \mathcal{B}_\mu)$ contains just the identity and zero maps in case $\lambda = \mu$ and only zero map otherwise. 
\end{proposition}
\begin{proof}
Since a morphism of crystals commutes with all the structure maps, it sends the highest weight element of $\mathcal{B}_\lambda$ to either the highest weight element of $\mathcal{B}_\mu$ or to $0$. If the highest weight element of $\Cr{B}_\lambda$ is sent to $0$, then the whole $\Cr{B}_\lambda$ is sent to $0$ and it is a zero map.  Otherwise $\lambda = \mu$. Since any element of $\mathcal{B}_\lambda$ can be obtained by acting on the highest weight element by multiple $f_i$'s, the morphism is uniquely determined and is equal to the identity map. 
\end{proof}

Note that the tensor product of crystals is not symmetric, i.e. the map 
$$\text{flip} : \mathcal{B}_1 \otimes \mathcal{B}_2 \rightarrow \mathcal{B}_2 \otimes \mathcal{B}_1$$ defined by $b_1 \tens b_2 \mapsto b_2 \tens b_1$ is not a morphism of crystals. 

In the work of Henriques and Kamnitzer \cite{Henriques_Kamnitzer}  natural isomorphisms $$\sigma_{\mathcal{B}_1, \mathcal{B}_2} : \mathcal{B}_1 \otimes \mathcal{B}_2 \rightarrow \mathcal{B}_2 \otimes \mathcal{B}_1$$ were defined for any two $\mathcal{B}_1, \mathcal{B}_2 \in \operatorname{Ob}(\underline{\textbf{$\mathfrak{g}$-$\text{Crystals}$}})$. This family of maps was called \textbf{commutor}.

\begin{proposition}
    \label{prop_1.3}
    $\sigma_{\mathcal{B}_1, \mathcal{B}_2} \circ \sigma_{\mathcal{B}_2, \mathcal{B}_1} = 1$.
\end{proposition}

\begin{theorem}
\label{th_1.2} 
The following diagram commutes in \underline{\textbf{$\mathfrak{g}$-$\text{Crystals}$}}: \\
    
\centering
\begin{tikzcd}[column sep=3em]
\mathcal{B}_1 \otimes \mathcal{B}_2 \otimes \mathcal{B}_3 \arrow[r, "1_{\mathcal{B}_1} \otimes \sigma_{\mathcal{B}_2, \mathcal{B}_3} "] \arrow[d,swap,"\sigma_{\mathcal{B}_1, \mathcal{B}_2} \otimes 1_{\mathcal{B}_3} " ]  &
\mathcal{B}_1 \otimes \mathcal{B}_3 \otimes \mathcal{B}_2 \arrow[d,"\sigma_{\mathcal{B}_1 \otimes \mathcal{B}_3 , \mathcal{B}_2}"] \\
\mathcal{B}_2 \otimes \mathcal{B}_1 \otimes \mathcal{B}_3 \arrow[r,"\sigma_{\mathcal{B}_2 \otimes \mathcal{B}_1 , \mathcal{B}_3}" ]  &  \mathcal{B}_2 \otimes \mathcal{B}_1 \otimes \mathcal{B}_3 
\end{tikzcd}
\end{theorem}
The proofs of \ref{prop_1.3} and \ref{th_1.2} can be found in \cite{Henriques_Kamnitzer}.


\subsection{ Crystal bases. Explicit construction of the closed family.}
\label{crystal_bases}

In this chapter we dive into the explicit construction of the closed family $(\mathbb{B}, \iota)$. This chapter mostly follows Kashiwara's review article \cite{Kashiwara_1995}. 

We shall define $U_q(\mathfrak{g})$ for a complex semisimple Lie algebra $\mathfrak{g}$. Let $\mathfrak{h}$ be a Cartan subalgebra of $\mathfrak{g}$.  We denote  simple roots by $\{\alpha_i \in \mathfrak{h}^{*}\}_{i \in I}$ and simple co-roots by $\{\alpha^{\vee}_i \in \mathfrak{h}\}_{i \in I}$, where $I$ is a finite index set representing the vertices of a Dynkin diagram. Let us take an inner product $(\cdot, \cdot)$ on $\mathfrak{h}^{*}_{\mathbb{R}} := \operatorname{span}_{\mathbb{R}} \{\alpha_i, i \in I\} $, such that 
\begin{gather}(\alpha_i, \alpha_i) \in 2 \mathbb{Z}_{>0}, \\ \langle \alpha_i^{\vee}, \lambda\rangle = \frac{2(\alpha_i, \lambda)}{(\alpha_i, \alpha_i)}, \\
(\alpha_i, \alpha_j) \leqslant 0, \hspace{1mm} \text{for} \hspace{1mm} i \neq j.
\end{gather}
for any $i, j \in I$ and $\lambda \in \mathfrak{h}_{\mathbb{R}}^{*}$.  Here, $\langle \cdot, \cdot \rangle$ is the canonical pairing between $\mathfrak{h}^{*}$ and $\mathfrak{h}$.  Let $\{\Lambda_i\}_{i \in I}$ be the dual base of $\{\alpha^{\vee}_i\}$ with respect to $\langle \cdot, \cdot \rangle$ and set $P = \bigoplus\mathbb{Z}\Lambda_i$ to be the weight lattice of $\mathfrak{g}$ and $P^{*} = \bigoplus \mathbb{Z}\alpha_i^{\vee}$ to be the dual lattice of $P$.  The set of dominant integral weights is denoted by $P_+ = \{ \lambda \in P; \langle \lambda, \alpha_i^{\vee}\rangle \geqslant 0, \forall i \in I\}$.

\begin{definition}
    The quantized universal enveloping algebra $U_q(\mathfrak{g})$ is the algebra over $\mathbb{C}(q)$ generated by the symbols $e_i, f_i, i \in I$, and $q^{h}, h \in P^{*}$ with the following defining relations:
    \begin{enumerate}
        \item $q^h = 1$ for $h = 0$.
        \item $q^{h_1}q^{h_2} = q^{h_1 + h_2}$ for $h_1, h_2 \in P^{*}$.
        \item For $i \in I$ and $h \in P^{*}$, \begin{gather*} q^{h} e_i q^{-h} = q^{\langle h, \alpha_i\rangle} e_i \\ q^{h}f_iq^{-h} = q^{-\langle h, \alpha_i\rangle}f_i \end{gather*}
        \item  $[e_i, f_j] = \delta_{i,j} \cdot \frac{t_i - t_i^{-1}}{q_i - q_i^{-1}}, \hspace{2mm} \text{where} \hspace{2mm} q_i = q^{\frac{(\alpha_i, \alpha_i)}{2}} \hspace{2mm} \text{and} \hspace{2mm} t_i = q^{\frac{(\alpha_i, \alpha_i)}{2}\alpha_i^{\vee}}.$
        \item (Serre relations) For $i \neq j$ $$  \sum_{k=0}^{b} (-1)^{k} e_i^{(k)}e_je_i^{(b-k)} = \sum_{k=0}^{b} (-1)^{k} f_i^{(k)}f_jf_i^{(b-k)} = 0.$$ Here $b =  1 - \langle \alpha_i^{\vee}, \alpha_j\rangle$ and  \begin{gather} e_i^{(k)} = \frac{e_i^{k}}{[k]_{i}!}, \hspace{2mm} f_i^{(k)} = \frac{f_i^{k}}{[k]_{i}!}; \\ [n]_{i} = \frac{q_i^{n} - q_i^{-n}}{q_i - q_i^{-1}}, \hspace{2mm}  [k]_{i}! = \prod_{n=1}^{k} [n]_{i} \end{gather}

    \end{enumerate}
\end{definition}
Let us denote by $U_q(\mathfrak{g})_i$ the subalgebra of $U_q(\mathfrak{g})$ generated by $e_i, f_i, t_i, t_i^{-1}$. 
One can easily check, that in $U_q(\mathfrak{g})$ we have \begin{gather}
    t_ie_jt_i^{-1} = q^{(\alpha_i, \alpha_j)}e_j  = q_i^{\langle \alpha_i^{\vee}, \alpha_j\rangle}e_j, \\
    t_if_jt_i^{-1} = q^{-(\alpha_i, \alpha_j)} f_j = q_i^{-\langle \alpha_i^{\vee}, \alpha_j\rangle}f_j,
\end{gather}
Then $U_q(\mathfrak{g})_i$ is isomorphic to $U_{q_i}(\mathfrak{sl}_2)$. Hence we can say that $U_q(\mathfrak{g})$ is made up of several quantized $\mathfrak{sl}_2$, 

\begin{definition}
    A left $U_q(\mathfrak{g})$-module $M$ is called integrable if it satisfies 
    \begin{enumerate}
        \item $M = \bigoplus_{\lambda \in P} M[\lambda]$, where $M[\lambda] = \{v \in M; q^{h}v = q^{\langle h, \lambda \rangle} v, \hspace{1mm} \forall h \in P^{*}\}$.
        \item For any $i \in I$, $M$ is a union of finite dimensional $U_q(\mathfrak{g})_i$-submodules. 
    \end{enumerate}
\end{definition}
Let $M$ be an integrable $U_q(\mathfrak{g})$-module. Then by theory of integrable representations of $U_q(\mathfrak{sl}_2)$ we have:
\begin{gather}
M = \bigoplus_{ 0  \leqslant k \leqslant \langle \alpha_i^{\vee}, \lambda\rangle } f^{(k)}_i(M[\lambda ] \cap \ker e_i) 
\end{gather}

We define operators $\tilde{e}_i, \tilde{f}_i$ acting on $M$ by 
\begin{gather}
    \tilde{e}_if_{i}^{(k)}u = f^{(k-1)}_i u \hspace{2mm} \text{and} \hspace{2mm} \tilde{f}_if_i^{(k)}u = f_i^{(k+1)} u, 
\end{gather}
for $u \in M \cap \ker e_i$ and $0 \leqslant k \leqslant \langle \alpha_i^{\vee}, \lambda \rangle$.  These operators are usually called \textbf{Kashiwara operators}. 

Let us take a field $k$ and let $K = k(q)$ be the field of rational functions in a variable $q$ with coefficients in $k$. Let $V$ be a $K$-vector space. For a subring $C$ of $K$, \textbf{$C$-lattice} of $V$ is, by definition, a $C$-submodule $L$ of $V$ such that $V \cong K \tens_C L$. Denote by $A$ the subring of $K$ consisting of functions $f(q)$ in $K$ without a pole at $q=0$. By evaluation map $\operatorname{ev}_0: f(q)  \mapsto f(0)$  we have an isomorphism $$A/qA \cong k$$
\begin{definition}
\label{def_2.-2}
Let $V$ be a $K$-vector space. A \textbf{local base} of $V$ at $q = 0$ is a pair $(L, B)$, where 
\begin{itemize}
    \item $L$ is a $A$-lattice of $V$ that is a free $A$-module 
    \item $B$ is a base of $k$-vector space $L/qL$.
\end{itemize}
    
\end{definition}

Next we define direct sum and tensor product of local bases. 

Let $\{V_j\}$ be a family of $K$-vector spaces and let $(L_j, B_j)$ be a local base of $V_j$. Then $L = \oplus_j L_j \subset \oplus_j V_j$ is a free $A$-lattice of $\oplus_j V_j$ and $B = \coprod_j B_j \subset \oplus (L_j/qL_j) \cong L/qL$ is a base of $L/qL$. Hence, $(L, B)$ is a local base of  $\oplus_j V_j$. We call it \textbf{direct sum} of $\{(L_j, B_j)\}_j$ and denote it by $\oplus_j(L_j, B_j)$. 

Let $V_1$ and $V_2$ be two $K$-vector spaces and let $(L_j, B_j)$ be a local base of $V_j, (j = 1,2)$.  Then $L = L_1 \tens_A L_2 \subset V_1 \tens_K V_2$ is a free $A$-lattice of $V_1 \tens_K V_2$ and \begin{align}
    B= B_1 \tens B_2 = \{b_1 \tens b_2; b_1 \in B_1, b_2 \in B_2 \} \\ 
    \subset (L_1/qL_1) \tens_k (L_2/qL_2) \cong L/qL
\end{align}
is a base of $L/qL$. Hence $(L, B)$ is a local base of $V_1 \tens_K V_2$. We call it \textbf{tensor product} of $(L_1, B_1)$ and $(L_2, B_2)$ and denote it by $(L_1, B_1) \tens (L_2, B_2)$. 

Set $K = \mathbb{C}(q)$. Note that $U_q(\mathfrak{g})$ is a $K$-algebra.  

\begin{definition}
    \label{def_2.1}
Let us take an integrable $U_q(\mathfrak{g})$-module $M$ and let $M = \oplus_{\lambda \in P} M[\lambda]$ be its weight space decomposition.  A \textbf{crystal base} of $M$ is a local base $(L, B)$ of the $K$-vector space $M$, satisfying the following conditions.

\begin{enumerate}
    \item There is a local base $(L[\lambda], B[\lambda])$ of $M[\lambda]$ such that $$(L, B) = \oplus_{\lambda \in P} (L[\lambda], B[\lambda]).$$
    \item $\tilde{f}_iL \subset L$, and $\tilde{e}_i L \subset L$.
    \item $\tilde{f}_iB \subset B \cup \{0\}$, and  $\tilde{e}_i B \subset B \cup \{0\}$.
    \item For $u, v \in B$ and $i \in I$, $u = \tilde{e}_iv$ if and only if $v = \tilde{f}_i u$. 
\end{enumerate}
\end{definition}
It is clear from the definition that direct sum of crystal bases of integrable $U_q(\mathfrak{g})$-modules is a crystal base of the direct sum of these modules.  One of the most important properties of crystal bases of integrable $U_q(\mathfrak{g})$-modules is the "stability" by tensor product. 

\begin{theorem}[Kashiwara \cite{Kashiwara_1990}]
    Let $M_i$ be integrable $U_q(\mathfrak{g})$-modules and let $(L_i, B_i)$ be a crystal base of $M_i$ $(i = 1, 2)$.
    \begin{enumerate}
        \item $(L_1, B_1) \tens (L_2, B_2)$ is a crystal base in $M_1 \tens_{\mathbb{C}(q)} M_2$. 
        \item For $b_i \in B_i$ $(i = 1, 2)$ we have 
    \begin{gather*}
    \tilde e_i (b_1 \tens b_2) = \begin{cases}
    b_1 \tens \tilde e_i  b_2, \hspace{1mm} \text{if} \hspace{1mm} \varepsilon_i(b_2) > \varphi_i(b_1)
    \\
    \tilde e_i b_1 \tens b_2, \hspace{1mm} \text{otherwise}
    \end{cases}\\
    \tilde f_i  (b_1 \tens b_2) = \begin{cases}
    b_1\tens\tilde  f_i b_2, \hspace{1mm} \text{if} \hspace{1mm} \varepsilon_i(b_2) \geqslant \varphi_i(b_1)
    \\
    \tilde f_i b_1 \tens b_2, \hspace{1mm} \text{otherwise}
    \end{cases}
    \end{gather*}

    where \begin{gather*}\varepsilon_i(b) = \max\{ n \geqslant 0 ; \tilde e_i ^{n} b \neq 0\} \\ 
    \varphi_i(b) = \max\{ n \geqslant 0 ; \tilde f_i ^{n} b \neq 0\}
\end{gather*}

    \end{enumerate}
\end{theorem}

In general, an integrable $U_q(\mathfrak{g})$-module is not completely reducible. However, for a semisimple Lie algebra $\mathfrak{g}$ (in fact, even for Kac-Moody Lie algebra), there exists a family of completely reducible integrable $U_q(\mathfrak{g})$-modules. 

\begin{definition}
    Let $\mathcal{O}_{int}(\mathfrak{g})$ denote the category of integrable $U_q(\mathfrak{g})$-modules $M$ such that for any $v \in M$, there exists an integer $l \geqslant 1$, such that $$e_{i_1} \ldots e_{i_l} v = 0$$ for any $i_1, \ldots i_l \in I$. 
\end{definition}
For $\lambda \in P_+$ we denote by $V_\lambda$ the $U_q(\mathfrak{g})$-module generated by $u_\lambda$ with the following defining relations: 

\begin{gather}
    q^{h} u_\lambda = q^{\langle \lambda, h\rangle} u_\lambda, \\
    e_{i}u_\lambda = 0, \\
    f_{i}^{1 + \langle \alpha^{\vee}_i, \lambda\rangle}u_{\lambda} = 0. 
\end{gather}

Then we have the following theorem.

\begin{theorem}[Lusztig \cite{Lusztig}]
\label{th_2.0}
Let $\lambda \in P_+$. 
\begin{enumerate}
    \item $V_\lambda[\lambda] = \{u \in V_\lambda :  e_iu = 0, \forall i \in I\} = \mathbb{C}(q)u_\lambda \neq 0$. 
    \item $V_\lambda$ is an irreducible integrable $U_q(\mathfrak{g})$-module in $\mathcal{O}_{int}(\mathfrak{g})$.
    \item Any $U_q(\mathfrak{g})$-module in $\mathcal{O}_{int}(\mathfrak{g})$ is completely reducible.
    \item Any irreducible $U_q(\mathfrak{g})$-module in $\mathcal{O}_{int}(\mathfrak{g})$ is isomorphic to $V_\lambda$ for some $\lambda \in P_+$. 
\end{enumerate}
\end{theorem}

The following results were proven in the original work of Kashiwara \cite{Kashiwara_1990} for $\mathfrak{g} = A_n, B_n, C_n, D_n$ and then extended to an arbitrary Kac-Moody Lie algebra $\mathfrak{g}$ in \cite{Kashiwara_1991}.   

\begin{theorem}[Kashiwara, \cite{Kashiwara_1990}, \cite{Kashiwara_1991}]
\label{th_2.1}
\begin{enumerate}
    \item $V_\lambda$ has a unique crystal base $(L_\lambda, B_\lambda)$, such that $L_\lambda[\lambda]$ = $Au_\lambda$ and $B_\lambda[\lambda] = \{u_\lambda \bmod qL_{\lambda}\} $. 
    \item $B_\lambda = \{\tilde{f}_{i_1}\tilde{f}_{i_2}\ldots \tilde{f}_{i_k}u_{\lambda} \bmod qL_{\lambda}: l \geqslant 0, i_1, \ldots, i_k  \in I\} \backslash \{0\}$
\end{enumerate}
\end{theorem}

\begin{theorem}[Kashiwara, \cite{Kashiwara_1991}]
\label{th_2.2}
Let  $M$ be a $U_q(\mathfrak{g})$-module in $\Cr{O}_{int}(\mathfrak{g})$ and let $(L,B)$ be a crystal base of $M$ then there exists a $U_q(\mathfrak{g})$-linear isomorphism $$M \rightarrow \bigoplus_j V_{\lambda_j}$$ by which $(L, B)$ is isomorphic to $\bigoplus_j (L_{\lambda_j}, B_{\lambda_j})$. 
\end{theorem}
It is a direct check that for each $\lambda \in P_+$, $B_\lambda$ defined in Theorem \ref{th_2.1} is a normal highest weight crystal of highest weight $\lambda$. We claim that these crystals $B_\lambda, \lambda \in P_+$ form a closed family of normal crystals, which is unique by Theorem \ref{th_1.1}. It suffices to show for any $\lambda, \nu \in P_+$ that $B_{\lambda + \nu}$ can be embedded into $B_\lambda \tens B_\nu$, as any such embedding would send the highest weight element of $B_{\lambda + \nu}$ to the tensor product of highest weight elements of $B_\lambda$ and $B_\nu$. Consider $U_q(\mathfrak{g})$-module $V_\lambda  \tens V_\nu$. The element $u_\lambda \tens u_\nu \in V_\lambda \tens V_\nu$ is annihilated by all $e_i, i \in I$: $$\Delta(e_i) (u_\lambda \tens u_\nu) = e_iu_\lambda \tens t_i^{-1}u_\nu  + u_\lambda \tens e_i u_\nu.$$
 and is of the weight $\lambda + \nu$: 
 $$\Delta(q^{h}) (u_\lambda \tens u_\nu) = q^{\langle \lambda + \nu, h\rangle } u_\lambda \tens u_\nu.$$
Therefore, since the module $V_\lambda \tens V_\nu \in \operatorname{Ob}(\Cr{O}_{int}(\mathfrak{g}))$  is completely reducible by Theorem \ref{th_2.0}, it has an irreducible component isomorphic to $V_{\lambda + \nu}$ generated by the vector $u_\lambda \tens u_\nu$. Hence, by Theorem \ref{th_2.2} there is a $U_q(\mathfrak{g})$-linear embedding of $V_{\lambda + \nu}$ into $V_\lambda \tens V_\nu$, which induces an embedding of $B_{\lambda + \nu}$ into $B_\lambda \tens B_\nu$  and we are done.  

Note that the closedness of the category \underline{\textbf{$\mathfrak{g}$-$\text{Crystals}$}} under the tensor product operation now obviously follows from Theorem \ref{th_2.2}.

\subsection{Spinor crystal}
\label{spinor_crystal}

In the previous two chapters we assumed that $\mathfrak{g}$ is a finite-dimensional complex semisimple Lie algebra. In fact, most of the results discussed in these chapters hold for arbitrary Kac-Moody Lie algebras. Now we consider a very special case by setting $\mathfrak{g} = D_n$. For $\mathfrak{g} = D_n$ there is a fundamental object in the \underline{\textbf{$\mathfrak{g}$-Crystals}} category called \textit{spinor crystal}.  Spinor crystal arises from the crystal base of $U_q(\mathfrak{g})$-module $V_{sp}^{-} \oplus V_{sp}^{+} $ described in section $6.4$ of \cite{Kashiwara_Nakashima}.  Spinor crystal has several remarkable properties discussed in this chapter. 

First, we specify the notations. Let $\mathfrak{h}$ be the Cartan subalgebra of $\mathfrak{g} =D_n$ and  $\{\epsilon_i, i = 1, 2, \ldots n \}$ be the basis of $\mathfrak{h}^{*}$, such that $(\epsilon_i, \epsilon_j) = \delta_{ij}$, where $(\cdot, \cdot)$ is the ad-invariant scalar product on $\mathfrak{h}_{\mathbb{R}}^{*} = \operatorname{span}_{\mathbb{R}}\{\epsilon_i, i = 1, 2, \ldots, n\}$. The roots of $\mathfrak{g}$ have the form $\pm \epsilon_i \pm \epsilon_j, (i \neq j).$ We choose a polarization so that the set of simple roots $\Pi = \{ \alpha_1,  \ldots, \alpha_{n-1}, \alpha_n \}$, where $\alpha_i = \epsilon_i - \epsilon_{i+1}$ for $1 \leqslant i < n$ and $\alpha_n = \epsilon_{n-1} + \epsilon_n$. Let $\Pi^{*} = \{\alpha^{\vee}_1, \ldots \alpha^{\vee}_n \}$ be the set of simple co-roots and $\{\Lambda_i\}_{1 \leqslant i \leqslant n}$ be the set of fundamental weights of $\mathfrak{g}$. We write all weights in the basis $\{\epsilon_i, i = 1, 2, \ldots n \},$ so $(\lambda_1, \lambda_2, \ldots, \lambda_n) = \sum_{i=1}^{n} \lambda_i \epsilon_i.$ Then the weight lattice has the form $$P = \{(\lambda_1, \cdots, \lambda_n) \hspace{1mm} | \hspace{1mm} \lambda_i \in \frac{1}{2}\mathbb{Z}, \lambda_i - \lambda_j \in \mathbb{Z}\}$$ and the set of dominant integral weights $$P_+ = \{(\lambda_1, \cdots, \lambda_n) \in P\hspace{1mm} | \hspace{1mm} \lambda_1 \geqslant \lambda_2 \geqslant \cdots \geqslant \lambda_n, \lambda_{n-1} + \lambda_n \geqslant 0\}.$$ We have the following explicit formulas for fundamental weights of $\mathfrak{g}$: $$\Lambda_i  = \epsilon_1 + \ldots + \epsilon_i$$ for $i = 1, 2 \ldots, n-2$ and \begin{gather*} \Lambda_{n-1} = \frac{1}{2}(\epsilon_1 + \ldots + \epsilon_{n-1} - \epsilon_{n}), \\ \Lambda_{n} = \frac{1}{2}(\epsilon_1 + \ldots + \epsilon_{n-1}+ \epsilon_{n}).
\end{gather*}
\begin{definition}
\label{def_3.1}
Simple roots of $D_n$ are labeled by the index set $\{1, 2, \ldots, n\}$. We say that the index $i$ is \textbf{adjacent} to $j$ iff the roots $\alpha_i, \alpha_j \in \Pi$ share a common edge in the corresponding Dynkin diagram (see the picture below). 

We denote the subset of indices adjacent to $i$ by $\operatorname{AD}(i)$.
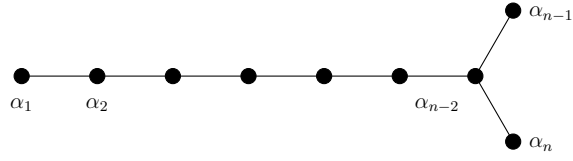
    \begin{center}
    \dynkin[edge length=1cm, root radius=0.1cm, arrow width=0.2cm, labels = {\alpha_1,\alpha_2, , , , , \alpha_{n-2}, \alpha_{n-1},  \alpha_n}]{D}{9}
    \end{center}

Looking at the picture above we can write for $n \geqslant 4$,  \begin{itemize}
    \item $\operatorname{AD} (1) = \{2\}$,
    \item $\operatorname{AD} (i) = \{i-1, i+1\}, \hspace{2mm} \forall  1 < i < n-2$,
    \item $\operatorname{AD}(n-2) = \{n-3, n-1, n\}$,
    \item $\operatorname{AD}(n-1) = \operatorname{AD}(n) = \{n-2\}$.
\end{itemize}
\end{definition}

It is clear from the definition above that $i \in \operatorname{AD}(j)$ $\Leftrightarrow$ $j \in \operatorname{AD}(i)$. 
Notice also that $$ (\alpha_i, \alpha_j) = \begin{cases}
-1, \hspace{1mm} \text{if} \hspace{1mm} i \in \operatorname{AD}(j), \\ 2,  \hspace{1mm}\text{if} \hspace{1mm} i =j , \\ 0, \hspace{1mm } \text{otherwise}.\end{cases}$$

\begin{definition}
\label{def_3.2}
Denote by
$$\mathcal{B}_S = \{ b = (i_1, \ldots, i_n); i_j = \pm \}.$$ We set $$
\operatorname{wt}(b) = \frac{1}{2} \sum_{j=1}^{n} i_j \epsilon_j$$ for $b = (i_1, i_2, \ldots, i_n)$.   
We define the action  of the crystal operators $\tilde{e}_j, \tilde{f}_j, j = 1, 2, \ldots, n-1$ on the set $\mathcal{B_S}$ as 
\begin{gather}
    \tilde{e}_j(i_1, \ldots i_n) = \begin{cases}(i_1, \ldots, \overset{j}{+}, \overset{j+1}{-}, \ldots, i_n),  \hspace{4mm} i_j = - \hspace{2mm} \text{and} \hspace{2mm} i_{j+1} = +,\\
    0, \hspace{2mm} \text{otherwise};
    \end{cases} \\
 \tilde{f}_j(i_1, \ldots i_n) = \begin{cases}(i_1, \ldots, \overset{j}{-}, \overset{j+1}{+}, \ldots, i_n),  \hspace{4mm} i_j = + \hspace{2mm} \text{and} \hspace{2mm} i_{j+1} = -,\\
    0, \hspace{2mm} \text{otherwise}.
    \end{cases}    
\end{gather}
Operators $\tilde{e}_n, \tilde{f}_n$ act the following way 
\begin{gather}
    \tilde{e}_n(i_1, \ldots i_n) = \begin{cases}(i_1, \ldots, \overset{n-1}{+}, \overset{n}{+}),  \hspace{4mm} i_{n-1} = - \hspace{2mm} \text{and} \hspace{2mm} i_{n} = -,\\
    0, \hspace{2mm} \text{otherwise};
    \end{cases} \\
  \tilde{f}_n(i_1, \ldots i_n) = \begin{cases}(i_1, \ldots, \overset{n-1}{-}, \overset{n}{-}),  \hspace{4mm} i_{n-1} = + \hspace{2mm} \text{and} \hspace{2mm} i_{n} = +,\\
    0, \hspace{2mm} \text{otherwise}.
    \end{cases}    
\end{gather}

The set $\mathcal{B}_S$ equipped with the maps $\operatorname{wt}, \tilde{e}_j,\tilde{f}_j $ $(j = 1, 2, \ldots n)$ is an object of the category  \underline{\textbf{$D_n$-Crystals}}  and is called \textbf{spinor crystal}. 

\end{definition}

It is a direct computation that $\mathcal{B}_{S} = \mathcal{B}_{\Lambda_{n-1}} \oplus 
\mathcal{B}_{\Lambda_{n}}$, where $\mathcal{B}_{\Lambda_{n-1}}, 
\mathcal{B}_{\Lambda_{n}}$ are the highest weight crystals, arising from crystal bases of $U_q(\mathfrak{g})$ modules $V_{\Lambda_{n-1}}, V_{\Lambda_{n}}$ respectively. Here we remind that by default the word "crystal" means an object in the \underline{\textbf{$\mathfrak{g}$-Crystals}}  category.  We denote by $P[S] = \Big\{
\big( \pm \frac{1}{2}, \pm \frac{1}{2}, \ldots, \pm \frac{1}{2} \big) \Big\}$ the set of weights of the elements of the spinor crystal. 

Let us observe some simple properties of the spinor crystal. 

\begin{proposition}
\label{prop_3.1}
For any element $b \in \Cr{B}_S$: 
\begin{gather} \label{eq_1}
\tilde e_i^{2}b = \tilde f_i^{2}b  = 0 \Leftrightarrow \varepsilon_i(b) + \varphi_i(b) \leqslant 1, \hspace{2mm} \forall i \in \{1, \ldots, n \} \\ \label{eq_2}
\varepsilon_i(b) \cdot \varepsilon_j(b) = 0 = \varphi_i (b) \cdot \varphi_j(b), \hspace{1mm} \text{if} \hspace{1mm} i \in \operatorname{AD}(j).
\end{gather}
Also, if $b_1, b_2 \in \Cr{B}_S$ and $\operatorname{wt}(b_1) = \operatorname{wt}(b_2)$ then $b_1 = b_2$. 

\end{proposition}
\begin{proof}
The first line is clear from the definition of crystal operators $\tilde e_i, \tilde f_i$. 

If $\varepsilon_i(b) > 0$, then $\varepsilon_i(b) = 1$ and $\varphi_i(b) = 0$ by the first line.  Consider the element $\tilde e_ib$. It follows from the first axiom of $\mathfrak{g}$-crystals that $$\langle \operatorname{wt}(\tilde e_ib), \alpha_j^{\vee} \rangle = \varphi_j(\tilde e_ib) - \varepsilon_j(\tilde e_ib)$$
On the other hand $\operatorname{wt}(\tilde e_ib) = \operatorname{wt}(b) + \alpha_i$, so $$\langle \operatorname{wt}(\tilde e_ib), \alpha_j^{\vee} \rangle = \langle \operatorname{wt}(b) , \alpha_j^{\vee} \rangle + \langle \alpha_i , \alpha_j^{\vee} \rangle = \varphi_j(b) - \varepsilon_j(b) - 1 $$
So, if $\varepsilon_j(b) > 0$, then $\varepsilon_j(b) = 1$,  $\varphi_j(b) =0$  and therefore,$$\varphi_j(\tilde e_ib) - \varepsilon_j(\tilde e_ib) = -2,$$
which is a contradiction since $$\varphi_j(\tilde e_ib) - \varepsilon_j(\tilde e_ib) \geqslant -\varphi_j(\tilde e_ib) - \varepsilon_j(\tilde e_ib) \geqslant -1$$ by the first line. 

The last part of the proposition is evident from the definition of the weight of an element of spinor crystal.  
\end{proof}

The latter proposition implies that the crystal graph of the spinor crystal admits the proper edge coloring (no edges of the same color share a common vertex), where the colors correspond to the indices of the respective crystal operator.  
It turns out that the spinor crystal $\Cr{B}_S$ is in a sense the "smallest generator" of the category \underline{\textbf{$D_n$-Crystals}}.  Any highest weight crystal $\Cr{B}_\lambda$, $\lambda \in P_+$ can be embedded into the tensor power of the spinor crystal $\Cr{B}_S^{\otimes N}$ for some natural $N$. So, by taking tensor powers of the spinor crystal and looking at their connected (irreducible) components we obtain all simple objects of the \underline{\textbf{$D_n$-Crystals}} category. Hence, it is natural to study tensor powers of the spinor crystal.  The following proposition is crucial for this study. 

\begin{proposition}
\label{prop_3.2}
    Let $\mathcal{B}_S$ be the spinor crystal and $\Cr{B}_\lambda$ be the highest weight crystal of the highest weight $\lambda \in P_+$. Then, 
    \begin{gather} \label{eq_3} \mathcal{B}_\lambda \otimes \mathcal{B}_S = \bigoplus_{\substack{\mu \in P[S],  \\ \lambda + \mu \in P_+}} \mathcal{B}_{\lambda + \mu},\end{gather}
    and the highest weight element of the crystal $\mathcal{B}_{\lambda + \mu}$ has the following form:
    $$b_\lambda \otimes b_\mu,$$
    where $b_\mu$ is the only element of the weight $\mu$ in $\mathcal{B}_S$ and $b_\lambda$ is the highest weight element in $\mathcal{B}_\lambda$.
\end{proposition}

\begin{proof}
 $\Cr{B}_\lambda \tens \Cr{B}_S$ can be decomposed into the direct sum of highest weight crystals: \begin{gather} \label{eq_4} \Cr{B}_\lambda \tens \Cr{B}_S = \bigoplus_{\nu \in X} \Cr{B}_{\nu}^{\oplus m_\nu},\end{gather} where $X \subset P_+$. To compute this decomposition it is enough to find all the highest weight elements in $\Cr{B}_\lambda \tens \Cr{B}_S$. 

Let  $b_1 \tens b_2$  be a highest weight element in $\Cr{B}_\lambda \tens \Cr{B}_S$. Then, by definition 
\begin{gather*}
    \tilde e_i (b_1 \tens b_2) = 0, \hspace{2mm} \forall i \in \{1, 2, \ldots n\}
\end{gather*}
Notice, that if $\varepsilon_i (b_1) > 0$, then $\tilde e_i(b_1 \tens b_2) \neq 0$. In fact, from the definition of tensor product, we have \begin{gather}\varepsilon_i(b_1 \tens b_2) = \varepsilon_i(b_1) + \max(0, \varepsilon_i(b_2) - \varphi_i(b_1)) \geqslant \varepsilon_i(b_1) > 0\end{gather}
Therefore, if $b_1 \tens b_2$ is a highest weight element, then $\varepsilon_i(b_1) = 0, \hspace{2mm} \forall i \in \{1, \ldots n\}$ meaning that $b_1$ is the highest weight element of  $\Cr{B_\lambda}$, which we denoted by $b_\lambda$. It follows that in the formula \ref{eq_4} all multiplicities $m_{\nu} = 1$, because each component $\Cr{B}_{\nu}$ is generated by a highest weight element of highest weight $\nu$, and all the highest weight elements in $\Cr{B}_{\lambda} \otimes \Cr{B}_S$ have different weights, as there are no elements of the same weight in $\Cr{B}_S$.

Let $b_{\mu}$ be the unique element of weight $\mu \in P[S]$ in the spinor crystal $\Cr{B}_S$. We want to determine when $b_{\lambda} \tens b_{\mu}$ is a highest weight element in $\Cr{B}_\lambda \tens \Cr{B}_S$. Suppose $\lambda + \mu \in P_+$. By definition of dominant integral weights we have: 
$$\langle \lambda + \mu , \alpha_i^{\vee}\rangle \geqslant 0, \hspace{2mm} \forall i \in \{1, \ldots n\}.$$ 
By the first axiom of crystals: 
\begin{gather}
\varphi_{i}(b_\lambda) = \varphi_i(b_\lambda) - \varepsilon_i(b_\lambda) = \langle \lambda, \alpha_{i}^{\vee}\rangle, \hspace{2mm}  \\
\varphi_i(b_\mu) - \varepsilon_i(b_\mu) = \langle \mu, \alpha_{i}^{\vee}\rangle.
\end{gather}
By linearity of contraction $\langle \lambda + \mu , \alpha_i^{\vee}\rangle = \langle \lambda , \alpha_i^{\vee}\rangle + \langle \mu , \alpha_i^{\vee}\rangle$. So, we obtain that 
\begin{gather}
\label{eq_5}
\varphi_{i}(b_{\lambda}) + \varphi_i(b_\mu) - \varepsilon_i(b_\mu) \geqslant 0,  \hspace{2mm} \forall i \in \{1, \ldots n\} 
\end{gather}
By definition of the tensor product of crystals  \begin{gather}
\varepsilon_i(b_\lambda \tens b_\mu) = \varepsilon_i(b_\lambda) + \max(0, \varepsilon_i(b_\mu) - \varphi_i(b_\lambda)) = \max(0, \varepsilon_i(b_\mu) - \varphi_i(b_\lambda)) \end{gather}
Inequality \ref{eq_5} can be rewritten as 
$$ \varphi_i(b_\mu)  \geqslant \varepsilon_i(b_\mu) - \varphi_{i}(b_{\lambda}),  \hspace{2mm} \forall i \in \{1, \ldots n\}.$$
Since $b_\mu \in \Cr{B}_S$, we have $\varepsilon_i(b_\mu) + \varphi_i(b_\mu) \leqslant 1$, and therefore,
$$1 - \varepsilon_i(b_\mu) \geqslant \varepsilon_i(b_\mu) - \varphi_i(b_\lambda).$$
Notice that if $\varepsilon_i(b_\mu) = 1$, then $\varepsilon_i(b_\mu) - \varphi_i(b_\lambda) \leqslant 0$ by the inequality above. On the other hand, if $\varepsilon_i(b_\mu) = 0$, then obviously $\varepsilon_i(b_\mu) - \varphi_i(b_\lambda) \leqslant 0$, as $\varphi_i(b_\lambda) \geqslant 0$. Therefore, 
$$\varepsilon_i(b_\lambda \tens b_\mu) = \max(0, \varepsilon_i(b_\mu) - \varphi_i(b_\lambda)) = 0, \hspace{2mm} \forall i \in \{1, \ldots n\}.$$
So, if $\lambda + \mu \in P_+$ , then $b_\lambda \tens b_{\mu}$ is a highest weight element in $\Cr{B}_\lambda \tens \Cr{B}_S$. 

Conversely, assume that $\lambda + \mu \notin P_+$. Then, there exists $i \in \{1, 2, \ldots n\}$, such that $\langle \lambda + \mu , \alpha_i^{\vee}\rangle < 0$. For such $i$ we have $$\varphi_i(b_\lambda) + \varphi_{i}(b_\mu) - \varepsilon_{i}(b_\mu) =\langle \lambda , \alpha_i^{\vee}\rangle + \langle \mu , \alpha_i^{\vee}\rangle =\langle \lambda + \mu , \alpha_i^{\vee}\rangle  < 0.$$
This inequality can be rewritten as $$\varepsilon_i(b_\mu) - \varphi_i(b_\lambda) > \varphi_{i}(b_\mu) \geqslant 0.$$ 
Hence, $\varepsilon_i(b_\lambda \tens b_\mu) = \max(0, \varepsilon_i(b_\mu) - \varphi_i(b_\lambda)) > 0$. So, if $\lambda + \mu \notin P_+$, then $b_\lambda \tens b_{\mu}$ is not a highest weight element in $\Cr{B}_\lambda \tens \Cr{B}_S$. 

We conclude that the highest weight elements in $\Cr{B}_\lambda \tens \Cr{B}_S$ are exactly those $b_\lambda \tens b_\mu$, for which $\lambda + \mu \in P_+$. It is evident that the element $b_\lambda \tens b_\mu$ is of the weight $\lambda + \mu$, so it is a highest weight element of the connected component of $\Cr{B}_\lambda \tens \Cr{B}_S$, isomorphic to $\Cr{B}_{\lambda + \mu}$. As a consequence, we obtain the following decomposition of $\Cr{B}_\lambda \tens \Cr{B}_S$ into the direct sum of highest weight crystals: $$\mathcal{B}_\lambda \otimes \mathcal{B}_S = \bigoplus_{\substack{\mu \in P[S],  \\ \lambda + \mu \in P_+}} \mathcal{B_{\lambda + \mu}}.$$

\end{proof}

From Proposition \ref{prop_3.2} we immediately get the following corollary: 

\begin{corollary}
\label{cor_3.1}
    The crystal $\mathcal{B}_S^{\otimes N}$ is isomorphic to the following direct sum of the highest weight crystals: 
\begin{gather}\label{eq_6} \mathcal{B}_S^{\otimes N} \cong \bigoplus_{(\mu_1, \mu_2, \ldots ,\mu_N) \in T^N} \mathcal{B}_{\mu_1 + \mu_2 + \ldots + \mu_N}, \end{gather}

where
\begin{gather} 
\label{eq_7}
    T^{N} = \Big\{(\mu_1, \mu_2, \ldots, \mu_N) \hspace{1mm} | \hspace{1mm} \mu_i \in P[S], \sum_{i=1}^{k} \mu_i \in P_+, \forall k \in \{1,2,\ldots, N\}  \Big\}.
\end{gather}
Moreover, the highest weight element of the connected component $\mathcal{B}_{\mu_1 + \mu_2 + \ldots + \mu_N}$ is equal to $b_{\mu_1} \otimes b_{\mu_2} \otimes \ldots \otimes b_{\mu_N}$, where $b_\mu$ is the only element of $\mathcal{B}_S$ of the weight $\mu$.
\end{corollary}

It is clear that the set $T^N$ can be naturally decomposed into the disjoint union of the following sets:

\begin{gather}
    T^N = \coprod_{\lambda \in \Delta^N} T_\lambda^N,
\end{gather}

where

\begin{gather} \label{eq_8}
    T_\lambda^N = \Big\{ (\mu_1, \mu_2, \ldots, \mu_N) \in T^N\hspace{1mm} | \hspace{1mm} \sum_{i=1}^{N} \mu_i = \lambda\Big\}, \\ \label{eq_9}
    \Delta^N = \Big\{ \sum_{i=1}^{N} \mu_i  \hspace{1mm} | \hspace{1mm} (\mu_1, \mu_{2}, \ldots, \mu_N) \in T^N \Big\}
\end{gather}

The set $\Delta^N$ consists of the highest weights that are present in the decomposition of $\mathcal{B}_S^{\otimes N}$ into the direct sum of highest weight crystals. Connected components of $\mathcal{B}_S^{\otimes N}$ isomorphic to $\mathcal{B}_\lambda$ are indexed by the set $T_\lambda^N$. 

It turns out that both sets $\Delta^N$, $T_\lambda^N$ have nice combinatorial interpretations. We are going to discuss them in chapter \ref{cell_tables_&_sssyt} .

\subsection{${B}_{\infty}$  crystal. Kashiwara's involution}
\label{b_infty&Kas_involution}
In the previous chapters we primarily considered normal crystals from the category \underline{\textbf{$\mathfrak{g}$-Crystals}}. Now, for technical reasons we need to consider a couple of examples of crystals that are not normal. Assume we are in the setting of the chapter \ref{general_info}. 

\begin{example}
\label{ex_10.1}
For $\lambda \in P$, $A_\lambda$ is the crystal consisting of a single element $a_\lambda$ with $\operatorname{wt}(a_\lambda) = \lambda$, $\varepsilon_i(a_\lambda) = \varphi_i(a_\lambda) = - \infty$.  \end{example}
$A_\lambda$ is not a lower or upper normal crystal. We also consider another example of a crystal which is not normal. In the literature this crystal is often denoted by $T_\lambda$. But in this paper we are going to use the notation $T_\lambda^N$ for a completely different thing, so to avoid misunderstandings we are going to use the notation $A_\lambda$ instead. 

\begin{example}
\label{ex_10.2}
For $i \in I$ define the crystal $B_i$ as follows
\begin{gather*}
B_i = \{b_i(n) \hspace{1mm} | \hspace{1mm} n \in \mathbb{Z} \},\\
\operatorname{wt}(b_i(n)) = n\alpha_i, \\
\varphi_i(b_i(n)) = n, \hspace{2mm} \varepsilon_i(b_i(n)) = -n, \\
\varphi_j(b_i(n)) =\varepsilon_j(b_i(n)) = -\infty  \hspace{2mm}\text{if}  \hspace{2mm}j \neq i,
\end{gather*}
with the action 
\begin{gather*}
    \tilde e_i(b_i(n)) = b_i(n+1), \hspace{2mm} \tilde f_i(b_i(n)) = b_i(n-1), \\
    \tilde e_j(b_i(n)) = \tilde f_j(b_i(n)) = 0 \hspace{2mm}\text{if}  \hspace{2mm} j \neq i. 
\end{gather*}
\end{example}

\begin{example}
\label{ex_10.3}
${B}_\infty$ is a crystal associated with the crystal base of $U_q^{-}(\mathfrak{g})$ (cf. \cite{Kashiwara_1991}). We set \begin{gather*}\varepsilon_i(b) = \max \{n \geqslant 0; \tilde e_i^{n}b \neq 0\}, \\ \varphi_i(b) = \varepsilon_i(b) + \langle \alpha_i^\vee, \operatorname{wt}(b) \rangle. \end{gather*} Then ${B}_\infty$ is upper normal but not lower normal. The unique element of ${B}_{\infty}$ of weight $0$ is denoted by $b_\infty$. 
   
\end{example}
It follows from Theorem $5$ in \cite{Kashiwara_1991} that $\Cr{B}_\lambda$ can be regarded as a subcrystal of $B_\infty \tens A_\lambda$. We denote  the respective embedding of crystals by  $\mathcal{\kappa}_{\lambda, \infty}:  \Cr{B}_\lambda \hookrightarrow {B}_\infty \tens A_\lambda$.  It sends $b_\lambda $ to $b_\infty \tens a_\lambda$ and commutes with all the  $\tilde e_i$'s.  Note that this embedding is not strict, since it does not commute with all the $\tilde f_i$'s. 
 
We denote by $*$ the antiautomorphism of $U_q(\mathfrak{g})$ as $\mathbb{C}(q)$-algebra defined by
\begin{gather*}e_i^{*} = e_i, \\ f_i^{*} = f_i,\\ (q^{h})^{*} = q^{-h}.\end{gather*}
This antiautomorphism preserves $U_q^{-}(\mathfrak{g})$. It was proven in \cite{Kashiwara_1991} and \cite{Kashiwara_1993} that $*$ preserves the crystal lattice $L_\infty $ of $U_q^{-}(\mathfrak{g})$ and induces the involutive action on ${B}_{\infty}$ also denoted by $*$, see Theorem $2.1.1$ in \cite{Kashiwara_1993} .  Involution $*$ on $B_\infty$ is called Kashiwara's involution. For $b \in B_\infty$ we shall set \begin{gather*}
    \tilde e_i^{*}(b) :=  (\tilde e_i (b^*))^*, \\
    \tilde f_i^{*} := (\tilde f_i (b^*))^*, \\
    \varepsilon_i^{*}(b) := \varepsilon_i(b^*), \\
    \varphi_i^{*}(b) := \varphi_i(b^*),
\end{gather*}
The following theorem was proved in \cite{Kashiwara_1993}. 
\begin{theorem}[Kashiwara, theorem 2.2.1 in \cite{Kashiwara_1993}]
 \label{th_10.1}  

 For any $i \in I$ there exists a unique strict embedding of crystals: 
 $$\Psi_i: B_\infty \hookrightarrow B_\infty \tens B_i,$$ that sends $b_\infty$ to $b_\infty \tens b_i(0)$. 
  
  Let $b \in B_\infty$, such that $b = \tilde f_i^{*m}b_0$ with $\tilde e_i^* b_0 = 0$  ($b_0$ and $m$ are uniquely determined by $b$). Then, $$\Psi_i(b) = b_0 \tens  \tilde f_i^{m} b_i(0) = b_0 \tens b_i(-m).$$
 Moreover, if $$\Psi_i(b) = b_0 \tens b_i(-m),$$ then $$\Psi_i(\tilde f_i^*b) = b_0 \tens b_i(-m-1).$$
 \end{theorem}
We are interested in computing Kashiwara's involution on a certain type of elements in $B_\infty$. 
\begin{proposition}
\label{prop_10.1}    
Let $b_\omega$ be one of the highest weight elements in spinor crystal $\Cr{B}_S = \Cr{B}_{\Lambda_{n-1}} \oplus \Cr{B}_{\Lambda_n}$ $($i.e. $b_\omega = b_{\Lambda_{n-1}}$ or $b_\omega = b_{\Lambda_n})$. 
Let $i_1, i_2, \ldots, i_m$ be a sequence, such that $i_j \in I$  and $\tilde f_{i_m} \ldots \tilde f_{i_2} \tilde f_{i_1} b_{\omega} \neq 0$.  Consider the following composition of embeddings $$\Phi =\Psi_{i_1} \circ \Psi_{i_{2}} \circ \ldots \circ \Psi_{i_m}: B_\infty \hookrightarrow B_\infty \tens B_{i_m} \hookrightarrow B_\infty \tens B_{i_{m-1}} \tens B_{i_m} \hookrightarrow \ldots \hookrightarrow B_\infty \tens B_{i_1} \tens \ldots \tens B_{i_m} $$
Then $$\Phi(\tilde f_{i_m}^{*} \ldots \tilde f_{i_2}^{*} \tilde f_{i_1}^{*}b_\infty) = b_\infty \tens b_{i_1}(-1) \tens b_{i_2}(-1) \tens \ldots \tens b_{i_m}(-1)$$
and $$\Phi(\tilde f_{i_1} \tilde f_{i_2} \ldots \tilde f_{i_m} b_{\infty}) =  b_\infty \tens b_{i_1}(-1) \tens b_{i_2}(-1) \tens \ldots \tens b_{i_m}(-1).$$
\end{proposition}
\begin{proof}
To prove the first equality it suffices to show that for all $1 \leqslant k \leqslant m$, $$\Psi_{i_k}(\tilde f_{i_k}^* \tilde f_{i_{k-1}}^* \ldots \tilde f_{i_1}^{*}b_\infty) = \tilde f_{i_{k-1}}^{*}\ldots \tilde f_{i_1}^{*}b_\infty \tens b_{i_{k}}(-1).$$
 Theorem \ref{th_10.1} implies that it is enough to show that $\tilde e_{i_k}^*\tilde f_{i_{k-1}}^{*}\ldots \tilde f_{i_1}^{*}b_\infty = 0$. Observe that $\tilde e_{i_k}^*\tilde f_{i_{k-1}}^{*}\ldots \tilde f_{i_1}^{*}b_\infty = (\tilde e_{i_k}\tilde f_{i_{k-1}}\ldots \tilde f_{i_1}b_\infty)^*$.  Since $\tilde f_{i_k} \tilde f_{i_{k-1}} \ldots \tilde f_{i_1} b_\omega \neq 0$ we have $\varphi_{i_k}(\tilde f_{i_{k-1}} \ldots \tilde f_{i_1} b_\omega) = 1$ and hence $\varepsilon_{i_k}(\tilde f_{i_{k-1}} \ldots \tilde f_{i_1} b_\omega) = 0$ or equivalently $\tilde e_{i_k} \tilde f_{i_{k-1}} \ldots \tilde f_{i_1} b_\omega  = 0$. Consider the embedding $\kappa_{\omega, \infty}$ of $\Cr{B}_\omega$ (the connected component of $b_\omega$ in $\Cr{B}_S$) into $B_\infty \tens A_\omega$ defined above. Clearly,  $$ \kappa_{\omega, \infty}(\tilde f_{i_{k-1}} \ldots \tilde f_{i_1} b_\omega) =  \tilde f_{i_{k-1}} \ldots \tilde f_{i_1} b_\infty$$Therefore, since $\kappa_{\omega, \infty}$ commutes with all the $\tilde e_i$'s  $$ \tilde e_{i_k}\tilde  f_{i_{k-1}} \ldots \tilde f_{i_1} b_\infty = \tilde e_{i_k} \kappa_{\omega, \infty}(\tilde f_{i_{k-1}} \ldots \tilde f_{i_1} b_\omega) = \kappa_{\omega, \infty}(\tilde e_{i_k}\tilde f_{i_{k-1}} \ldots \tilde f_{i_1} b_\omega) = 0$$
As $*$ is an involution on $B_\infty$, it follows that $(\tilde e_{i_k}\tilde f_{i_{k-1}}\ldots \tilde f_{i_1}b_\infty)^* = 0$ and we are done. 

To prove the second equality it suffices to show that for all $1 \leqslant k \leqslant m$ 
$$\Psi_{i_k}(\tilde f_{i_1} \ldots \tilde f_{i_{k-1}}\tilde f_{i_k}   b_\infty) =  \tilde f_{i_1} \ldots \tilde f_{i_{k-1}} b_\infty \tens b_{i_k}(-1).$$
Since $\Psi_{i_k}$ is a strict embedding, we have $$\Psi_{i_k}(\tilde f_{i_1} \ldots \tilde f_{i_{k-1}}\tilde f_{i_k}   b_\infty) =  \tilde f_{i_1} \ldots \tilde f_{i_{k-1}} \tilde f_{i_k} (b_\infty \tens b_{i_k}(0)).$$
Clearly, $0 = \varphi_{i_k}(b_\infty) \leqslant \varepsilon_{i_k}(b_{i_k}(0)) = 0$. Hence 
$$\Psi_{i_k}(\tilde f_{i_1} \ldots \tilde f_{i_{k-1}}\tilde f_{i_k}   b_\infty) =  \tilde f_{i_1} \ldots \tilde f_{i_{k-1}} (b_\infty \tens \tilde f_{i_k}  b_{i_k}(0)) =  \tilde f_{i_1} \ldots \tilde f_{i_{k-1}} (b_\infty \tens  b_{i_k}(-1)).$$
It is clear that all $\tilde f_{i_r}$'s for which $i_r \neq i_k$ act on the left multiple, since for any $m \in \mathbb{Z}$,  $\varepsilon_{i_r}(b_{i_k}(m)) = - \infty $ and for any $b \in B_\infty$, $\varphi_{i_r}(b) \in \mathbb{Z}$.  

Let $1\leqslant l < k$ be an index, such that $i_l =  i_k$.  Assume that 
$$\tilde f_{i_1} \ldots \tilde f_{i_{k-1}} (b_\infty \tens  b_{i_k}(-1)) =  \tilde f_{i_1} \ldots \tilde f_{i_{l}} (\tilde f_{i_{l+1}} \ldots \tilde f_{i_{k-1}}b_\infty \tens  b_{i_k}(-1)).$$The latter equality is clear if $l$ is the biggest such index.  Since $\tilde f_{i_k} \ldots \tilde f_{i_1} b_\omega \neq 0$   we have \begin{gather} \langle \operatorname{wt}(\tilde f_{i_l} \ldots \tilde f_{i_1} b_\omega), \alpha_{i_k}^{\vee}\rangle = \varphi_{i_k}(\tilde f_{i_l} \ldots \tilde f_{i_1} b_\omega) -  \varepsilon_{i_k}(\tilde f_{i_l} \ldots \tilde f_{i_1} b_\omega) = -1,  \\
\langle \operatorname{wt}(\tilde f_{i_{k-1}} \ldots \tilde f_{i_1} b_\omega), \alpha_{i_k}^{\vee}\rangle = \varphi_{i_k}(\tilde f_{i_{k-1}} \ldots \tilde f_{i_1} b_\omega) -  \varepsilon_{i_k}(\tilde f_{i_{k-1}} \ldots \tilde f_{i_1} b_\omega) = 1.
\end{gather}
Subtracting last two equalities, we obtain \begin{gather}
-\sum_{j = l +1}^{k-1}\langle  \alpha_j, \alpha_{i_k}^{\vee}\rangle  = \langle \operatorname{wt}(\tilde f_{i_{k-1}} \ldots \tilde f_{i_1} b_\omega) -  \operatorname{wt}(\tilde f_{i_l} \ldots \tilde f_{i_1} b_\omega), \alpha_{i_k}^\vee\rangle = 2.
\end{gather}
Evidently, $$\langle \operatorname{wt}(\tilde f_{i_{l+1}} \ldots \tilde f_{i_{k-1}}b_\infty), \alpha_{i_k}^{\vee} \rangle = -\sum_{j = l +1}^{k-1}\langle  \alpha_j, \alpha_{i_k}^{\vee}\rangle + \langle \operatorname{wt}(b_\infty), \alpha_{i_k}^{\vee}\rangle = 2.$$
On the other hand, 
$$\langle \operatorname{wt}(\tilde f_{i_{l+1}} \ldots \tilde f_{i_{k-1}}b_\infty), \alpha_{i_k}^{\vee} \rangle = \varphi_{i_k}(\tilde f_{i_{l+1}} \ldots \tilde f_{i_{k-1}}b_\infty) - \varepsilon_{i_k}(\tilde f_{i_{l+1}} \ldots \tilde f_{i_{k-1}}b_\infty) = 2.$$ Since the crystal $B_\infty$ is upper-normal $\varepsilon_{i_k}(\tilde f_{i_{l+1}} \ldots \tilde f_{i_{k-1}}b_\infty) \geqslant 0$ and hence $\varphi_{i_k}(\tilde f_{i_{l+1}} \ldots \tilde f_{i_{k-1}}b_\infty) \geqslant 2$. Therefore, as $\varphi_{i_k}(\tilde f_{i_{l+1}} \ldots \tilde f_{i_{k-1}}b_\infty) > \varepsilon_{i_k}(b_{i_k}(-1)) = 1$, 
$$\tilde f_{i_1} \ldots \tilde f_{i_{l}} (\tilde f_{i_{l+1}} \ldots \tilde f_{i_{k-1}}b_\infty \tens  b_{i_k}(-1)) = \tilde f_{i_1} \ldots \tilde f_{i_{l-1}} (\tilde f_{i_{l}} \ldots \tilde f_{i_{k-1}}b_\infty \tens  b_{i_k}(-1)).$$
So, we have shown that if for $l <  k$,  $i_l = i_k$ and all the crystal operators $\tilde f_{i_{l+1}}, \ldots, \tilde f_{i_{k-1}}$ act on the left multiple, then so does the crystal operator $\tilde f_{i_l}$.  We also know that all the crystal operators $\tilde f_{i_r}$ such that $i_r \neq  i_k$  act on the left multiple as well. Hence, by induction we obtain 
$$\Psi_{i_k}(\tilde f_{i_1} \ldots \tilde f_{i_{k-1}}\tilde f_{i_k}   b_\infty) = \tilde f_{i_1} \ldots \tilde f_{i_{k-1}} (b_\infty \tens  b_{i_k}(-1)) =  \tilde f_{i_1} \ldots \tilde f_{i_{k-1}}b_\infty \tens  b_{i_k}(-1)$$ and we are done.
\end{proof}
From the previous proposition we get an immediate key corollary.
\begin{corollary}
    \label{cor_10.1}
Let $b_\omega$ be one of the highest weight elements in spinor crystal $\Cr{B}_S = \Cr{B}_{\Lambda_{n-1}} \oplus \Cr{B}_{\Lambda_n}$ $($i.e. $b_\omega = b_{\Lambda_{n-1}}$ or $b_\omega = b_{\Lambda_n})$.  Let $i_1, i_2, \ldots, i_m$ be a sequence, such that $i_j \in I$  and $\tilde f_{i_m} \ldots \tilde f_{i_2} \tilde f_{i_1} b_{\omega} \neq 0$. Then \begin{gather}
  ( \tilde f_{i_m} \ldots \tilde f_{i_2} \tilde f_{i_1} b_\infty)^* = \tilde f_{i_1} \tilde f_{i_2} \ldots \tilde f_{i_m} b_\infty
\end{gather}
Here $*$ stands for Kashiwara's involution on $B_\infty$ and $b_\infty$ is the highest weight element of $B_\infty$.  
\end{corollary}
\begin{proof}
Evidently, $(\tilde f_{i_m} \ldots \tilde f_{i_2} \tilde f_{i_1} b_\infty)^* =\tilde f_{i_m}^{*} \ldots \tilde f_{i_2}^{*} \tilde f_{i_1}^{*} b_\infty $. By Proposition \ref{prop_10.1} $\Phi(\tilde f_{i_m}^{*} \ldots \tilde f_{i_2}^{*} \tilde f_{i_1}^{*}b_\infty) = \Phi(\tilde f_{i_1} \tilde f_{i_2} \ldots \tilde f_{i_m} b_{\infty})$, where $\Phi$ is a composition of strict embeddings, i.e. is also a strict embedding. Thus, canceling $\Phi$ we obtain the equality needed. 
\end{proof}

\subsection{Some technical results about the structure of crystals}
\label{technical_results}

This subsection is dedicated to some technical results.

\begin{lemma}
\label{lemma_4.5}
Let $\Cr{B}$ be a crystal. Take $b \in \Cr{B}$ and $b_\mu \in \Cr{B}_S$, such that $\operatorname{wt}(b) + \operatorname{wt}(b_{\mu}) \in P_+$. Then, $\tilde{e}_i b_\mu \neq 0$ implies $\tilde f_i b \neq 0$. 
\end{lemma}
\begin{proof}

By the first axiom of $\mathfrak{g}$-crystals and linearity of contraction: 
$$\langle \operatorname{wt}(b) + \operatorname{wt}(b_\mu), \alpha_{i}^{\vee} \rangle =\langle \operatorname{wt}(b), \alpha_{i}^{\vee} \rangle + \langle \operatorname{wt}(b_\mu), \alpha_{i}^{\vee} \rangle = \varphi_i(b) - \varepsilon_i(b) + \varphi_{i}(b_\mu) - \varepsilon_i(b_\mu).$$ Since $\operatorname{wt}(b) + \operatorname{wt}(b_{\mu}) \in P_+$, we have for any $i \in \{1, 2, \ldots n\}$
$$\langle \operatorname{wt}(b) + \operatorname{wt}(b_\mu), \alpha_{i}^{\vee} \rangle \geqslant 0.$$
So, for any $i \in \{1, \ldots n\}$
$$\varphi_i(b) - \varepsilon_i(b) + \varphi_{i}(b_\mu) - \varepsilon_i(b_\mu) \geqslant 0.$$
 If $\tilde e_ib_\mu \neq 0$, then $\varepsilon_i(b_\mu) = 1$ and $\varphi_i(b_\mu) = 0$.  Substituting these values into the latter inequality we obtain
$$\varphi_i(b) \geqslant \varphi_i(b) - \varepsilon_i(b) \geqslant 1.$$

Therefore, $\tilde f_i b \neq 0$ and we are done. 
\end{proof}

The following corollary is easily derived from the previous lemma. 

\begin{corollary}
\label{cor_4.1}
Let $\Cr{B}$ be a crystal. Take $b \in \Cr{B}$ and $b_\mu \in \Cr{B}_S$, such that $\operatorname{wt}(b) + \operatorname{wt}(b_{\mu}) \in P_+$. Then, for any sequence of indexes $i_1, \ldots, i_r \in \{1, \ldots, n\}$ \hspace{1mm} $\tilde{e}_{i_r} \ldots \tilde{e}_{i_1} b_\mu \neq 0$ implies $\tilde{f}_{i_r} \ldots \tilde{f}_{i_1} b \neq 0$.
\end{corollary}

We are ready to formulate the central result of this chapter. 

\begin{theorem}
\label{th_4.1}

Let $\lambda \in P_+$, $\mu \in P[S]$ and $\lambda +\mu \in P_+$. Denote by $b_\lambda$ the highest weight element in $\Cr{B}_{\lambda}$, by $b_\mu$ the element of weight $\mu$ in $\Cr{B}_{S}$. Let $\tilde e_{i_k} \tilde e_{i_{k-1}}\ldots \tilde e_{i_1} b_{\mu}$ be the highest weight element in $\Cr{B}_{S}$. Then in the crystal $\Cr{B}_S \tens \Cr{B}_{\lambda}$ there exists the unique highest weight element of weight $\lambda + \mu$ and it is equal to \begin{gather}\tilde e_{i_k} \tilde e_{i_{k-1}}\ldots \tilde e_{i_1} b_{\mu} \tens \tilde f_{i_k} \tilde f_{i_{k-1}}\ldots \tilde f_{i_1} b_\lambda .\end{gather}   
\end{theorem}

\begin{proof}
       
Let $\tilde e_{i_r} \ldots \tilde e_{i_1}$ be any composition of ascending crystal operators such that $\tilde e_{i_r} \ldots \tilde e_{i_1} b_{\mu} \neq 0$.  Since $\operatorname{wt}(b_\lambda) + \operatorname{wt}(b_\mu) = \lambda + \mu \in P_+$, applying Corollary \ref{cor_4.1}, we obtain $\tilde f_{i_r} \ldots \tilde f_{i_1} b_\lambda \neq 0$. Thus, $\tilde e_{i_k} \tilde e_{i_{k-1}}\ldots \tilde e_{i_1} b_{\mu} \tens \tilde f_{i_k} \tilde f_{i_{k-1}}\ldots \tilde f_{i_1} b_\lambda $ is a non-zero element in $\Cr{B}_S \tens \Cr{B_\lambda}$.

As $\lambda + \mu \in P_+$, Proposition \ref{prop_3.2} implies that $b_\lambda \tens b_\mu$ is the unique highest weight element in $\Cr{B}_\lambda \tens \Cr{B}_S$ of weight $\lambda + \mu$.  Then, $\sigma_{\Cr{B_\lambda, \Cr{B_S}}}(b_\lambda \tens b_\mu)$ is a unique highest weight element of the weight $\lambda + \mu$ in $\Cr{B}_S \tens \Cr{B}_\lambda$.  Hence, we need to show that $$\sigma_{\Cr{B_\lambda, \Cr{B_S}}}(b_\lambda \tens b_\mu) = \tilde e_{i_k} \tilde e_{i_{k-1}}\ldots \tilde e_{i_1} b_{\mu} \tens \tilde f_{i_k} \tilde f_{i_{k-1}}\ldots \tilde f_{i_1} b_\lambda.$$
In the work of Joel Kamnitzer and Peter Tingley \cite{Kamnitzer_Tingley} it was proven that if $b_\lambda \tens c$ is a highest weight element in $\Cr{B}_\lambda \tens \Cr{B}_\nu$, where $\lambda, \nu \in P_+$, then $$\sigma_{\Cr{B}_\lambda, \Cr{B}_S}(b_\lambda \tens c) = b_\nu \tens  c^{*},$$
where $b_\lambda, b_\nu$ are the highest weight elements of crystals $\Cr{B}_\lambda, \Cr{B}_\nu$ respectively and $*$ is Kashiwara's involution on $\Cr{B}_\infty$. Spinor crystal is a direct sum of two highest weight crystals $\Cr{B}_S = \Cr{B}_{\Lambda_{n-1}} \oplus \Cr{B}_{\Lambda_{n}}$. So, $b_\mu$ is an element of a highest weight crystal $\Cr{B}_\omega$, where $\omega$ is either $\Lambda_{n-1}$ or $\Lambda_n$. Commutor is a natural map, so $\sigma_{\Cr{B_\lambda}, \Cr{B_S}}$ sends $\Cr{B}_\lambda \tens \Cr{B}_\omega \subset \Cr{B}_\lambda \tens \Cr{B}_S$ to $\Cr{B}_\omega \tens \Cr{B}_\lambda \subset \Cr{B}_S \tens \Cr{B}_\lambda$. Hence, $$\sigma_{\Cr{B}_\lambda, \Cr{B}_S}(b_\lambda \tens b_\mu) = \sigma_{\Cr{B}_\lambda, \Cr{B}_\omega}(b_\lambda \tens b_\mu) = b_\omega \tens  b_\mu^{*}.$$ By the assumption of the Theorem $b_\omega = \tilde e_{i_k} \tilde e_{i_{k-1}}\ldots \tilde e_{i_1} b_{\mu}$. Hence, as $\tilde f_{i_1} \ldots \tilde f_{i_{k-1}} \tilde f_{i_k}b_\omega = b_\mu \neq 0$, Corollary \ref{cor_10.1} implies that $$b_\mu^{*} = (\tilde f_{i_1} \ldots \tilde f_{i_{k-1}} \tilde f_{i_k} b_\omega)^{*} = \tilde f_{i_k} \tilde f_{i_{k-1}} \ldots  \tilde f_{i_1} b_\lambda.$$ Here we assumed that $\Cr{B}_\omega$ and $\Cr{B}_{\lambda}$ are embedded into $\Cr{B}_\infty$ (more precisely into $\Cr{B}_{\infty} \tens A_{\omega}$ or $\Cr{B}_{\infty } \tens A_{\lambda}$), so that the highest weight elements $b_\lambda$, $b_\nu$ are identified with the  element $b_\infty$ of $\Cr{B}_\infty$. Therefore, $$\sigma_{\Cr{B}_\lambda, \Cr{B}_S}(b_\lambda \tens b_\mu) = b_\omega \tens b_\mu^{*} = \tilde e_{i_k} \tilde e_{i_{k-1}}\ldots \tilde e_{i_1} b_{\mu} \tens \tilde f_{i_k} \tilde f_{i_{k-1}} \ldots \tilde f_{i_1} b_\lambda$$ and we are done. 


\end{proof}

We close this chapter with a proposition that will prove useful in the subsequent chapter. 

\begin{proposition}
\label{prop_4.4}
Let $\Cr{B}$ be a  crystal. Take $b \in \Cr{B}$ and $b_\mu, b_\nu \in \Cr{B}_S$, such that $\operatorname{wt}(b) + \operatorname{wt}(b_\nu) \in P_+$ and $\operatorname{wt}(b) + \operatorname{wt}(b_\mu) + \operatorname{wt}(b_\nu) \in P_+$. Then $\tilde e_{i_r} \ldots \tilde{e}_{i_1} b_\nu \neq 0$ implies that $$\tilde{f}_{i_r}\ldots \tilde f_{i_1}(b \tens b_\mu) =\tilde{f}_{i_r}\ldots \tilde f_{i_1} b \tens b_\mu \neq 0. $$
\end{proposition}
\begin{proof}
  Assume that $\tilde e_{i_r} \ldots \tilde{e}_{i_1} b_\nu \neq 0$.  It follows from Corollary \ref{cor_4.1} that $\tilde{f}_{i_r}\ldots \tilde f_{i_1}(b \tens b_\mu)  \neq 0$, given that $\operatorname{wt}(b \tens b_\mu) + \operatorname{wt}(b_\nu) = \operatorname{wt}(b) + \operatorname{wt}(b_\mu) + \operatorname{wt}(b_\nu)  \in P_+$. 
  Next we show by induction on $r$ that  $\tilde{f}_{i_r}\ldots \tilde f_{i_1}(b \tens b_\mu) =\tilde{f}_{i_r}\ldots \tilde f_{i_1} b \tens b_\mu$. If $r = 0$, then there is nothing to prove.  By induction hypothesis $(\tilde e_{i_{r-1}} \ldots \tilde{e}_{i_1} b_\nu \neq 0)$ we have $$\tilde f_{i_r}\tilde f_{i_{r-1}}\ldots \tilde  f_{i_1}(b \tens b_\mu) = \tilde{f}_{i_r}( \tilde f_{i_{r-1}}\ldots \tilde  f_{i_1}b \tens b_\mu)$$
Notice that $\operatorname{wt}( \tilde f_{i_{r-1}}\ldots \tilde  f_{i_1}b) + \operatorname{wt}(\tilde e_{i_{r-1}} \ldots \tilde{e}_{i_1} b_\nu ) = \operatorname{wt}(b) + \operatorname{wt}(b_\nu) \in P_+$. So, as $\varepsilon_{i_r}(\tilde e_{i_{r-1}} \ldots \tilde{e}_{i_1} b_\nu ) = 1$, Lemma \ref{lemma_4.5} implies  $\varphi_{i_r}(\tilde f_{i_{r-1}}\ldots \tilde  f_{i_1}b) \geqslant 1$. 
Assume that  $$\tilde{f}_{i_r}( \tilde f_{i_{r-1}}\ldots \tilde  f_{i_1}b \tens b_\mu) = \tilde f_{i_{r-1}}\ldots \tilde  f_{i_1}b \tens\tilde{f}_{i_r} b_\mu$$ By definition of the tensor product $\tilde{f}_{i_r}$  acts on the second multiple of $\tilde f_{i_{r-1}}\ldots \tilde  f_{i_1}b \tens b_\mu$ iff $\varepsilon_{i_r}(b_\mu) \geqslant \varphi_{i_r}(\tilde f_{i_{r-1}}\ldots \tilde  f_{i_1}b)$.  Since $\varphi_{i_r}(\tilde f_{i_{r-1}}\ldots \tilde  f_{i_1}b) \geqslant 1$, we conclude that $$\varepsilon_{i_r}(b_\mu) =\varphi_{i_r}(\tilde f_{i_{r-1}}\ldots \tilde  f_{i_1}b) = 1.$$
But then $\varphi_{i_r}(b_\mu) = 0$, and thus $$\tilde f_{i_r}\tilde f_{i_{r-1}}\ldots \tilde  f_{i_1}(b \tens b_\mu)= \tilde{f}_{i_r}( \tilde f_{i_{r-1}}\ldots \tilde  f_{i_1}b \tens b_\mu) = \tilde f_{i_{r-1}}\ldots \tilde  f_{i_1}b \tens\tilde{f}_{i_r} b_\mu = 0,$$
which is a contradiction. 
Hence,  $$\tilde{f}_{i_r}( \tilde f_{i_{r-1}}\ldots \tilde  f_{i_1}b \tens b_\mu) = \tilde{f}_{i_r}\tilde f_{i_{r-1}}\ldots \tilde  f_{i_1}b \tens b_\mu.$$
and we are done. 
\end{proof}

\section{Cell tables and short semi-standard Young tables}
\label{cell_tables_&_sssyt}

This chapter follows my previous paper \cite{Svyatnyy}. It is dedicated to presenting combinatorial interpretations of the sets $\Delta^N$, $T_\lambda^{N} \hspace{1mm} \text{for}\hspace{1mm} \lambda \in \Delta^{N}$ defined in \ref{eq_8}, \ref{eq_9}.

\begin{definition}
\label{def_5.1}

A \textbf{\textit{regular cell diagram}} of the length $N$ and of the height $n$ is a composition of cells and a vertical straight line, called \textbf{axis} of the cell diagram constructed by two sets of non-negative integers $\textbf{{l}} = (l_i)_{1 \leqslant i \leqslant n}$, $\textbf{{r}} = (r_i)_{1 \leqslant i \leqslant n}$ satisfying the following conditions:
\begin{itemize}
 \item $r_i + l_i = N, \hspace{1mm} \forall i \in \{1, \ldots, n\}$
 \item $r_1 \geqslant r_2 \geqslant \ldots \geqslant r_n$
 \item $r_{n-1} \geqslant l_{n}$
\end{itemize}
The construction looks as follows. For every $i \in \{1, \ldots, n\}$ in the $i$-th row to the right of the axis draw $r_i$ cells and to the left of the axis draw $l_i$ cells. The enumeration of rows goes from top to bottom.
\end{definition}

\begin{notation}
We denote the regular cell diagram constructed by the sets $\textbf{{l}}, \textbf{{r}}$ by $D(\textbf{l}, \textbf{r})$.
\end{notation}

\begin{notation}
The set of regular cell diagrams of the length $N$ and height $n$ is denoted by  $\mathfrak{D}(N, n)$.
\end{notation}

\begin{example}
\label{ex_5.1}
As an example we draw the regular cell diagram $D(\textbf{l}, \textbf{r}) \in \mathfrak{D}(7,4)$, where $\textbf{r} =  (5, 4, 4, 3), \hspace{2mm} \textbf{l} = (2, 3, 3, 4)$.

\centering 
\begin{tikzpicture}[scale = 0.6]
\draw[thick,<->] (6,-0.5) -- (6, 4.5) node[anchor=north west] {};
\draw (2,0) rectangle (3,1);
\draw (3,0) rectangle (4,1);
\draw (4,0) rectangle (5,1);
\draw (5,0) rectangle (6,1);
\draw (6,0) rectangle (7,1);
\draw (7,0) rectangle (8,1);
\draw (8,0) rectangle (9,1);
\draw (3,1) rectangle (4,2);
\draw (4,1) rectangle (5,2);
\draw (5,1) rectangle (6,2);
\draw (6,1) rectangle (7,2);
\draw (7,1) rectangle (8,2);
\draw (8,1) rectangle (9,2);
\draw (9,1) rectangle (10,2);
\draw (3,2) rectangle (4,3);
\draw (4,2) rectangle (5,3);
\draw (5,2) rectangle (6,3);
\draw (6,2) rectangle (7,3);
\draw (7,2) rectangle (8,3);
\draw (8,2) rectangle (9,3);
\draw (9,2) rectangle (10,3);
\draw (5,3) rectangle (4,4);
\draw (6,3) rectangle (5,4);
\draw (7,3) rectangle (6,4);
\draw (8,3) rectangle (7,4);
\draw (9,3) rectangle (8,4);
\draw (10,3) rectangle (9,4);
\draw (11,3) rectangle (10,4);
\end{tikzpicture}
\end{example}

For each $\lambda \in P_+$ we define the tuples $$\textbf{{l}}(\lambda, N) = (l_i)_{1 \leqslant i \leqslant n}$$ and $$\textbf{{r}}(\lambda, N) = (r_i)_{1 \leqslant i \leqslant n}.$$
The pair $(l_i, r_i), \forall i \in \{1, \ldots, n\}$ is a solution of the following system of equations: 
\begin{gather}
\begin{cases} 
r_i + l_i = N \\
r_i - l_i = 2 \lambda_i
\end{cases}
\end{gather}
In other words: 

\begin{gather} 
\begin{cases} 
r_i = \frac{N}{2} + \lambda_i \\
l_i = \frac{N}{2} - \lambda_i
\end{cases}
\end{gather}

\begin{proposition}[Proposition 4.3 in \cite{Svyatnyy}]
\label{prop_5.1}
For any $\lambda \in P_+$, $\lambda \in \Delta^N $ iff $\textbf{{r}}(\lambda, N)$, $\textbf{{l}}(\lambda, N) \in (\mathbb{Z}_{+})^{\times n}.$ 
\end{proposition}

It is easy to see that for any $\lambda \in P_+$ the tuples $\textbf{{r}}(\lambda, N)$, $\textbf{{l}}(\lambda, N)$ satisfy the conditions from the definition. 
Therefore, the following definition makes sense. 

\begin{definition}
\label{def_5.2}
 A \textbf{cell diagram, corresponding to}  $\lambda \in \Delta^N$, denoted by $D^N_{\lambda}$, is a regular cell diagram of the length $N$ and height $n$, constructed by $\textbf{{l}}(\lambda, N)$, $\textbf{{r}}(\lambda, N)$.  
$$D_\lambda^N := D(\textbf{{l}}(\lambda, N),\textbf{{r}}(\lambda, N)).$$
\end{definition}

\begin{proposition}[Proposition 4.7 in \cite{Svyatnyy}]
\label{prop_5.2}
The map $\mathcal{K}_N : \Delta^N  \rightarrow \mathfrak{D}(N, n)$, where $$\mathcal{K}_N(\lambda) =  D_{\lambda}^N $$ is a bijection. 
\end{proposition}

We want to give another combinatorial interpretation of the set $\Delta^N$. For that we will need the following corollary of the fundamental result of Roger Howe (Theorem 8 in \cite{Howe})

\begin{proposition}[Corollary 5.8 in \cite{Svyatnyy}]
\label{prop_5.3}

Let $V_S = V_{\Lambda_{n-1}} \oplus V_{\Lambda_{n}}$ be the spinor representation of the Lie algebra $D_n$. Here by $V_\lambda$ we denote the highest weight representation of $D_n$ with the highest weight $\lambda$, and $\Lambda_{n-1}, \Lambda_n$ are the fundamental weights defined in the previous section.

Then $V_S^{\otimes N}$ as a $D_n$-module can be decomposed into the direct sum of simple $D_n$-modules: 

$$V_S^{\otimes N} = \bigoplus_{\lambda \in \Delta^{N} } U_{\lambda} \otimes V_{\lambda},$$
where $U_\lambda$ is the multiplicity space of the simple module $V_\lambda$ in $V_S^{\otimes N}$.

Multiplicity space $U_\lambda$ has a natural structure of a simple $O_N$-module. Moreover, for different $\lambda \in \Delta^N$ multiplicity spaces $U_\lambda$ are non-isomorphic as $O_N$-modules. 
\end{proposition}

From the theorem $19.19$ in \cite{Fulton} we know that the isomorphism classes of simple $O_N$-modules are indexed by the Young diagrams such that the sum of the lengths of the first two columns does not exceed $N$. 

\begin{definition}
\label{def_5.3}
A \textbf{short Young diagram} of the height $N$ is a Young diagram such that the lengths of the first column does not exceed $N$.  
\end{definition}

It follows from the proposition that for each $\lambda \in \Delta^N$ there exists a short Young diagram $\nu$, such that $U_\lambda \underset{O_N}{\cong} R_\nu$, where $R_\nu$ is the irreducible representation of $O_N$, corresponding to the short Young diagram $\nu$. 

\begin{notation}
    Denote by $\operatorname{SYD}(N, n)$ the set of short Young diagrams of the height $N$ that have at most $n$ columns.
\end{notation}

In \cite{Svyatnyy} the following theorem was proved. 
\begin{theorem}[Theorem 6.1 in \cite{Svyatnyy}]
\label{th_5.1}
Consider the following map $$\mathcal{F}_N: \mathfrak{D}(N, n) \rightarrow \operatorname{SYD}(N, n).$$ We define it the following way. If $n$ is even, then $$\mathcal{F}_N (D(\textbf{l}, \textbf{r})) =  (l_n, l_{n-1}, \ldots, l_1)^{t},$$ if $n$ is odd, then $$\mathcal{F}_N (D(\textbf{l}, \textbf{r})) =  (r_n, l_{n-1}, \ldots, l_1)^{t}.$$ We claim that the map $\mathcal{F}_N$ is a bijection and moreover the image of the cell diagram $D^{N}_{\lambda}$, corresponding to $\lambda \in \Delta^{N} $, under the map $\mathcal{F}_N$ is the short Young diagram $\nu$, such that the multiplicity space $U_\lambda$ as $O_N$-module is isomorphic to $R_{\nu}$.
\end{theorem}

We provide some examples illustrating the map $\mathcal{F}_N$. 

\begin{example}
\label{ex_5.2}
Let  $n = 4, N = 7$ and $ \hspace{1mm} \lambda = (\frac{3}{2}, \frac{1}{2}, \frac{1}{2}, -\frac{1}{2}).$
Then $D_{\lambda}^{N}$ is the regular cell diagram from the example \ref{ex_5.1}. Map $\mathcal{F}_N$ sends it to the short Young diagram drawn on the right. 

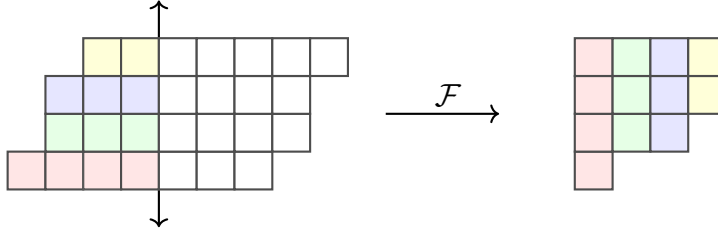
\centering{
\begin{tikzpicture}
\draw[thick,<->] (2,-0.5) -- (2, 2.5) node[anchor=north west] {};
\filldraw[color=black!70, fill=red!10,  thick](0,0) rectangle (0.5,0.5);
\filldraw[color=black!70, fill=red!10,  thick] (0.5,0) rectangle (1,0.5);
\filldraw[color=black!65, fill=red!10,  thick] (1,0) rectangle (1.5,0.5);
\filldraw[color=black!70, fill=red!10,  thick] (1.5,0) rectangle (2,0.5);
\filldraw[color=black!70, fill=white!10,  thick] (2,0) rectangle (2.5,0.5);
\filldraw[color=black!70, fill=white!10,  thick] (2.5,0) rectangle (3,0.5);
\filldraw[color=black!70, fill=white!10,  thick] (3,0) rectangle (3.5,0.5);

\filldraw[color=black!70, fill=green!10,  thick]  (0.5,0.5) rectangle (1,1);
\filldraw[color=black!70, fill=green!10,  thick]  (1,0.5) rectangle (1.5,1);
\filldraw[color=black!70, fill=green!10,  thick]  (1.5,0.5) rectangle (2,1);
\filldraw[color=black!70, fill=white!10,  thick] (2,0.5) rectangle (2.5,1);
\filldraw[color=black!70, fill=white!10,  thick] (2.5,0.5) rectangle (3,1);
\filldraw[color=black!70, fill=white!10,  thick] (3,0.5) rectangle (3.5,1);
\filldraw[color=black!70, fill=white!10,  thick] (3.5,0.5) rectangle (4,1);

\filldraw[color=black!70, fill=blue!10,  thick] (0.5,1) rectangle (1,1.5);
\filldraw[color=black!70, fill=blue!10,  thick] (1,1) rectangle (1.5,1.5);
\filldraw[color=black!70, fill=blue!10,  thick] (1.5,1) rectangle (2,1.5);
\filldraw[color=black!70, fill=white!10,  thick] (2,1) rectangle (2.5,1.5);
\filldraw[color=black!70, fill=white!10,  thick] (2.5,1) rectangle (3,1.5);
\filldraw[color=black!70, fill=white!10,  thick] (3,1) rectangle (3.5,1.5);
\filldraw[color=black!70, fill=white!10,  thick] (3.5,1) rectangle (4,1.5);

\filldraw[color=black!70, fill=yellow!15,  thick] (1,1.5) rectangle (1.5,2);
\filldraw[color=black!70, fill=yellow!15,  thick] (1.5,1.5) rectangle (2,2);
\filldraw[color=black!70, fill=white!10,  thick] (2,1.5) rectangle (2.5,2);
\filldraw[color=black!70, fill=white!10,  thick] (2.5,1.5) rectangle (3,2);
\filldraw[color=black!70, fill=white!10,  thick] (3,1.5) rectangle (3.5,2);
\filldraw[color=black!70, fill=white!10,  thick] (3.5,1.5) rectangle (4,2);
\filldraw[color=black!70, fill=white!10,  thick] (4,1.5) rectangle (4.5,2);

\draw[thick,->] (5,1) -- (6.5,1);
\filldraw[black] (5.5,1.25) circle (0pt) node[anchor=west]{$\mathcal{F}$};

\filldraw[color=black!70, fill=red!10,  thick](7.5,0) rectangle (8,0.5);
\filldraw[color=black!70, fill=red!10,  thick](7.5,0.5) rectangle (8,1);
\filldraw[color=black!70, fill=red!10,  thick](7.5,1) rectangle (8,1.5);
\filldraw[color=black!70, fill=red!10,  thick](7.5,1.5) rectangle (8,2);

\filldraw[color=black!70, fill=green!10,  thick] (8,0.5) rectangle (8.5,1);
\filldraw[color=black!70, fill=green!10,  thick] (8,1) rectangle (8.5,1.5);
\filldraw[color=black!70, fill=green!10,  thick] (8,1.5) rectangle (8.5,2);

\filldraw[color=black!70, fill=blue!10,  thick] (8.5,0.5) rectangle (9,1);
\filldraw[color=black!70, fill=blue!10,  thick] (8.5,1) rectangle (9,1.5);
\filldraw[color=black!70, fill=blue!10,  thick] (8.5,1.5) rectangle (9,2);

\filldraw[color=black!70, fill=yellow!15,  thick] (9,1) rectangle (9.5,1.5);
\filldraw[color=black!70, fill=yellow!15,  thick] (9,1.5) rectangle (9.5,2);

\end{tikzpicture}}
\newline

\end{example}
\begin{example}
\label{ex_5.3}
We also present an example for odd $n$. Let $n = 5$, $N = 6$ and $ \hspace{1mm} \lambda = (3,2, 1, 1,-1).$
Then $D_{\lambda}^{N}$ is the regular cell diagram drawn on the left. Map $\mathcal{F}_N$ sends it to the short Young diagram drawn on the right.  

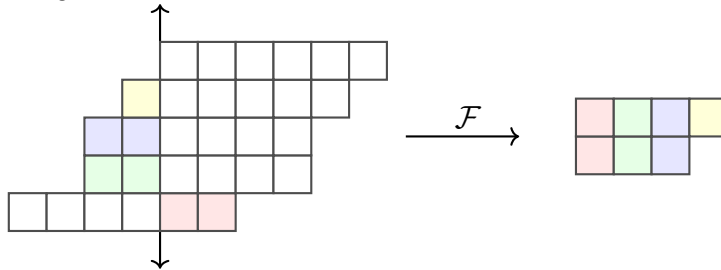
\centering
\begin{tikzpicture}

\draw[thick,<->] (2,-0.5) -- (2, 3) node[anchor=north west] {};
\filldraw[color=black!70, fill=white!10,  thick](0,0) rectangle (0.5,0.5);
\filldraw[color=black!70, fill=white!10,  thick] (0.5,0) rectangle (1,0.5);
\filldraw[color=black!70, fill=white!10,  thick] (1,0) rectangle (1.5,0.5);
\filldraw[color=black!70, fill=white!10,  thick] (1.5,0) rectangle (2,0.5);
\filldraw[color=black!70, fill=red!10,  thick] (2,0) rectangle (2.5,0.5);
\filldraw[color=black!70, fill=red!10,  thick] (2.5,0) rectangle (3,0.5);

\filldraw[color=black!70, fill=green!10,  thick]  (1,0.5) rectangle (1.5,1);
\filldraw[color=black!70, fill=green!10,  thick]  (1.5,0.5) rectangle (2,1);
\filldraw[color=black!70, fill=white!10,  thick] (2,0.5) rectangle (2.5,1);
\filldraw[color=black!70, fill=white!10,  thick] (2.5,0.5) rectangle (3,1);
\filldraw[color=black!70, fill=white!10,  thick] (3,0.5) rectangle (3.5,1);
\filldraw[color=black!70, fill=white!10,  thick] (3.5,0.5) rectangle (4,1);

\filldraw[color=black!70, fill=blue!10,  thick] (1,1) rectangle (1.5,1.5);
\filldraw[color=black!70, fill=blue!10,  thick] (1.5,1) rectangle (2,1.5);
\filldraw[color=black!70, fill=white!10,  thick] (2,1) rectangle (2.5,1.5);
\filldraw[color=black!70, fill=white!10,  thick] (2.5,1) rectangle (3,1.5);
\filldraw[color=black!70, fill=white!10,  thick] (3,1) rectangle (3.5,1.5);
\filldraw[color=black!70, fill=white!10,  thick] (3.5,1) rectangle (4,1.5);

\filldraw[color=black!70, fill=yellow!15,  thick] (1.5,1.5) rectangle (2,2);
\filldraw[color=black!70, fill=white!10,  thick] (2,1.5) rectangle (2.5,2);
\filldraw[color=black!70, fill=white!10,  thick] (2.5,1.5) rectangle (3,2);
\filldraw[color=black!70, fill=white!10,  thick] (3,1.5) rectangle (3.5,2);
\filldraw[color=black!70, fill=white!10,  thick] (3.5,1.5) rectangle (4,2);
\filldraw[color=black!70, fill=white!10,  thick] (4,1.5) rectangle (4.5,2);

\filldraw[color=black!70, fill=white!10,  thick] (2,2) rectangle (2.5,2.5);
\filldraw[color=black!70, fill=white!10,  thick] (2.5,2) rectangle (3,2.5);
\filldraw[color=black!70, fill=white!10,  thick] (3,2) rectangle (3.5,2.5);
\filldraw[color=black!70, fill=white!10,  thick] (3.5,2) rectangle (4,2.5);
\filldraw[color=black!70, fill=white!10,  thick] (4,2) rectangle (4.5,2.5);
\filldraw[color=black!70, fill=white!10,  thick] (4.5,2) rectangle (5,2.5);

\draw[thick,->] (5.25,1.25) -- (6.75,1.25);
\filldraw[black] (5.75,1.5) circle (0pt) node[anchor=west]{$\mathcal{F}$};

\filldraw[color=black!70, fill=red!10,  thick](7.5,0.75) rectangle (8,1.25);
\filldraw[color=black!70, fill=red!10,  thick](7.5,1.25) rectangle (8,1.75);

\filldraw[color=black!70, fill=green!10,  thick](8,0.75) rectangle (8.5,1.25);
\filldraw[color=black!70, fill=green!10,  thick](8,1.25) rectangle (8.5,1.75);

\filldraw[color=black!70, fill=blue!10,  thick](8.5,0.75) rectangle (9,1.25);
\filldraw[color=black!70, fill=blue!10,  thick](8.5,1.25) rectangle (9,1.75);

\filldraw[color=black!70, fill=yellow!15,  thick] (9,1.25) rectangle (9.5,1.75);

\end{tikzpicture}
\end{example}

Now we want to construct the combinatorial interpretation of the set $T_\lambda^N$ for all $\lambda \in \Delta^N$. We start with the necessary definitions. 

\begin{definition}
\label{def_5.4}
Let $D_1 = D(\textbf{l}^{(1)}, \textbf{r}^{(1)}) \in \mathfrak{D}(N_1, n)$ and $D_2 = D(\textbf{l}^{(2)}, \textbf{r}^{(2)}) \in \mathfrak{D}(N_2, n)$ be two regular cell diagrams of height $n$. We say that the diagram $D_1$ contains diagram $D_2$ (notation: $D_1 \supset D_2$)  iff  $l_i^{(1)}  \geqslant l_i^{(2)}$ and $r_i^{(1)}  \geqslant r_i^{(2)}$ for all $i \in \{1, 2, \ldots, n\}$. 
\end{definition}

\begin{definition}
\label{def_5.5}
\textbf{Regular cell table} of the shape $D^{(N)} \in \mathfrak{D}(N, n)$ is a sequence of regular cell diagrams $(D^{(1)}, D^{(2)}, \ldots, D^{(N)})$, such that $D^{(i)} \in \mathfrak{D}(i, n)$ and $D^{(i)} \supset D^{(j)}$ for $i \geqslant j$.  

We denote the set of regular cell tables of the shape $D^{(N)}$ by $\mathfrak{Ctab}(D^{(N)})$. 
\end{definition}

\begin{proposition}[Proposition 7.4 in \cite{Svyatnyy}]
\label{prop_5.4}
For $\lambda \in \Delta^N $ we consider a map $$\mathcal{I}_{\lambda}: T_\lambda^{N} \rightarrow \mathfrak{Ctab}(D_{\lambda}^{N}),$$ defined by the formula below.
\begin{gather} 
\mathcal{I}_{\lambda}:  (\mu_1, \mu_{2}, \ldots, \mu_N) \mapsto (\mathcal{K}_1(\mu_1), \mathcal{K}_2(\mu_1 + \mu_2), \ldots, \mathcal{K}_i\Big(\sum_{j=1}^{i} \mu_j\Big), \ldots, \mathcal{K}_N (\lambda)),
\end{gather}
where $(\mu_1, \mu_{2}, \ldots, \mu_N) \in T_\lambda^N$.
Then the map $\mathcal{I}_\lambda$ is a bijection. 
\end{proposition}

\begin{example}
\label{ex_5.4}
For clarity we provide an example of a regular cell table.

Let $x = (D^{(1)}, D^{(2)}, \ldots, D^{(7)})$ be the regular cell table of the shape $D^{(7)} = D^7_{\lambda} \in \mathfrak{D}(7,4)$, where $\lambda = (\frac{3}{2}, \frac{1}{2}, \frac{1}{2}, -\frac{1}{2})$. Set 

\begin{itemize}
    \item $\mu_1 = (\frac{1}{2}, \frac{1}{2},\frac{1}{2}, -\frac{1}{2})$
    \item $\mu_2 = (\frac{1}{2}, \frac{1}{2},\frac{1}{2}, \frac{1}{2})$
    \item $\mu_3 = (\frac{1}{2}, -\frac{1}{2},-\frac{1}{2}, -\frac{1}{2})$
    \item $\mu_4 = (\frac{1}{2}, \frac{1}{2},\frac{1}{2}, \frac{1}{2})$
    \item $\mu_5 = (\frac{1}{2}, -\frac{1}{2},-\frac{1}{2}, -\frac{1}{2})$
    \item $\mu_6 = (-\frac{1}{2}, \frac{1}{2},\frac{1}{2}, -\frac{1}{2})$
    \item $\mu_7 = (-\frac{1}{2}, -\frac{1}{2},-\frac{1}{2}, \frac{1}{2})$
\end{itemize}

Notice that $(\mu_1, \mu_2, \ldots, \mu_7) \in T_\lambda^{7}$.  
Then $x = \mathcal{I}_\lambda(\mu_1, \mu_2, \ldots, \mu_7) \in \mathfrak{Ctab}(D^7_{\lambda})$. We can visualize $x$ in the following way:

\centering

\begin{tikzpicture}
\draw[thick,<->] (6,-0.5) -- (6, 4.5) node[anchor=north west] {};
\draw (2,0) rectangle (3,1);
\draw (2.75, 0.5) circle (0pt)  node[anchor=east]{$6$};
\draw (3,0) rectangle (4,1);
\draw (3.75, 0.5) circle (0pt)  node[anchor=east]{$5$};
\draw (4,0) rectangle (5,1);
\draw (4.75, 0.5) circle (0pt)  node[anchor=east]{$3$};
\draw (5,0) rectangle (6,1);
\draw (5.75, 0.5) circle (0pt)  node[anchor=east]{$1$};
\draw (6,0) rectangle (7,1);
\draw (6.75, 0.5) circle (0pt)  node[anchor=east]{$2$};
\draw (7,0) rectangle (8,1);
\draw (7.75, 0.5) circle (0pt)  node[anchor=east]{$4$};
\draw (8,0) rectangle (9,1);
\draw (8.75, 0.5) circle (0pt)  node[anchor=east]{$7$};

\draw (3,1) rectangle (4,2);
\draw (3.75, 1.5) circle (0pt)  node[anchor=east]{$7$};
\draw (4,1) rectangle (5,2);
\draw (4.75, 1.5) circle (0pt)  node[anchor=east]{$5$};
\draw (5,1) rectangle (6,2);
\draw (5.75, 1.5) circle (0pt)  node[anchor=east]{$3$};
\draw (6,1) rectangle (7,2);
\draw (6.75, 1.5) circle (0pt)  node[anchor=east]{$1$};
\draw (7,1) rectangle (8,2);
\draw (7.75, 1.5) circle (0pt)  node[anchor=east]{$2$};
\draw (8,1) rectangle (9,2);
\draw (8.75, 1.5) circle (0pt)  node[anchor=east]{$4$};
\draw (9,1) rectangle (10,2);
\draw (9.75, 1.5) circle (0pt)  node[anchor=east]{$6$};

\draw (3,2) rectangle (4,3);
\draw (3.75, 2.5) circle (0pt)  node[anchor=east]{$7$};
\draw (4,2) rectangle (5,3);
\draw (4.75, 2.5) circle (0pt)  node[anchor=east]{$5$};
\draw (5,2) rectangle (6,3);
\draw (5.75, 2.5) circle (0pt)  node[anchor=east]{$3$};
\draw (6,2) rectangle (7,3);
\draw (6.75, 2.5) circle (0pt)  node[anchor=east]{$1$};
\draw (7,2) rectangle (8,3);
\draw (7.75, 2.5) circle (0pt)  node[anchor=east]{$2$};
\draw (8,2) rectangle (9,3);
\draw (8.75, 2.5) circle (0pt)  node[anchor=east]{$4$};
\draw (9,2) rectangle (10,3);
\draw (9.75, 2.5) circle (0pt)  node[anchor=east]{$6$};

\draw (5,3) rectangle (4,4);
\draw (4.75, 3.5) circle (0pt)  node[anchor=east]{$7$};
\draw (6,3) rectangle (5,4);
\draw (5.75, 3.5) circle (0pt)  node[anchor=east]{$6$};
\draw (7,3) rectangle (6,4);
\draw (6.75, 3.5) circle (0pt)  node[anchor=east]{$1$};
\draw (8,3) rectangle (7,4);
\draw (7.75, 3.5) circle (0pt)  node[anchor=east]{$2$};
\draw (9,3) rectangle (8,4);
\draw (8.75, 3.5) circle (0pt)  node[anchor=east]{$3$};
\draw (10,3) rectangle (9,4);
\draw (9.75, 3.5) circle (0pt)  node[anchor=east]{$4$};
\draw (11,3) rectangle (10,4);
\draw (10.75, 3.5) circle (0pt)  node[anchor=east]{$5$};
\end{tikzpicture}

Here $D^{(i)} \in \mathfrak{D}(i, n)$ is the regular cell diagram formed by all the cells with the numbers less than or equal to $i$. 
\end{example}

We conclude that the set of regular cell tables of the shape $D_\lambda^N$ is the combinatorial interpretation of the set $T_\lambda^N$. We want to discuss another combinatorial interpretation of the set $T_\lambda^N$. 

\begin{definition}
\label{def_5.6}
A \textbf{semi-standard short Young tableau} of the shape $\nu \in \operatorname{SYD}(N, n)$ is the sequence of the short Young diagrams $(\nu^{(1)}, \nu^{(2)}, \ldots, \nu^{(N)})$, such that $\nu^{(i)} \in \operatorname{SYD}(i, n), \nu^{(N)} = \nu$, $\nu^{(i)} \subset \nu^{(i+1)}$ and $\nu^{(i+1)} - \nu^{(i)}$ is a horizontal strip for all $1 \leqslant i < N$. 
    
We denote the set of semi-standard short Young tableaux of the shape $\nu$ by $\operatorname{SSSYT}(\nu, N)$. 
\end{definition}

\begin{proposition}[Corollary 7.11 in \cite{Svyatnyy}]
\label{prop_5.5}
Consider a map  $\mathcal{Y}_\lambda: \mathfrak{Ctab}(D_\lambda^{N}) \rightarrow \operatorname{SSSYT}(\nu, N)$, where $\nu = \mathcal{F}_N(D_\lambda^{N})$ defined by the following formula.
    \begin{gather} \label{bijection_ctab_sssyt}
    \mathcal{Y}_\lambda: (D^{(1)}, D^{(2)}, \ldots, D^{(N)} = D_\lambda^N) \mapsto (\mathcal{F}_1(D^{(1)}), \mathcal{F}_2(D^{(2)}), \ldots, \nu).
\end{gather}

The map $\mathcal{Y}_\lambda$ is a bijection. 
\end{proposition}

It follows that the set of short semi-standard Young tableaux of the shape $\nu = D_\lambda^{N}$ is another combinatorial interpretation of the set $T_\lambda^N$.

\begin{remark}
    The latter bijection $\mathcal{Y}_\lambda$ comes from the identification of two bases in the multiplicity space $U_\lambda$. One of them is the Gelfand-Tsetlin basis with respect to chain of the orthogonal groups $O_1 \subset O_2 \subset \ldots \subset O_N$ indexed by the set of short semi-standard Young tableaux $\operatorname{SSSYT}(\nu, N)$, where $\nu = \Cr{F}_N(D_\lambda^{N})$. Another one is called the principal basis and it consists of the highest weight vectors of the weight $\lambda$ of the irreducible components of $V_S^{\otimes N}$ that occur if we consecutively decompose the tensor product adding one multiple each time. It turns out that the principal basis of $U_\lambda$ is indexed by the set $T_\lambda^N \cong \mathfrak{Ctab}(D_\lambda^N)$ and is in fact equal to the Gelfand-Tsetlin basis of $U_\lambda$. Identification of  these bases gives us the bijection from the previous proposition.
\end{remark}

\section{Cactus group}
\label{final_chapter}

\subsection{Action of the cactus group on crystals}
\label{def_of_action_of_cactus_group}
Let $\mathcal{B}_1, \mathcal{B}_2, \ldots, \mathcal{B}_N \in$ $Ob($\underline{$\mathfrak{g}$-$\text{Crystals}$}). The word crystal is reserved for an object of the \underline{$\mathfrak{g}$-$\text{Crystals}$} category for the rest of this chapter. Using the commutors $\sigma_{A,B} : A\tens B \rightarrow B \tens A$ defined in \cite{Henriques_Kamnitzer} one can construct isomorphisms of crystals $$\Cr{B}_1 \tens \Cr{B}_2 \tens \ldots \tens \Cr{B}_N \rightarrow \Cr{B}_{\rho(1)} \tens \Cr{B}_{\rho(2)} \tens \ldots \tens \Cr{B}_{\rho(N)},$$ for any permutation $\rho$.  If $1 \leq p \leq r < q \leq N$, we denote by $\sigma_{p,r,q}$ the isomorphism defined by the following formula: 
\begin{gather*}
    (\sigma_{p,r,q})_{\mathcal{B}_1 \otimes \mathcal{B}_2 \otimes \ldots \otimes \mathcal{B}_N} := 1_{\mathcal{B}_1 \otimes \ldots \otimes \mathcal{B}_{p-1}} \otimes \sigma_{\mathcal{B}_{p} \otimes \ldots \otimes \mathcal{B}_r, \mathcal{B}_{r+1} \otimes \ldots \otimes \mathcal{B}_q} \otimes 1_{\mathcal{B}_{q+1} \otimes \ldots \otimes \mathcal{B}_{N}}: \\
    \mathcal{B}_1 \otimes \ldots \otimes  \mathcal{B}_{p-1} \otimes \mathcal{B}_{p} \otimes \ldots \otimes \mathcal{B}_{r} \otimes \mathcal{B}_{r+1} \otimes \ldots \otimes \mathcal{B}_{q} \otimes \mathcal{B}_{q+1} \otimes \ldots \otimes   \mathcal{B}_N \rightarrow \\ \rightarrow \mathcal{B}_1 \otimes \ldots \otimes  \mathcal{B}_{p-1} \otimes \mathcal{B}_{r+1} \otimes \ldots \otimes \mathcal{B}_{q} \otimes \mathcal{B}_{p} \otimes \ldots \otimes \mathcal{B}_{r} \otimes \mathcal{B}_{q+1} \otimes \ldots \otimes   \mathcal{B}_{N}
\end{gather*}
We will use isomorphisms $\sigma_{p,r,q}$ as building blocks to construct some natural isomorphisms. For $1 \leq p \leq q \leq N$ we define isomorphisms 
\begin{gather*}s_{p, q}: \mathcal{B}_1 \otimes \ldots \otimes  \mathcal{B}_{p-1} \otimes \mathcal{B}_{p} \otimes \mathcal{B}_{p+1} \otimes \ldots \otimes \mathcal{B}_{q-1} \otimes \mathcal{B}_{q}  \otimes \mathcal{B}_{q+1}\otimes \ldots \mathcal{B}_N \rightarrow \\ \rightarrow \mathcal{B}_1 \otimes \ldots \otimes  \mathcal{B}_{p-1} \otimes \mathcal{B}_{q} \otimes \mathcal{B}_{q-1} \otimes \ldots \otimes \mathcal{B}_{p+1} \otimes \mathcal{B}_{p}  \otimes \mathcal{B}_{q+1}\otimes \ldots \mathcal{B}_N 
\end{gather*}
recursively by setting $s_{p, p+1} = \sigma_{p, p, p+1}$ and $s_{p,q} = \sigma_{p, p, q} \circ s_{p+1, q}$ for $q-p > 1$. By convention $s_{p,p} = 1$.

We will also adopt the following notation 
\begin{notation}
  For $p < q$  we denote by $\tilde{s}_{p, q}$ the involutive element of the symmetric group $S_n$ which reverses the segment $[p, q]$. 
  $$\tilde{s}_{p, q} = \begin{pmatrix}
      1 & \ldots & p-1 & p & \ldots & q & q+1 & \ldots & n \\
      1 & \ldots & p-1 & q & \ldots & p & q+1 & \ldots & n 
  \end{pmatrix}$$
\end{notation}

It was proved in \cite{Henriques_Kamnitzer} that the isomorphisms defined above satisfy certain relations. We restate these results here. 

\begin{proposition}[Lemma 3 \& Lemma 4 in \cite{Henriques_Kamnitzer}]
\label{prop_6.1}
Isomorphisms $s_{p,q}$ satisfy the following conditions:

    \begin{enumerate}
        \item $s_{p, q} \circ s_{p, q} = 1$
        \item $s_{p, q} \circ s_{k, l} = s_{k, l} \circ s_{p, q}$  if  $p < q  < k < l$ 
        \item $s_{p,q} \circ s_{k,l} = s_{m,n} \circ s_{p,q} $ if $p < q$ contains $k < l$, where $m = \tilde{s}_{p,q}(l)$ and $n = \tilde{s}_{p,q}(k)$. 
        \item $s_{p,q} = \sigma_{p,r,q} \circ s_{p,r} \circ s_{r+1, q}$
    \end{enumerate}
    
\end{proposition}

It makes sense to give the definition of the cactus group right after this proposition. 

\begin{definition}
\label{def_6.1}
Let $C_N$ be the group with generators $\mathbf{s}_{p,q}$ for $1 \leq p < q \leq N$ and relations: 

    \begin{enumerate}
        \item $\mathbf{s}_{p,q}^{2} = 1.$
        \item $\mathbf{s}_{p,q}\mathbf{s}_{k,l} = \mathbf{s}_{k,l}\mathbf{s}_{p,q}$ if $p < q $ and $k < l$ are disjoint. 
        \item $\mathbf{s}_{p,q}\mathbf{s}_{k,l} = \mathbf{s}_{m,n}\mathbf{s}_{p,q}$ if $p < q$ contains $k < l$, where $m = \tilde{s}_{p,q}(l)$ and $n = \tilde{s}_{p,q}(k)$. 
    \end{enumerate}
    
The group $C_N$ is called the \textbf{cactus group}. 

\end{definition}

Let $\mathcal{B}$ be a crystal. Notice that there is a natural action of the cactus group $C_N$ on $\mathcal{B}^{\otimes N}$ by automorphisms. 

We define the group homomorphism $\psi_\mathcal{B}: C_N \rightarrow \text{Aut}(\mathcal{B}^{\otimes N})$ by the following formula
\begin{gather}
    \psi_\mathcal{B}: \mathbf{s}_{p,q} \mapsto s_{p,q} \in \text{Aut}(\mathcal{B}^{\otimes N})
\end{gather}

\begin{notation}
    Denote the set of the highest weight elements of a crystal $\mathcal{B}$ by $\operatorname{Sing}(\mathcal{B})$. We also denote by $\operatorname{Sing}_\lambda(\mathcal{B}) \subset \operatorname{Sing}(\Cr{B})$ the set of the highest weight elements of the crystal $\mathcal{B}$ of weight $\lambda$. It is clear that $$\coprod_{\lambda  \in P_+} \operatorname{Sing}_{\lambda} (\Cr{B}) = \operatorname{Sing}(\Cr{B}).$$
\end{notation}

Any morphism of crystals sends $g: \Cr{B}_1 \rightarrow \Cr{B}_2$ the highest weight elements to the highest weight elements, therefore it induces a map between the sets of highest weight elements in these crystals $g^{res}: \operatorname{Sing(\Cr{B}_1)} \rightarrow \operatorname{Sing}(\Cr{B}_2)$. On the other hand, highest weight elements generate the whole crystal, so morphism $g$ can be recovered from the induced map $g^{res}$.

It follows from the Schur's lemma \ref{prop_1.2} that an automorphism $g$ of a crystal $\Cr{B}$ permutes its connected components (which are highest weight crystals) of the same highest weight. The induced map $g^{res}$ permutes the highest weight elements of the same weight in $\Cr{B}$ and these permutations uniquely determine the automorphism $g$. In other words, for any $g \in \operatorname{Aut}(\Cr{B})$ we have $g^{res} \in \operatorname{Aut}(\operatorname{Sing}(\Cr{B}))$ and $g^{res}(\operatorname{Sing}_{\lambda}(\Cr{B})) = \operatorname{Sing}_\lambda(\Cr{B})$ for any $\lambda \in P_+$.



In the work of Michael Chmutov, Max Glick and Pavlo Pylyavskii \cite{Chmutov} another description of the cactus group was presented. 

\begin{theorem}[Theorem 1.8 in \cite{Chmutov}]

\label{th_6.1}

Let $G_N$ be the group with generators $\mathbf{t}_i$, $i \in \{1, 2,\ldots N-1\}$ and relations 
\begin{gather}
\mathbf{t}_i^{2} = 1 \\
\mathbf{t}_i\mathbf{t}_j = \mathbf{t}_j\mathbf{t}_i, \hspace{1mm} \text{if} \hspace{1mm}  |i - j| > 1 \\
(\mathbf{t}_i\mathbf{s}_{k-1}\mathbf{s}_{k-j}\mathbf{s}_{k-1})^{2} = 1, \hspace{1mm} \text{where} \hspace{1mm} i+1 < j < k , \hspace{1mm} \text{and}\\ 
\mathbf{s}_i = \mathbf{t}_1(\mathbf{t}_2\mathbf{t}_1)\ldots (\mathbf{t}_i\mathbf{t}_{i-1}\ldots \mathbf{t}_1)
\end{gather}
Then $G_N$ is isomorphic to the cactus group $C_N$. The isomorphism is given by the map
$$\mathbf{t}_1 \mapsto \mathbf{s}_{1,2}, \hspace{1mm} \mathbf{t}_2 \mapsto \mathbf{s}_{1,2}\mathbf{s}_{1,3}\mathbf{s}_{1,2}, \hspace{1mm} \mathbf{t}_i \mapsto \mathbf{s}_{1,i}\mathbf{s}_{1,i+1}\mathbf{s}_{1,i}\mathbf{s}_{1, i-1}, \hspace{1mm} \text{if} \hspace{1mm} i> 2$$
in one direction and $$\mathbf{s}_{i,j} \mapsto \mathbf{s}_{j-1}\mathbf{s}_{j-i}\mathbf{s}_{j-1}, \hspace{1mm} \text{where} \hspace{1mm} \mathbf{s}_i = \mathbf{t}_1(\mathbf{t}_2\mathbf{t}_1)\ldots (\mathbf{t}_i\mathbf{t}_{i-1}\ldots \mathbf{t}_1)$$ in the other.

\end{theorem}

Identify $G_N$ with $C_N$. We call the generators $\mathbf{t}_i$ of cactus group $C_N$ the \textit{Bender-Knuth generators}. It turns out that the action of Bender-Knuth generators $\mathbf t_i$ on $\Cr{B}^{\tens N}$ given by $\psi_{\Cr{B}}$ can be expressed as a composition of no more than two commutors (i.e., automorphisms $\sigma_{p,r,q}$).  

\begin{proposition}

\label{prop_6.2}

Let $\mathcal{B}$ be a crystal. Bender-Knuth generators $\mathbf{t}_i$  of the group $C_N$ for $i \in \{2, \ldots, N-1\}$ act on $\mathcal{B}^{\otimes N}$ by the following composition of commutors: \begin{gather} \label{eq_24}\psi_{\mathcal{B}}(\mathbf{t}_i) = \sigma_{1,1,i} \circ \sigma_{1, i, i+1} \in \text{Aut}(\mathcal{B}^{\otimes N}), \end{gather} where $\sigma_{p,r,q}$ is the isomorphism defined at the start of this chapter. 

For $i = 1$, the isomorphism $\sigma_{1,1, i}$ was not defined, but if we define it to be the identity map, then the formula \ref{eq_24} will be true  for $i=1$ too. 

\end{proposition}

\begin{proof}

Let $i > 2$. From Theorem \ref{th_6.1} we know that $t_i := \psi_{\mathcal{B}}(\mathbf{t}_i) \in \text{Aut}(\mathcal{B}^{\otimes N})$ satisfies the equation \begin{gather} t_i = s_{1,i} \circ s_{1, i+1} \circ s_{1, i} \circ s_{1, i-1}.\end{gather}
By definition of the automorphism $s_{p,q}$ we can rewrite the previous equation in the following way: 
\begin{gather} t_i = (\sigma_{1,1,i} \circ s_{2,i}) \circ s_{1, i+1} \circ s_{1, i} \circ s_{1, i-1}\end{gather}
From the third equation of Proposition \ref{prop_6.1} it follows that $s_{1, i+1} \circ s_{2,i} = s_{2,i} \circ s_{1, i+1} $. 
So, we get 
\begin{gather} \label{eq_25}t_i = \sigma_{1,1,i} \circ (s_{1, i+1}  \circ s_{2,i}) \circ s_{1, i} \circ s_{1, i-1} \end{gather}
Using the last equation of  Proposition \ref{prop_6.1} we can write  $$s_{1, i+1} = \sigma_{1, i, i+1} \circ s_{1, i} \circ s_{i+1, i+1} = \sigma_{1, i, i+1} \circ s_{1, i}$$
Substituting it into our formula \ref{eq_25} we get 
\begin{gather} \label{eq_26} t_i = \sigma_{1,1,i} \circ (\sigma_{1, i, i+1} \circ s_{1, i})  \circ s_{2,i} \circ s_{1, i} \circ s_{1, i-1}\end{gather}
By definition of $s_{p,q}$ we have $$s_{1,i} = \sigma_{1,1,i} \circ s_{2,i}.$$Once again using the last equation of Proposition \ref{prop_6.1} we obtain $$s_{1,i} = \sigma_{1,i-1,i} \circ s_{1, i-1}.$$
Substituting the last two equalities into the formula \ref{eq_26} as follows:
\begin{gather} t_i = \sigma_{1,1,i} \circ \sigma_{1, i, i+1} \circ (\sigma_{1,1,i}\circ s_{2, i})  \circ s_{2,i} \circ (\sigma_{1,i-1, i} \circ s_{1, i-1}) \circ s_{1, i-1} = \sigma_{1,1,i} \circ \sigma_{1, i, i+1} \circ (\sigma_{1,1,i}\circ \sigma_{1,i-1, i}) \end{gather}
But from Proposition \ref{prop_1.3} it follows that $\sigma_{1,1,i}\circ \sigma_{1,i-1, i} = 1$. Therefore,
\begin{gather} \label{eq_27} t_i = \sigma_{1,1,i} \circ \sigma_{1, i, i+1}\end{gather}

We also need to consider the case $i = 2$ separately. 
From Theorem \ref{th_6.1}:
\begin{gather}t_2 = s_{1,2} \circ s_{1,3} \circ s_{1,2}\end{gather}
Last equation of Proposition \ref{prop_6.1} implies that $s_{1,3} = \sigma_{1,2,3} \circ s_{1,2}$. Also, it is clear that  $s_{1,2} = \sigma_{1,1,2}$. Therefore, we get: \begin{gather}
t_2 = \sigma_{1,1,2} \circ \sigma_{1,2,3} \circ s_{1,2} \circ s_{1,2} = \sigma_{1,1,2} \circ \sigma_{1,2,3}. 
\end{gather}
For $i = 1$, Theorem \ref{th_6.1} implies  that \begin{gather}t_1 = s_{1,2} = \sigma_{1,1,2}. \end{gather}
So, we are done.
\end{proof}

The primary goal of this chapter is to explore the action of the Bender-Knuth generators $\mathbf{t}_i$ of the cactus group $C_N$ on the $N$-th tensor power of the spinor crystal $\mathcal{B}_S^{\otimes N}$. In other words, we want to compute $t_i := \psi_{\mathcal{B}_S}(\mathbf{t}_i) \in \text{Aut}(\mathcal{B}_S^{\otimes N})$. As was observed above, the automorphism $t_i$ can be recovered from the induced map $t_i^{res} \in \operatorname{Aut}(\operatorname{Sing}(\mathcal{B}_S^{\otimes N}))$. The rest of this chapter is dedicated to the computation of the induced map $t_i^{res}$.                                                                                                                                                                                                         
Firstly, we prove the locality of $t_i^{res}$. Locality in this case means that $t_i^{res}$  may change only $i$-th and $i+1$-th  factors of the highest weight element  $b_{\mu_1} \otimes b_{\mu_2} \otimes \ldots \otimes b_{\mu_{N}} \in \text{Sing}(\mathcal{B}_S^{\otimes N})$. 

\begin{proposition}
\label{prop_6.3}
Denote by $b_{\mu_i} \in \Cr{B}_S$ the element of weight $\mu_i \in P[S]$.  Let $b_{\mu_1} \otimes b_{\mu_{2}} \otimes \ldots \otimes b_{\mu_N} \in \operatorname{Sing}_{\lambda}(\mathcal{B}_S^{\otimes N})$ be a highest weight element of highest weight $\lambda$ in $\mathcal{B}_S^{\otimes N}$. In other words, $(\mu_1, \mu_2, \ldots \mu_N) \in T_\lambda^N$. Then $\forall i \in \{1, 2, \ldots N-1\}$:
$$t_i^{res}(b_{\mu_1} \otimes  \ldots  b_{\mu_{i-1}}\otimes b_{\mu_{i}} \otimes b_{\mu_{i+1}} \otimes b_{\mu_{i+2}} \ldots \otimes b_{\mu_N}) = b_{\mu_1} \otimes  \ldots  b_{\mu_{i-1}}\otimes b_{\tilde\mu_{i+1}} \otimes b_{\tilde\mu_{i}} \otimes b_{\mu_{i+2}} \ldots \otimes b_{\mu_N},$$ where $(\mu_1, \ldots, \mu_{i-1}, \tilde\mu_{i}, \tilde\mu_{i+1}, \mu_{i+2}, \ldots, \mu_N) \in T_\lambda^{N}$.

\end{proposition}
\begin{proof}
It follows from proposition \ref{prop_6.2} that
\begin{gather}t_i^{res}(b_{\mu_1} \tens b_{\mu_{2}} \tens \ldots \tens b_{\mu_N}) = \sigma_{1,1,i} \circ \sigma_{1, i, i+1}(b_{\mu_1} \tens b_{\mu_{2}} \tens \ldots \tens b_{\mu_N}).\end{gather}
It is evident that $t_i^{res}$ does not act on the tensor factors with indices greater than $i+1$: 
\begin{gather} \label{eq_29} \sigma_{1,1,i} \circ \sigma_{1,i,i+1}(b_{\mu_1} \tens \ldots \tens b_{\mu_N}) = \sigma_{1,1,i} \circ \sigma_{1,i,i+1}( b_{\mu_{1}} \tens \ldots \tens b_{\mu_{i+1}}) \tens \ldots \tens b_{\mu_N}. \end{gather}

Therefore, it is sufficient to compute $\sigma_{1,1,i} \circ \sigma_{1,i,i+1}(b_{\mu_{1}} \tens \ldots \tens b_{\mu_{i+1}})$. From Corollary \ref{cor_3.1} we know all the highest weight elements in $\Cr{B}_S^{\tens N}$. In particular, for all $j \leqslant N$, the element $b_{\mu_1} \tens \ldots \tens b_{\mu_{j}}$ is the highest weight element in $\Cr{B}_{S}^{\tens j}$ of the highest weight $\sum_{k = 1}^{j} \mu_k$.

Consider $\sigma_{1,i,i+1}(b_{\mu_{1}}  \tens \ldots \tens b_{\mu_{i+1}})$. The element $b_{\mu_{1}} \tens \ldots \tens b_{\mu_{i+1}}$ is the unique highest weight element of highest weight $\sum_{k=1}^{i+1} \mu_k$ in the crystal $\Cr{B}_{\mu_{1} + \ldots + \mu_i} \tens \Cr{B}_S  \subset \Cr{B}_S^{\tens i} \tens \Cr{B}_S   = \Cr{B}_S^{\tens i+1}$. Since the commutor $\sigma_{1,i,i+1}$ is a natural map, we have $\sigma_{1,i, i+1}( \Cr{B}_{\mu_{1} + \ldots + \mu_i} \tens\Cr{B}_S ) = \Cr{B}_S \tens \Cr{B}_{\mu_{1} + \ldots + \mu_i}  \subset  \Cr{B}_S \tens\Cr{B}_{S}^{\tens i} $. Hence, $\sigma_{1,i,i+1}$ maps $b_{\mu_{1}} \tens \ldots \tens b_{\mu_{i+1}}$ to the unique highest weight element of highest weight $\sum_{k=1}^{i+1} \mu_k$ in $\Cr{B}_S \tens \Cr{B}_{\mu_{1} + \ldots + \mu_i}  $. Let us investigate what this element looks like. Clearly, $\Cr{B}_S \tens \Cr{B}_{\mu_{1} + \ldots + \mu_i}  \subset  \Cr{B}_S \tens\Cr{B}_{S}^{\tens i} = \Cr{B}_{S}^{\tens i+1}$, and we know the structure of the highest weight elements in $\Cr{B}_{S}^{\tens i+1}$. Thus, we can write the following:
$$\sigma_{1,i,i+1}(b_{\mu_{1}} \tens  \ldots \tens b_{\mu_{i+1}}) = b_{\
\Lambda(\mu_{i+1})} \tens \tilde{\mathbf{f}}(b_{\mu_1} \tens \ldots \tens b_{\mu_i}) ,$$
where $\Lambda(\mu_{i+1}) = \begin{cases}\Lambda_{n-1} \\ \Lambda_n \end{cases}$ depending on the highest weight of the connected component of the element $b_{\mu_{i+1}}$ in $\Cr{B}_S$, and $\tilde{\mathbf{f}}$ is some composition of descending crystal operators. The right side of the latter equality can be rewritten as follows:
$$b_{\Lambda(\mu_{i+1})} \tens \tilde{\mathbf{f}}(b_{\mu_1} \tens \ldots \tens b_{\mu_i}) = 
b_{\Lambda(\mu_{i+1})} \tens \tilde{\mathbf{f}_1}(b_{\mu_1} \tens \ldots \tens b_{\mu_{i-1}}) \otimes \tilde{\mathbf{f}_2} (b_{\mu_i}),$$
where $\tilde{\mathbf{f}}_1$, $\tilde{\mathbf{f}}_2$ are some compositions of descending crystal operators.

As $b_{\Lambda(\mu_{i+1})} \tens \tilde{\mathbf{f}_1}(b_{\mu_1} \tens \ldots \tens b_{\mu_{i-1}}) \otimes \tilde{\mathbf{f}_2} (b_{\mu_i})$ is a highest weight element in $\Cr{B}_S^{\tens i+1}$, it follows that $b_{\Lambda(\mu_{i+1})} \tens \tilde{\mathbf{f}_1}(b_{\mu_1} \tens \ldots \tens b_{\mu_{i-1}})$ is a highest weight element in $\Cr{B}_S^{\tens i}$. Moreover, $b_{\Lambda(\mu_{i+1})} \tens \tilde{\mathbf{f}_1}(b_{\mu_1} \tens \ldots \tens b_{\mu_{i-1}})$ clearly lies in $\Cr{B}_S \tens \Cr{B}_{\mu_{1}+ \ldots +\mu_{i-1}} \subset \Cr{B}_{S}^{\tens i}$. Therefore, $\sigma_{1,1,i}(b_{\Lambda(\mu_{i+1})} \tens \tilde{\mathbf{f}_1}(b_{\mu_1} \tens \ldots \tens b_{\mu_{i-1}}))$ is a highest weight element in $ \Cr{B}_{\mu_{1}+ \ldots +\mu_{i-1}}\tens \Cr{B}_S  \subset \Cr{B}_{S}^{\tens i}$. Then, from proposition \ref{prop_3.2} we obtain that
$$\sigma_{1,1,i}(b_{\Lambda(\mu_{i+1})} \tens \tilde{\mathbf{f}_1}(b_{\mu_1} \tens \ldots \tens b_{\mu_{i-1}})) = b_{\mu_{1}} \tens \ldots \tens b_{\mu_{i-1}} \tens b_{\tilde \mu_{i}},$$
where $\tilde{\mu}_{i} \in P[S]$ is some weight, such that $(\mu_1, \ldots, \mu_{i-1}, \tilde \mu_{i}) \in T^{i}$

From the above, we derive the following chain of equalities:
\begin{gather}
\label{eq_30}
    \sigma_{1,1,i} \circ \sigma_{1,i,i+1}( b_{\mu_{1}} \tens \ldots \tens b_{\mu_{i+1}}) = \sigma_{1,1,i}( b_{\Lambda(\mu_{i+1})} \tens \tilde{\mathbf{f}_1}(b_{\mu_1} \tens \ldots \tens b_{\mu_{i-1}}) \otimes \tilde{\mathbf{f}_2} (b_{\mu_i})) =\\ \label{eq_31} = \sigma_{1,1,i} (b_{\Lambda(\mu_{i+1})} \tens \tilde{\mathbf{f}_1}(b_{\mu_1} \tens \ldots \tens b_{\mu_{i-1}})) \otimes \tilde{\mathbf{f}_2} (b_{\mu_i}) = b_{\mu_{1}} \tens \ldots \tens b_{\mu_{i-1}} \tens b_{\tilde \mu_{i}} \otimes \tilde{\mathbf{f}_2} (b_{\mu_i}).
\end{gather}

Substituting the computation results of \ref{eq_30} and \ref{eq_31} into \ref{eq_29}, we obtain:
\begin{gather} \label{eq_200}
\sigma_{1,1,i} \circ \sigma_{1,i,i+1}(b_{\mu_1}  \tens \ldots \tens b_{\mu_N}) = b_{\mu_{1}} \tens \ldots \tens b_{\mu_{i-1}} \tens b_{\tilde \mu_{i}} \otimes \tilde{\mathbf{f}_2} (b_{\mu_i}) \tens b_{\mu_{i+2}} \ldots \tens b_{\mu_N}.
\end{gather}
Setting $b_{\tilde \mu_{i+1}} := \tilde{\mathbf{f}_2} (b_{\mu_i})$, we get the needed equality. As $t_i^{res}(\operatorname{Sing}_{\lambda}(\Cr{B}_S^{\tens N})) = \operatorname{Sing}_{\lambda}(\Cr{B}_S^{\tens N})$, we conclude that $$ b_{\mu_1} \otimes  \ldots  \tens b_{\tilde\mu_{i}} \otimes b_{\tilde\mu_{i+1}} \otimes  \ldots \otimes b_{\mu_N} \in \operatorname{Sing}_{\lambda}(\Cr{B}^{\tens N}_S) \Leftrightarrow (\mu_1, \ldots, \tilde\mu_{i}, \tilde\mu_{i+1}, \ldots, \mu_N) \in T_\lambda^{N}.$$
Thus, the locality of the map $t_i^{res} \in \Aut(\Sing(\Cr{B}_S^{\tens N}))$ is proved.
\end{proof}

The next step is to explicitly compute $t_i^{res}$. It follows from the proof of Proposition \ref{prop_6.3} that in order to compute $t_i^{res}$ it is enough to find $b_{\tilde\mu_{i+1}}= \tilde {\textbf{f}}_2(b_{\mu_i})$, as $t_i^{res}$ preserves weights. 
By definition of $\tilde {\textbf{f}}_2$ we have $$\sigma_{1,i,i+1}(b_{\mu_{1}} \tens  \ldots \tens b_{\mu_{i+1}}) = b_{\Lambda(\mu_{i+1})} \tens \tilde{\mathbf{f}_1}(b_{\mu_1} \tens \ldots \tens b_{\mu_{i-1}}) \otimes \tilde{\mathbf{f}_2} (b_{\mu_i}). $$
Element $b_{\mu_1} \tens \ldots \tens b_{\mu_i} \tens b_{\mu_{i+1}}$ is the only highest weight element in $\Cr{B}_{\mu_1 + \ldots + \mu_i} \tens \Cr{B}_S \subset \Cr{B}_S^{\tens i} \tens \Cr{B}_S$ of highest weight $\sum_{k=1}^{i+1} \mu_k$. So, $\sigma_{1,i,i+1}(b_{\mu_{1}} \tens  \ldots \tens b_{\mu_{i+1}})$ is the only highest weight element in $ \Cr{B}_S \tens \Cr{B}_{\mu_1 + \ldots + \mu_i} \subset \Cr{B}_S \tens \Cr{B}_S^{\tens i}$ of highest weight $\sum_{k=1}^{i+1} \mu_k$. By Theorem \ref{th_4.1} the highest weight element in $\Cr{B}_S \tens \Cr{B}_{\mu_1 + \ldots + \mu_i}$ of highest weight $\sum_{k=1}^{i+1} \mu_k$ can be written as 
$$\sigma_{1,i,i+1}(b_{\mu_{1}} \tens  \ldots \tens b_{\mu_{i+1}}) = \tilde e_{i_r} \ldots \tilde e_{i_1}b_{\mu_{i+1}} \tens \tilde f_{i_r} \ldots \tilde f_{i_1}(b_{\mu_1} \tens \ldots \tens b_{\mu_i}),$$where $\tilde e_{i_r} \ldots \tilde e_{i_1}$ is any composition of ascending crystal operators, such that $\tilde e _{i_1} \ldots \tilde e_{i_r}b_{\mu_{i+1}} = b_{\Lambda(\mu_{i+1})}$ is a highest weight element in $\Cr{B}_S$.  For brevity, set $\lambda  := \mu_1 + \ldots + \mu_{i-1}$ and  $b_\lambda := b_{\mu_1} \tens \ldots \tens b_{\mu_{i-1}}$.
Our goal is to understand which descending crystal operators from composition $\tilde f_{i_r} \ldots \tilde f_{i_1}$ act on the first factor and which on the second in $\tilde f_{i_r} \ldots \tilde f_{i_1}(b_\lambda \tens b_{\mu_i}) = \tilde{\mathbf{f}_1}(b_{\lambda}) \otimes \tilde{\mathbf{f}_2} (b_{\mu_i})$.

Note that in general there are many ways to choose the composition of ascending crystal operators $\tilde e _{i_r} \ldots \tilde e_{i_1}$ that sends $b_{\mu_{i+1}}$ to a highest weight element of $\Cr{B}_S$. We are going to fix the choice of  this composition for each triple $(\lambda, \mu_i, \mu_{i+1})$, such that \begin{itemize}
    \item $\lambda \in \Delta^{i-1}$, 
    \item  $\mu_i, \mu_{i+1} \in P[S]$,
    \item $\lambda + \mu_i, 
    \hspace{1mm}\lambda + \mu_i + \mu_{i+1 } \in P_+$.
\end{itemize} 
Triplets satisfying the conditions above are called \textbf{admissible}. For an admissible triplet $(\lambda, \mu_i, \mu_{i+1})$ we define the \textit{free intervals} of the set $\{1, 2, \ldots n\}$. 

A \textbf{Free interval} denoted by $\operatorname{Fr}$ is a maximal by inclusion subset of $\{1,2, \ldots, n\}$, such that 
\begin{itemize}
    \item $\lambda_{j_1} = \lambda_{j_2}, \hspace{2mm} \forall j_1, j_2 \in \operatorname{Fr}.$
    \item $(\mu_{i})_{j}  \cdot (\mu_{i+1})_j < 0, \hspace{2mm}  \forall j \in \operatorname{Fr}$ 
\end{itemize}

\begin{example}
\label{ex_6.1}
We give an example of an admissible triplet  $(\lambda, \mu_{i}, \mu_{i+1})$  and list the free intervals defined by it. 

\begin{center}
\begin{tabular}{ c | c c c c c c c c c c c c }
 № & 1 & \textbf{\textcolor{red}{2}} & \textbf{\textcolor{red}{3}} & \textbf{\textcolor{red}{4}} & 5 & 6 & \textbf{\textcolor{blue}{7}} & \textbf{\textcolor{blue}{8}} & 9 & \textbf{\textcolor{green}{10}} & \textbf{\textcolor{green}{11}} & \textbf{\textcolor{purple}{12}} \\ 
 \hline
  $\mu_{i+1}$ & $+$  & $-$  &  $-$ &$+$&$-$&$+$&$-$&$-$&$-$&$-$&$+$&$+$ \\  
 \hline
 $\mu_{i}$ & $+$ & $+$ & $+$ & $-$ & $-$ & $+$ & $+$ & $+$ & $-$ & $+$ &$-$ &$-$ \\
\hline
 $\lambda$ & $4$ & $4$&$4$ & $4$ &$4$ &$2$ &$2$ &$2$ &$2$ &$1$ & $1$ &$0$
\end{tabular}
\end{center}
\vspace{2mm}
The list of the free intervals for this triplet of weights: \begin{itemize}
    \item $\operatorname{Fr}^{1} = \{2,3,4\}$,
    \item $\operatorname{Fr}^{2} = \{7,8\}$,
    \item $\operatorname{Fr}^{3} = \{10,11\}$,
    \item $\operatorname{Fr}^{4} = \{12\}$.
\end{itemize} 

Free intervals are highlighted with different colors in the first row of the table.
\end{example}

Evidently, the free intervals defined by an admissible triplet $(\lambda, \mu_i, \mu_{i+1})$  have the following properties: \begin{enumerate}
    \item If two free intervals intersect, then they coincide. 
    \item If $j_1 < j_2$, and $j_1, j_2 \in \operatorname{Fr}$,  then $\forall j_1 \leqslant j \leqslant j_2$, $j \in \operatorname{Fr}$. 
    \item If $j_1, j_2 \in \operatorname{Fr},$ such that $(\mu_{i+1})_{j_1} = -\frac{1}{2}$ and $(\mu_{i+1})_{j_2} = + \frac{1}{2}$, then $j_1 < j_2$
\end{enumerate}

We define the \textbf{positive part} of a free interval, denoted by $\operatorname{Fr}_+ \subset \operatorname{Fr}$, as 

$$\operatorname{Fr}_{+} := \Big\{j \in \operatorname{Fr} \hspace{1mm} | \hspace{1mm} (\mu_{i+1})_j = + \frac{1}{2}\Big\}.$$ 
In the same way we define the \textbf{negative part} of a free interval $\operatorname{Fr}_- \subset \operatorname{Fr}$: 
$$\operatorname{Fr}_{-} := \Big\{j \in \operatorname{Fr} \hspace{1mm} | \hspace{1mm} (\mu_{i+1})_j = - \frac{1}{2}\Big\}.$$
It is clear that $\operatorname{Fr} = \operatorname{Fr}_- \cup \operatorname{Fr}_+ $. We say that a free interval $\operatorname{Fr}$ is \textbf{non-degenerate} if both $\operatorname{Fr}_-, \operatorname{Fr}_+ \neq \emptyset$. Every non-degenerate free interval  is of the form $$\operatorname{Fr} = [\operatorname{Fr}_{min}, \operatorname{Fr}_{mid}] \cup [\operatorname{Fr}_{mid}+ 1, \operatorname{Fr}_{max}],$$
where $\operatorname{Fr}_{min}$$(\text{resp}. \hspace{1mm}\operatorname{Fr}_{max})$ is the minimal (resp. maximal) element in $\operatorname{Fr}$, and $\operatorname{Fr}_{mid} \in \operatorname{Fr}$, such that $$\operatorname{
Fr}_- = [\operatorname{Fr}_{min}, \operatorname{Fr}_{mid}], $$
and 
$$\operatorname{
Fr}_+ = [\operatorname{Fr}_{mid} + 1, \operatorname{Fr}_{max}].$$
For a degenerate free interval $\operatorname{Fr}_{mid}$ is not defined. 

Assume that $\operatorname{Fr}^{1}, \operatorname{Fr}^{2}, \ldots \operatorname{Fr}^{m}$ are the free intervals of the admissible triplet $(\lambda, \mu_i, \mu_{i+1})$.  For clarity we assume that $\operatorname{Fr}^{1}_{min} < \operatorname{Fr}^{2}_{min} < \ldots < \operatorname{Fr}^{m}_{min}$. For each free interval $\operatorname{Fr}^{k}, k \in \{1, 2, \ldots m\}$ we define the  composition of ascending crystal operators $\tilde{\textbf{e}}_{\operatorname{Fr}^{k}}$: 

\begin{gather}
\tilde{ \textbf{e}}_{\operatorname{Fr}^{k}}= \begin{cases} 1, \hspace{1mm} \text{if}  \hspace{1mm} \operatorname{Fr}^{k} \hspace{1mm} \text{is degenerate},   \\ \tilde e_{[\operatorname{Fr}^{k}_{min} +\operatorname{Fr}^{k}_{max} - \operatorname{Fr}^{k}_{mid} -1, \operatorname{Fr}_{max}^{k} - 1]} \ldots \tilde e_{[\operatorname{Fr}^{k}_{min}+1, \operatorname{Fr}^{k}_{mid}+1]}  \tilde e_{[\operatorname{Fr}^{k}_{min}, \operatorname{Fr}^{k}_{mid}]}, \hspace{1mm} \text{otherwise}.
 \end{cases}.  
\end{gather}

Respectively, we define the composition of descending crystal operators $\tilde{\textbf{f}}_{\operatorname{Fr}^{k}}$: 

\begin{gather}
\tilde{ \textbf{f}}_{\operatorname{Fr}^{k}}= \begin{cases} 1, \hspace{1mm} \text{if}  \hspace{1mm} \operatorname{Fr}^{k} \hspace{1mm} \text{is degenerate},   \\ \tilde f_{[\operatorname{Fr}^{k}_{min} +\operatorname{Fr}^{k}_{max} - \operatorname{Fr}^{k}_{mid} -1, \operatorname{Fr}_{max}^{k} - 1]} \ldots \tilde f_{[\operatorname{Fr}^{k}_{min}+1, \operatorname{Fr}^{k}_{mid}+1]}  \tilde f_{[\operatorname{Fr}^{k}_{min}, \operatorname{Fr}^{k}_{mid}]}, \hspace{1mm} \text{otherwise}.
 \end{cases}.  
\end{gather}

Here we used the short notation $\tilde e_{[a,b]} = \tilde e_a \tilde e_{a+1} \ldots \tilde e_{b}$ (resp. $\tilde f_{[a,b]} = \tilde f_a \tilde f_{a+1} \ldots \tilde f_{b})$, where $a < b$ are natural numbers from the set $\{1, 2, \ldots n\}$.  

Notice that \begin{gather}
    \tilde{\textbf{e}}_{\operatorname{Fr}^{m}} \ldots \tilde{\textbf{e}}_{\operatorname{Fr}^{1}} b_{\mu_{i+1}} \neq 0. 
\end{gather}
Indeed, this composition of ascending crystal operators acts on $b_{\mu_{i+1}}$ as follows. For each non-degenerate free interval $\operatorname{Fr}^{k}$  it carries the pluses of $b_{\mu_{i+1}}$ in the coordinates from $\operatorname{Fr}^{k}_+$  through the minuses of $b_{\mu_{i+1}}$ in the coordinates from $\operatorname{Fr}^{k}_{-}$. 
Since  $\lambda + \mu_{i} + \mu_{i+1} \in P_+$, it follows from Corollary \ref{cor_4.1} that  

\begin{gather}
    \tilde{\textbf{f}}_{\operatorname{Fr}^{m}} \ldots \tilde{\textbf{f}}_{\operatorname{Fr}^{1}} (b_{\lambda} \tens  b_{\mu_{i}}) \neq 0
\end{gather}
For any $\tilde f_j$ in the composition we have $j \in [\operatorname{Fr}^{k}_{min}, \operatorname{Fr}^{k}_{max}-1]$ for some $k$. 
By definition of a free interval, $\forall j \in [\operatorname{Fr}^{k}_{min}, \operatorname{Fr}^{k}_{max}-1]$, $\lambda_{j} = \lambda_{j+1}$. As $b_\lambda$ is a highest weight element, one can write 
\begin{gather}
 0 =   \lambda_j - \lambda_{j+1} =    \langle \lambda, \alpha_j^{\vee} \rangle = \varphi_j (b_\lambda) - \varepsilon_j(b_\lambda) = \varphi_j(b_\lambda)
\end{gather}
Therefore, the whole composition $\tilde{\textbf{f}}_{\operatorname{Fr}^{m}} \ldots \tilde{\textbf{f}}_{\operatorname{Fr}^{1}}$ acts on the second factor $b_{\mu_i}$:  
\begin{gather}
    \tilde{\textbf{f}}_{\operatorname{Fr}^{m}} \ldots \tilde{\textbf{f}}_{\operatorname{Fr}^{1}} (b_{\lambda} \tens  b_{\mu_{i}}) = b_\lambda \tens  \tilde{\textbf{f}}_{\operatorname{Fr}^{m}} \ldots \tilde{\textbf{f}}_{\operatorname{Fr}^{1}} b_{\mu_{i}}
\end{gather}
The composition of descending crystal operators $\tilde{\textbf{f}}_{\operatorname{Fr}^{m}} \ldots \tilde{\textbf{f}}_{\operatorname{Fr}^{1}}$ acts on  $b_{\mu_i}$ as follows. For each non-degenerate free interval $\operatorname{Fr}^{k}$ it carries the minuses of $b_{\mu_i}$ in the coordinates from $\operatorname{Fr}^{k}_+$ through the pluses of $b_{\mu_{i}}$ in the coordinates from $\operatorname{Fr}^{k}_{-}$. 

For brevity, we use the following notation. 
\begin{notation}
We denote by $\mu_{i}^{*}$, $\mu_{i+1}^{*}$ the weights of the elements $ \tilde{\textbf{f}}_{\operatorname{Fr}^{m}} \ldots \tilde{\textbf{f}}_{\operatorname{Fr}^{1}} b_{\mu_{i}}$,  $\tilde{\textbf{e}}_{\operatorname{Fr}^{m}} \ldots \tilde{\textbf{e}}_{\operatorname{Fr}^{1}} b_{\mu_{i+1}}$ respectively. Equivalently, one can write: 
\begin{gather} \label{mu_i_star}b_{\mu_i^{*}} := \tilde{\textbf{f}}_{\operatorname{Fr}^{m}} \ldots \tilde{\textbf{f}}_{\operatorname{Fr}^{1}} b_{\mu_{i}}, \\ \label{mu_i+1_star}b_{\mu_{i+1}^{*}} := \tilde{\textbf{e}}_{\operatorname{Fr}^{m}} \ldots \tilde{\textbf{e}}_{\operatorname{Fr}^{1}} b_{\mu_{i+1}}.\end{gather}
\end{notation}
\begin{example}
  \label{ex_6.2}
    Let $(\lambda, \mu_i, \mu_{i+1})$ be the admissible triplet from Example 
    \ref{ex_6.1}. Then, the triplet $(\lambda, \mu_i^{*}, \mu_{i+1}^{*})$ equals 

\center{
\begin{tabular}{ c | c c c c c c c c c c c c }
 № & 1 & \textbf{\textcolor{red}{2}} & \textbf{\textcolor{red}{3}} & \textbf{\textcolor{red}{4}} & 5 & 6 & \textbf{\textcolor{blue}{7}} & \textbf{\textcolor{blue}{8}} & 9 & \textbf{\textcolor{green}{10}} & \textbf{\textcolor{green}{11}} & \textbf{\textcolor{purple}{12}} \\ 
 \hline
  $\mu_{i+1}^{*}$ & $+$  & $+$  &  $-$ & $-$ & $-$ & $+$ & $-$ & $-$ & $-$ & $+$ & $-$ & $+$ \\  
 \hline
 $\mu_{i}^{*}$ & $+$ & $-$ & $+$ & $+$ & $-$ & $+$ & $+$ & $+$ & $-$ & $-$ & $+$ & $-$ \\
\hline
 $\lambda$ & $4$ & $4$&$4$ & $4$ &$4$ &$2$ &$2$ &$2$ &$2$ &$1$ & $1$ &$0$
\end{tabular}}
    
\end{example}

\begin{proposition}
\label{prop_6.4}
If $(\lambda, \mu_i, \mu_{i+1})$ is an admissible triplet and  $\mu_{i}^{*}$, $\mu_{i+1}^{*}$ are the weights defined above, then for any $j \in \{1, 2, \ldots n-1\}$
    \begin{gather} \label{eq_33}
        \langle \lambda + \mu_{i+1}^{*}, \alpha_j^{\vee}\rangle \geqslant 0.
    \end{gather}
\end{proposition}

\begin{proof}
As $j < n$, by definition of a co-root we have $$\langle \lambda + \mu_{i+1}^{*}, \alpha_j^{\vee}\rangle  =  \lambda_{j} - \lambda_{j+1} + (\mu_{i+1}^{*})_j - (\mu_{i+1}^{*})_{j+1}$$
If $\lambda_{j} > \lambda_{j+1}$, then $\lambda_j - \lambda_{j+1} \geqslant 1$ and since $(\mu_{i+1}^{*})_j,  (\mu_{i+1}^{*})_{j+1} \in \{ \pm \frac{1}{2}\}$ we conclude that $$\langle \lambda + \mu_{i+1}^{*}, \alpha_j^{\vee}\rangle  =  \lambda_{j} - \lambda_{j+1} + (\mu_{i+1}^{*})_j - (\mu_{i+1}^{*})_{j+1} \geqslant 0.$$
Assume that $\lambda_j = 
\lambda_{j+1}$. Firstly, if both $j$ and $j+1$ lie in some free interval $\operatorname{Fr}^{k}$, then $(\mu_{i+1}^{*})_j \geqslant (\mu_{i+1}^{*})_{j+1}$ as in any free interval all positive coordinates of $\mu_{i+1}^{*}$ go before negative coordinates and our inequality \ref{eq_33} holds.  Secondly, if both $j$ and $j+1$ do not belong to any free interval, then $(\mu_{i+1}^{*})_j = (\mu_{i+1})_j$ and $(\mu_{i+1}^{*})_{j+1} = (\mu_{i+1})_{j+1}$. As $\lambda + \mu_i \in P_+$, it follows that $$ \langle \lambda + \mu_{i+1}^{*}, \alpha_j^{\vee}\rangle=  \lambda_{j} - \lambda_{j+1} + (\mu_{i+1})_j - (\mu_{i+1})_{j+1} = \langle \lambda + \mu_{i+1}, \alpha_j^{\vee}\rangle  \geqslant 0.$$
Finally, if $j$ doesn't belong to any free interval and $j+1 \in \operatorname{Fr}^{k}$, then $(\mu_{i+1}^{*})_j = \frac{1}{2}$ and the inequality \ref{eq_33} holds. Similarly, if $j+1$ doesn't belong to any free interval and $j$ does, then $(\mu_{i+1}^{*})_{j+1} = - \frac{1}{2}$ and the inequality \ref{eq_33} still holds. 
\end{proof}

We introduce the following classification of admissible triplets $(\lambda, \mu_i, \mu_{i+1})$. 

\begin{definition}
\label{def_6.2}
    Let $(\lambda, \mu_i, \mu_{i+1})$ be an admissible triplet and $\mu_{i}^{*}$, $\mu_{i+1}^{*}$ be the corresponding weights defined above. 
    We say that the admissible triplet $(\lambda, \mu_i, \mu_{i+1})$ is of the \begin{itemize}
    \item \textbf{type $0$}, if $\langle \lambda + \mu_{i+1}^{*}, \alpha_n^{\vee}\rangle \geqslant 0$. 
    \item \textbf{type $1$}, if $\langle \lambda + \mu_{i+1}^{*}, \alpha_n^{\vee}\rangle < 0$ and $n, n-1$ lie in different free intervals for the triplet $(\lambda, \mu_i, \mu_{i+1})$.
    \item\textbf{type $2$}, if $\langle \lambda + \mu_{i+1}^{*}, \alpha_n^{\vee}\rangle < 0$ and $n, n-1$ lie in the same free interval for the triplet $(\lambda, \mu_i, \mu_{i+1})$.
     \end{itemize}
\end{definition}

It is worth explaining why any admissible triplet is either of the type $0$, $1$ or $2$. 
Assume that $\langle \lambda + \mu_{i+1}^{*}, \alpha_n^{\vee}\rangle < 0$, otherwise $(\lambda, \mu_i, \mu_{i+1})$ is of the type $0$ . If  $n$ doesn't belong to any free interval defined by the triplet $(\lambda, \mu_i, \mu_{i+1})$, then $$(\mu_{i})_n = (\mu_{i+1})_n = (\mu_{i}^{*})_n = (\mu_{i+1}^{*})_n = -\frac{1}{2},$$
because if $(\mu_i^{*})_n = (\mu_{i+1}^{*})_n = \frac{1}{2}$, then $\langle \lambda + \mu_{i+1}^{*}, \alpha_n^{\vee} \rangle = \langle \lambda, \alpha_n^{\vee}\rangle + (\mu_{i+1}^{*})_{n-1} + (\mu_{i+1}^{*})_n \geqslant 0,$ which is a contradiction.
If $n$ does belong to some free interval of the triplet $(\lambda, \mu_i, \mu_{i+1})$, then evidently $$(\mu_{i}^{*})_{n} + (\mu_{i+1}^{*})_n = 0.$$
The same can be said about  $n-1$. Hence, if $n$ or $n-1$ (or both) doesn't belong to any free interval then $$\langle \mu_{i}^{*} + \mu_{i+1}^{*}, \alpha_n^{\vee} \rangle \leqslant -1. $$
We know that $\mu_i + \mu_{i+1} = \mu_{i}^{*} + \mu_{i+1}^{*}$, so it follows  $$\langle \lambda + \mu_{i}^{*} + \mu_{i+1}^{*}, \alpha_n^{\vee} \rangle \geqslant 0$$ and consequently
$$\langle \lambda, \alpha_n^{\vee}\rangle \geqslant 1$$
But, then  $$\langle \lambda + \mu_{i+1}^{*}, \alpha^{\vee}_n \rangle  = \langle \lambda, \alpha_n^{\vee} \rangle + (\mu_{i+1}^{*})_{n-1} + (\mu_{i+1}^{*})_n \geqslant 0,$$
which is a contradiction. 

So, both $n-1$ and $n$ lie in some free intervals of the triplet $(\lambda, \mu_i, \mu_{i+1})$, so $(\lambda, \mu_i, \mu_{i+1})$ is either of the type $1$ or $2$.

We deal with admissible triplets of the type $0$ first.

\begin{corollary}
\label{cor_6.1}
Let  $ (\lambda, \mu_i, \mu_{i+1})$ be an admissible triplet.  Then $(\lambda, \mu_i, \mu_{i+1})$ is of the type $0$ iff $(\lambda, \mu_{i+1}^{*}, \mu_{i}^{*})$ is admissible.
\end{corollary}
\begin{proof}
 It is immediate from Definition \ref{def_6.2} and Proposition \ref{prop_6.4} that $(\lambda, \mu_i, \mu_{i+1})$ is of the type $0$ iff $\lambda + \mu_{i+1}^{*} \in P_+$. Since $\mu_{i}^{*} + \mu_{i+1}^{*} = \mu_{i} + \mu_{i+1}$, it follows that $\lambda + \mu_{i}^{*} + \mu_{i+1}^{*} \in P_+$. Hence, $(\lambda, \mu_{i}, \mu_{i+1})$ is of the type $0$ iff the triplet $(\lambda, \mu_{i+1}^{*}, \mu_{i}^{*})$ is admissible. 
\end{proof}
\begin{corollary}
\label{cor_6.2}
Let $b_{\mu_1} \tens\ldots \tens b_{\mu_i} \tens b_{\mu_{i+1}} \tens \ldots \tens b_{\mu_N}$ be an element in $\operatorname{Sing}(\Cr{B}_S^{\tens N})$,  such that the (admissible) triplet  $(\lambda, \mu_i, \mu_{i+1})$, where $\lambda := \sum_{j=1}^{i-1} \mu_j$, is of the type $0$. Then 
    \begin{gather}
        t_i^{res}(b_{\mu_1} \tens\ldots \tens b_{\mu_i} \tens b_{\mu_{i+1}} \tens \ldots b_{\mu_N})= b_{\mu_1} \tens\ldots \tens b_{\mu_{i+1}^{*}} \tens b_{\mu_{i}^{*}} \tens \ldots b_{\mu_N}
    \end{gather}      
\end{corollary}
\begin{proof}
It follows from Proposition \ref{prop_6.3} that $t_i^{res}$ changes only $i$-th and $i+1$-th factors, preserving the sum of their weights. It follows from the proof of Proposition \ref{prop_6.3}  that the $i+1$-th factor of $t_i^{res}(b_{\mu_1} \tens\ldots \tens b_{\mu_i} \tens b_{\mu_{i+1}} \tens \ldots b_{\mu_N})$ is equal to $\tilde {\textbf{f}}_2(b_{\mu_i})$, which is equal to  the $i+1$-th factor of $\sigma_{1,i,i+1}(b_{\mu_{1}} \tens  \ldots \tens b_{\mu_{i+1}})$. So, the $i+1$-th factor of $t_i^{res}(b_{\mu_1} \tens\ldots \tens b_{\mu_i} \tens b_{\mu_{i+1}} \tens \ldots b_{\mu_N})$ is equal to the $i+1$-th factor of the unique highest weight element of highest weight $\sum_{j=1}^{i+1} \mu_j = \lambda + \mu_i + \mu_{i+1}$ in $\Cr{B}_S \tens \Cr{B}_{\mu_1 + \ldots \mu_{i}}$. By Theorem \ref{th_5.1} this highest weight element can be written as $$\tilde e_{i_r} \ldots \tilde e_{i_1}\tilde{\textbf{e}}_{\operatorname{Fr}^{m}} \ldots \tilde{\textbf{e}}_{\operatorname{Fr}^{1}} b_{\mu_{i+1}} \tens\tilde f_{i_r}\ldots \tilde f_{i_1}\tilde{\textbf{f}}_{\operatorname{Fr}^{m}} \ldots \tilde{\textbf{f}}_{\operatorname{Fr}^{1}}(b_{\lambda} \tens b_{\mu_{i}}), $$ where $b_\lambda = b_{\mu_1} \tens \ldots \tens b_{\mu_{i-1}}$ and $\tilde e_{i_r} \ldots \tilde e_{i_1}\tilde{\textbf{e}}_{\operatorname{Fr}^{m}} \ldots \tilde{\textbf{e}}_{\operatorname{Fr}^{1}} b_{\mu_{i+1}}$ is a highest weight element in $\Cr{B}_S$ .  We can rewrite it, using definitions of $b_{\mu_{i}^{*}}$ and $b_{\mu_{i+1}^{*}}$ as 

$$\tilde e_{i_r} \ldots \tilde e_{i_1} b_{\mu_{i+1}^{*}} \tens\tilde f_{i_r}\ldots \tilde f_{i_1}(b_{\lambda} \tens b_{\mu_{i}^{*}}), $$
Since $(\lambda, \mu_{i}, \mu_{i+1})$ is of the type $0$, meaning that $\lambda + \mu_{i+1}^{*} \in P_+$, and  $\tilde e_{i_r} \ldots \tilde e_{i_1} b_{\mu_{i+1}^{*}} \neq 0$, we obtain from Proposition \ref{prop_4.4} that 
$$\tilde e_{i_r} \ldots \tilde e_{i_1} b_{\mu_{i+1}^{*}} \tens\tilde f_{i_r}\ldots \tilde f_{i_1}(b_{\lambda} \tens b_{\mu_{i}^{*}}) = \tilde e_{i_r} \ldots \tilde e_{i_1} b_{\mu_{i+1}^{*}} \tens\tilde f_{i_r}\ldots \tilde f_{i_1}b_{\lambda} \tens b_{\mu_{i}^{*}}.$$
We conclude that $i+1$-th factor of $t_i^{res}(b_{\mu_1} \tens\ldots \tens b_{\mu_i} \tens b_{\mu_{i+1}} \tens \ldots b_{\mu_N})$ is equal to $b_{\mu_{i}^{*}}$. Since $t_i^{res}$ preserves the weight of the element, it follows that the weight of the $i$-th factor of the element  $t_i^{res}(b_{\mu_1} \tens\ldots \tens b_{\mu_i} \tens b_{\mu_{i+1}} \tens \ldots b_{\mu_N})$ equals $\mu_i + \mu_{i+1} - \mu_{i}^{*} = \mu_{i+1}^{*}$.  Therefore, 
\begin{gather}
        t_i^{res}(b_{\mu_1} \tens\ldots \tens b_{\mu_i} \tens b_{\mu_{i+1}} \tens \ldots b_{\mu_N})= b_{\mu_1} \tens\ldots \tens b_{\mu_{i+1}^{*}} \tens b_{\mu_{i}^{*}} \tens \ldots b_{\mu_N}
    \end{gather}  and we are done. 
\end{proof}

Now we want to calculate $t_i^{res}$ on the elements $b_{\mu_1} \tens\ldots \tens b_{\mu_i} \tens b_{\mu_{i+1}} \tens \ldots \tens b_{\mu_N} \in \operatorname{Sing}(\Cr{B}_S^{\tens N})$ for which the (admissible) triplet  $(\lambda, \mu_i, \mu_{i+1})$, where $\lambda := \sum_{j=1}^{i-1} \mu_j$, is of the type $1$ or $2$.  Our whole argument for the type $0$  triplets was built around the fact that $\lambda + \mu_{i+1}^{*} \in P_+$. For triplets of type $1$ and $2$ that is not the case, since $\langle \lambda + \mu_{i+1}^{*} , \alpha_n^{\vee}\rangle < 0$ by definition of the types. So, we would like to act on $b_{\mu_{i+1}^{*}}$ by some certain extra composition (depending on the type) of ascending crystal operators, so that the resulting weight of the image of $b_{\mu_{i+1}^{*}}$ plus $\lambda$ lies in $P_+$ and the respective composition of descending operators acts on the second factor in $b_\lambda \tens b_{\mu_i^{*}}$. For each type we will explicitly construct these compositions of ascending/descending crystal operators. 

\begin{remark}
\label{remark_6.1}
By definition of a co-root:$$\langle \lambda + \mu_{i+1}^{*} , \alpha_n^{\vee}\rangle  = \lambda_{n-1} + \lambda_{n}  + (\mu_{i+1}^{*})_{n-1} + (\mu_{i+1}^{*})_n.$$ Since $\lambda \in P_+$, $$\langle  \lambda, \alpha_{n}^{\vee} \rangle= \lambda_{n-1} + \lambda_n \geqslant 0.$$
So, it follows that $\langle \lambda + \mu_{i+1}^{*} , \alpha_n^{\vee}\rangle  < 0 $ iff $\lambda_{n-1} + \lambda_{n}  = 0$ and $(\mu_{i+1}^{*})_{n-1} =  (\mu_{i+1}^{*})_n = - \frac{1}{2}$. 

Then, it is clear that for the type $1$ triplets, $\{n\}$ must be a (degenerate) free interval with only negative part, as $(\mu_{i})_n = (\mu_{i+1}^{*})_{n} = -\frac{1}{2}.$ The free interval  containing $n-1$ must have non-empty negative part, because $(\mu_{i+1}^{*})_{n-1} = - \frac{1}{2}$.

For the type $2$ triplets $\lambda_n = \lambda_{n-1} = 0$, because $n$ and $n-1$ belong to the same free interval. Moreover, the negative part of the last free interval (which contains $n-1$ and $n$) must have cardinality at least $2$ because $(\mu_{i+1}^{*})_{n} = (\mu_{i+1}^{*})_{n-1} = - \frac{1}{2}$, and the positive part must have no more than $1$ element, because $\langle \lambda + \mu_i, \alpha^{\vee}_n\rangle = \lambda_{n-1} + \lambda_n + (\mu_{i})_{n-1} + (\mu_i)_n =  (\mu_{i})_{n-1} + (\mu_i)_n \geqslant 0 $.

\end{remark}

We start with type $1$ triplets. For the reference we give a concrete example of a type $1$ triplet.

\begin{example}
\label{ex_6.3}
Let $(\lambda, \mu_{i},\mu_{i+1})$ be an admissible triplet, defined by the table below.

\begin{center}
\begin{tabular}{ c | c c c c c c c c c c c c }
 № & \textbf{\textcolor{red}{1}} & \textbf{\textcolor{red}{2}} & \textbf{\textcolor{red}{3}} & \textbf{\textcolor{red}{4}} & 5 & \textbf{\textcolor{blue}{6}} & \textbf{\textcolor{blue}{7}} & \textbf{\textcolor{blue}{8}} & \textbf{\textcolor{blue}{9}} & \textbf{\textcolor{blue}{10}} & \textbf{\textcolor{blue}{11}} & \textbf{\textcolor{green}{12}} \\ 
 \hline
  $\mu_{i+1}$ & $-$  & $+$  &  $+$ & $+$ & $-$ & $-$ & $-$ & $-$ & $-$ & $+$ & $+$ & $-$ \\  
 \hline
 $\mu_{i}$ & $+$ & $-$ & $-$ & $-$ & $-$ & $+$ & $+$ & $+$ & $+$ & $-$ & $-$ & $+$ \\
\hline
 $\lambda$ & $2$ & $2$&$2$ & $2$ &$2$ &$1$ &$1$ &$1$ &$1$ &$1$ & $1$ &$-1$
\end{tabular}
\end{center}
\vspace{2mm} 
The respective weights  ${\mu}_{i}^{*}$, ${\mu}_{i+1}^{*}$ are written in the table below: 
\vspace{1mm} 
\begin{center}
\begin{tabular}{ c | c c c c c c c c c c c c }
 № & \textbf{\textcolor{red}{1}} & \textbf{\textcolor{red}{2}} & \textbf{\textcolor{red}{3}} & \textbf{\textcolor{red}{4}} & 5 & \textbf{\textcolor{blue}{6}} & \textbf{\textcolor{blue}{7}} & \textbf{\textcolor{blue}{8}} & \textbf{\textcolor{blue}{9}} & \textbf{\textcolor{blue}{10}} & \textbf{\textcolor{blue}{11}} & \textbf{\textcolor{green}{12}} \\ 
 \hline
  $\mu_{i+1}^{*}$ & $+$  & $+$  &  $+$ & $-$ & $-$ & $+$ & $+$ & $-$ & $-$ & $-$ & $-$ & $-$ \\  
 \hline
 $\mu_{i}^{*}$ & $-$ & $-$ & $-$ & $+$ & $-$ & $-$ & $-$ & $+$ & $+$ & $+$ & $+$ & $+$ \\
\hline
 $\lambda$ & $2$ & $2$&$2$ & $2$ &$2$ &$1$ &$1$ &$1$ &$1$ &$1$ & $1$ &$-1$
\end{tabular}
\end{center}
\vspace{2mm}
It is evident that $(\mu_{i+1}^{*})_{12} + (\mu_{i+1}^{*})_{11} + \lambda_{12} + \lambda_{11}  =  -1 < 0.$
The triplet $(\lambda, \mu_i, \mu_{i+1})$ is of the type $1$, since $\operatorname{Fr}^{3}= \{12\}$ and $\operatorname{Fr}^{2}= \{6,7,8,9, 10, 11\}$. As was observed in Remark\ref{remark_6.1}: $\operatorname{Fr}^{3} = \operatorname{Fr}^{3}_-$ and $\operatorname{Fr}^{2}_- \neq \emptyset$.
\end{example}

Let $(\lambda, \mu_i, \mu_{i+1})$ be an admissible triplet of type $1$. We denote by $\operatorname{Fr}^{1}, \ldots, \operatorname{Fr}^{m}$the free intervals defined by the triplet $(\lambda, \mu_i, \mu_{i+1})$. Just as before, we assume that $\operatorname{Fr}^{1}_{min} < \ldots < \operatorname{Fr}^{m}_{min}$. Since the triplet is of type $1$ we have $\operatorname{Fr}^{m} = \operatorname{Fr}_{-}^{m} = \{n\}$ and $n-1 \in \operatorname{Fr}^{m-1}$. Set $a := |\operatorname{Fr}^{m-1}_{-}| > 0$. Define the following composition of crystal ascending operators: \begin{gather} \tilde{\textbf{e}}_{\operatorname{type}_1} = \begin{cases}(\tilde e_{n-a}\ldots  \tilde e_{n-2})\tilde e_n, \hspace{1mm} \text{if} \hspace{1mm} a> 1, \\  \tilde e_n, \hspace{1mm} \text{otherwise}.\end{cases}\end{gather} Respectively, we define \begin{gather} \tilde{\textbf{f}}_{\operatorname{type}_1} = \begin{cases}(\tilde f_{n-a}\ldots  \tilde f_{n-2})\tilde f_n, \hspace{1mm} \text{if} \hspace{1mm} a> 1, \\  \tilde f_n, \hspace{1mm} \text{otherwise}.\end{cases}\end{gather}

Notice that $\tilde{\textbf{e}}_{\operatorname{type}_1}b_{\mu_{i+1}^{*}} \neq 0$. Indeed, $\tilde{\textbf{e}}_{\operatorname{type}_1}$ acts on $b_{\mu_{i+1}^{*}}$ by changing the minuses in the coordinates $n$ and $n-a$  to pluses. Similarly, $\tilde{\textbf{f}}_{\operatorname{type}_1}$ acts on $b_{\mu_{i}^{*}}$ by changing the pluses in the coordinates $n$ and $n-a$ to minuses.  
Since $$\operatorname{wt}(b_\lambda \tens b_{\mu_{i}^{*}}) + \operatorname{wt}(b_{\mu_{i+1}^{*}}) = \lambda + \mu_i^{*} + \mu_{i+1}^{*} = \lambda + \mu_i + \mu_{i+1} \in P_+,$$
and $\tilde{\textbf{e}}_{\operatorname{type}_1}b_{\mu_{i+1}^{*}} \neq 0$, Corollary \ref{cor_4.1} implies that $\tilde{\textbf{f}}_{\operatorname{type}_1}(b_\lambda \tens b_{\mu_{i}^{*}}) \neq 0$. 
Then, it follows that \begin{gather} \tilde{\textbf{f}}_{\operatorname{type}_1}(b_\lambda \tens b_{\mu_{i}^{*}}) = b_\lambda \tens \tilde{\textbf{f}}_{\operatorname{type}_1} b_{\mu_{i}^{*}}, \end{gather} because $$\varphi_n(b_\lambda) = \langle \lambda, \alpha_n^{\vee}\rangle = \lambda_{n-1} + \lambda_n = 0$$ and $$\varphi_j(b_\lambda) = \langle \lambda, \alpha_j^{\vee}\rangle= \lambda_{j} - \lambda_{j+1} = 0, \forall j \in \{n-a, \ldots, n-2\}. $$
In the latter equation we used that for $j \in \{n-a, \ldots, n-2\}$ one has $\{j, j+1\} \subset \{n-a, \ldots n-1\} \subset \operatorname{Fr}^{m-1}$ and hence by the definition of the free interval $\lambda_j = \lambda_{j+1}$. 

For brevity, we are going to use the following notation.
\begin{notation}
We denote by $\mu_{i}^{\diamond }$, $\mu_{i+1}^{\diamond}$ the weights of the elements $ \tilde{\textbf{f}}_{\operatorname{type}_1}b_{\mu_{i}^{*}}$,  $\tilde{\textbf{e}}_{\operatorname{type}_1}b_{\mu_{i+1}^{*}}$ respectively. Equivalently, one can write
\begin{gather}
b_{\mu_{i+1}^{\diamond}}  := \tilde{\textbf{e}}_{\operatorname{type}_1}b_{\mu_{i+1}^{*}},  \\ b_{\mu_{i}^{\diamond}} := \tilde{\textbf{f}}_{\operatorname{type}_1}b_{\mu_{i}^{*}} .
\end{gather}
\end{notation}
Clearly, $\mu_{i}^{*} + \mu_{i+1}^{*} = \mu_{i}^{\diamond} + \mu_{i+1}^{\diamond}$. 

\begin{example}
\label{ex_6.4}
Let $(\lambda, \mu_{i},  \mu_{i+1})$ be the admissible triplet of the type $1$ from Example \ref{ex_6.3}. Then  ${\mu}_{i+1}^{\diamond}$, $\mu_i^{\diamond}$ equal 

\center{
\begin{tabular}{ c | c c c c c c c c c c c c }
 № & \textbf{\textcolor{red}{1}} & \textbf{\textcolor{red}{2}} & \textbf{\textcolor{red}{3}} & \textbf{\textcolor{red}{4}} & 5 & \textbf{\textcolor{blue}{6}} & \textbf{\textcolor{blue}{7}} & \textbf{\textcolor{blue}{8}} & \textbf{\textcolor{blue}{9}} & \textbf{\textcolor{blue}{10}} & \textbf{\textcolor{blue}{11}} & \textbf{\textcolor{green}{12}} \\ 
 \hline
  ${\mu}_{i+1}^{\diamond}$ & $+$  & $+$  &  $+$ & $-$ & $-$ & $+$ & $+$ & $+$ & $-$ & $-$ & $-$ & $+$ \\  
 \hline
 ${\mu}_{i}^{\diamond}$ & $-$ & $-$ & $-$ & $+$ & $-$ & $-$ & $-$ & $-$ & $+$ & $+$ & $+$ & $-$ \\
\hline
 $\lambda$ & $2$ & $2$&$2$ & $2$ &$2$ &$1$ &$1$ &$1$ &$1$ &$1$ & $1$ &$-1$
\end{tabular}
}
\end{example}

\begin{proposition}
\label{prop_6.5}

If $(\lambda, \mu_i, \mu_{i+1})$ is an admissible triplet of the type $1$, then $\lambda + \mu_{i+1}^{\diamond} \in P_+$.  Equivalently, as $\mu_{i} + \mu_{i+1} = \mu_{i}^{\diamond} + \mu_{i+1}^{\diamond}$,  if $(\lambda, \mu_{i}, \mu_{i+1})$ is an admissible triplet of the type $1$, then the triplet $(\lambda, \mu_{i+1}^{\diamond}, \mu_{i}^{\diamond})$ is admissible. 
\end{proposition}
\begin{proof}
Let $(\lambda, \mu_i, \mu_{i+1})$ be an admissible triplet of the type $1$. Proposition \ref{prop_6.4} implies that $$\langle \lambda + \mu_{i+1}^{*}, \alpha_j^{\vee} \rangle  = \lambda_j - \lambda_{j+1} + (\mu_{i+1}^{*})_j -(\mu_{i+1}^{*})_{j+1}  \geqslant 0, \forall j \in \{1,2, \ldots n-1\}.$$
By definition, the weight $\mu_{i+1}^{\diamond}$ is obtained from  $\mu_{i+1}^{*}$ by changing $(\mu_{i+1}^{*})_n = (\mu_{i+1}^{*})_{n-a} = -\frac{1}{2}$ to  $+ \frac{1}{2}$. Here $a = |\operatorname{Fr}^{m-1}_-| > 0$ (see Remark \ref{remark_6.1}), where $\operatorname{Fr}^{m-1}$ is the free interval of the triplet $(\lambda, \mu_i, \mu_{i+1})$,  such that $n-1 \in \operatorname{Fr}^{m-1}$.  Therefore,
$$\langle \lambda + \mu_{i+1}^{\diamond}, \alpha_j^{\vee} \rangle \geqslant \langle \lambda + \mu_{i+1}^{*}, \alpha_j^{\vee} \rangle \geqslant 0, \hspace{1mm} \forall j \in \{1,2, \ldots n-1\},$$ unless $j = n-1$ or $n-a-1$.  By definition of the type $1$ triplet $n$ and $n-1$ belong to different free intervals, so $\lambda_{n-1} - \lambda_{n} \geqslant 1$. Thus for $j = n-1$: $$\langle \lambda + \mu_{i+1}^{\diamond}, \alpha_{n-1}^{\vee} \rangle = \lambda _{n-1}- \lambda_{n} + (\mu_{i+1}^{\diamond})_{n-1}- (\mu_{i+1}^{\diamond})_{n} \geqslant 1 +(\mu_{i+1}^{\diamond})_{n-1}- (\mu_{i+1}^{\diamond})_{n} \geqslant 0. $$ Now, assume that $j = n-a -1 > 0$. If $(\mu_{i+1}^{\diamond})_{n-a-1} = + \frac{1}{2}$, 
$$\langle \lambda + \mu_{i+1}^{\diamond}, \alpha_{n-a-1}^{\vee} \rangle = \lambda _{n-a-1} - \lambda_{n-a} + (\mu_{i+1}^{\diamond})_{n-a-1} - (\mu_{i+1}^{\diamond})_{n-a}  = \lambda_{n-a-1} - \lambda_{n-a} \geqslant 0.$$
Otherwise $(\mu_{i+1}^{\diamond})_{n-a-1} = - \frac{1}{2}$ implying that $\lambda_{n-a-1} - \lambda_{n-a}  \geqslant 1$, because if $\lambda_{n-a-1} = \lambda_{n-a}$,  then $ (\mu_{i+1}^{\diamond})_{n-a-1} =(\mu^{\diamond}_{i})_{n-a-1} = - \frac{1}{2}$ and $\langle \lambda + \mu_i^{\diamond} +  \mu_{i+1}^{\diamond}, \alpha_{n-a-1}^{\vee} \rangle < 0$, which is a contradiction.  So, $$\langle \lambda + \mu_{i+1}^{\diamond}, \alpha_{n-a-1}^{\vee} \rangle = \lambda _{n-a-1} - \lambda_{n-a} + (\mu_{i+1}^{\diamond})_{n-a-1} - (\mu_{i+1}^{\diamond})_{n-a} \geqslant 1 + (\mu_{i+1}^{\diamond})_{n-a-1} - (\mu_{i+1}^{\diamond})_{n-a} \geqslant 0.$$

It is left to check that $\langle \lambda + \mu_{i+1}^{\diamond}, \alpha_n^{\vee} \rangle \geqslant 0 $.  As $\lambda \in P_+$,  we have $\langle \lambda + \mu_{i+1}^{\diamond}, \alpha_n^{\vee} \rangle = \lambda _{n-1}+ \lambda_{n} + (\mu_{i+1}^{\diamond})_{n-1}+ (\mu_{i+1}^{\diamond})_{n} = \lambda _{n-1}+ \lambda_{n} + (\mu_{i+1}^{\diamond})_{n-1}+ \frac{1}{2} \geqslant 0$. 
\end{proof}

Similar to the case of the type $0$ triplets we get the following corollary. 

\begin{corollary}
    \label{cor_6.3}
Let $b_{\mu_1} \tens\ldots \tens b_{\mu_i} \tens b_{\mu_{i+1}} \tens \ldots \tens b_{\mu_N}$ be an element in $\operatorname{Sing}(\Cr{B}_S^{\tens N})$,  such that the (admissible) triplet  $(\lambda, \mu_i, \mu_{i+1})$, where $\lambda := \sum_{k=1}^{i-1} \mu_k$, is of the type $1$. Then 
    \begin{gather}
        t_i^{res}(b_{\mu_1} \tens\ldots \tens b_{\mu_i} \tens b_{\mu_{i+1}} \tens \ldots b_{\mu_N})= b_{\mu_1} \tens\ldots \tens b_{\mu_{i+1}^{\diamond}} \tens b_{\mu_{i}^{\diamond}} \tens \ldots b_{\mu_N}
    \end{gather}      
\end{corollary}
\begin{proof}
The proof of this corollary is almost identical to the proof of Corollary \ref{cor_6.2}. It suffices to find the $i+1$-th factor of the unique highest weight element of highest weight  $\sum_{j=1}^{i+1} \mu_j = \lambda + \mu_i + \mu_{i+1}$ in $\Cr{B}_S \tens \Cr{B}_{\mu_1 + \ldots \mu_{i}} \subset \Cr{B}_S \tens \Cr{B}_{S}^{\tens i}$. By Theorem \ref{th_5.1} this highest weight element can be written as $$\tilde e_{i_r} \ldots \tilde e_{i_1}\tilde{\textbf{e}}_{\operatorname{type}_1}\tilde{\textbf{e}}_{\operatorname{Fr}^{m}} \ldots \tilde{\textbf{e}}_{\operatorname{Fr}^{1}} b_{\mu_{i+1}} \tens\tilde f_{i_r}\ldots \tilde f_{i_1} \tilde{\textbf{f}}_{\operatorname{type}_1}\tilde{\textbf{f}}_{\operatorname{Fr}^{m}} \ldots \tilde{\textbf{f}}_{\operatorname{Fr}^{1}}(b_{\lambda} \tens b_{\mu_{i}}), $$ where $b_\lambda = b_{\mu_1} \tens \ldots \tens b_{\mu_{i-1}}$ and $\tilde e_{i_r} \ldots \tilde e_{i_1} \tilde{\textbf{e}}_{\operatorname{type}_1}\tilde{\textbf{e}}_{\operatorname{Fr}^{m}} \ldots \tilde{\textbf{e}}_{\operatorname{Fr}^{1}} b_{\mu_{i+1}}$ is a highest weight element in $\Cr{B}_S$ .  We can rewrite it, using definitions of $b_{\mu_{i}^{\diamond}}$ and $b_{\mu_{i+1}^{\diamond}}$ as 
$$\tilde e_{i_r} \ldots \tilde e_{i_1} b_{\mu_{i+1}^{\diamond}} \tens\tilde f_{i_r}\ldots \tilde f_{i_1}(b_{\lambda} \tens b_{\mu_{i}^{\diamond}}), $$
Since the triplet $(\lambda, \mu_{i}, \mu_{i+1})$ is of the type $1$, Proposition \ref{prop_6.5} implies that $\lambda + \mu_{i+1}^{\diamond} \in P_+$. Also it is clear that  $\tilde e_{i_r} \ldots \tilde e_{i_1} b_{\mu_{i+1}^{\diamond}} \neq 0$, so Proposition \ref{prop_4.4} implies that 
$$\tilde e_{i_r} \ldots \tilde e_{i_1} b_{\mu_{i+1}^{\diamond}} \tens\tilde f_{i_r}\ldots \tilde f_{i_1}(b_{\lambda} \tens b_{\mu_{i}^{\diamond}}) = \tilde e_{i_r} \ldots \tilde e_{i_1} b_{\mu_{i+1}^{\diamond}} \tens\tilde f_{i_r}\ldots \tilde f_{i_1}b_{\lambda} \tens b_{\mu_{i}^{\diamond}}.$$
We conclude that $i+1$-th factor of $t_i^{res}(b_{\mu_1} \tens\ldots \tens b_{\mu_i} \tens b_{\mu_{i+1}} \tens \ldots b_{\mu_N})$ is equal to $b_{\mu_{i}^{\diamond}}$. Since $t_i^{res}$ preserves the weight of the element and is local, it follows that the weight of the $i$-th factor of the element  $t_i^{res}(b_{\mu_1} \tens\ldots \tens b_{\mu_i} \tens b_{\mu_{i+1}} \tens \ldots b_{\mu_N})$ equals $\mu_i + \mu_{i+1} - \mu_{i}^{\diamond} = \mu_{i+1}^{\diamond}$.  Therefore, 
\begin{gather}
        t_i^{res}(b_{\mu_1} \tens\ldots \tens b_{\mu_i} \tens b_{\mu_{i+1}} \tens \ldots b_{\mu_N})= b_{\mu_1} \tens\ldots \tens b_{\mu_{i+1}^{\diamond}} \tens b_{\mu_{i}^{\diamond}} \tens \ldots b_{\mu_N}
    \end{gather}  and we are done. 
\end{proof}

It remains to deal with the triplets of the type $2$. Let $(\lambda, \mu_i, \mu_{i+1})$ be an admissible triplet of the type $2$. Then as was observed in Remark \ref{remark_6.1} $\lambda_n = \lambda_{n-1} = 0$.  Just as before we denote by $\operatorname{Fr}^{m}$ the free interval containing $\{n-1, n\}$. Let $a = |\operatorname{Fr}^{m}_{-}|$ and $b = |\operatorname{Fr}_{+}^{m}|$. such that $(\mu_{i+1})_j = -\frac{1}{2}$. Remark \ref{remark_6.1} implies that $a \geqslant 2$ and $b \leqslant 1$.  Then, last $a+b$ coordinates (from $n-(a+b - 1)$ to $n$) of the weight $\mu_{i+1}^{*}$ look as follows: 
$$\Big(\overset{n-a-b + 1}{\pm \frac{1}{2}}, -\frac{1}{2}, \ldots, -\overset{n}{\frac{1}{2}}\Big).$$
The sign of the $n-a-b+1$-th coordinate of $\mu_{i+1}^{*}$ depends on $b$. If $b = 1$, then the sign is $+$, otherwise $-$. Respectively, last $a+b$ coordinates from $n-(a+b - 1)$ to $n$ of the weight $\mu_{i}^{*}$ look as follows:
$$\Big(\overset{n-a-b + 1}{\mp \frac{1}{2}}, +\frac{1}{2}, \ldots, +\overset{n}{\frac{1}{2}}\Big).$$
The sign of the $n-a-b+1$-th coordinate of $\mu_i^{*}$ depends on $b$. If $b = 1$, then the sign is $-$, otherwise $+$. 

We denote by $\tilde{\textbf{e}}_{\operatorname{type}_2}$ the composition of ascending operators that acts on $b_{\mu_{i+1}}^{*}$ so that it only changes the last $a+b$ coordinates of the weight to 
$$\Big(\overset{n-a-b + 1}{+ \frac{1}{2}}, +\frac{1}{2}, \ldots, \pm\overset{n}{\frac{1}{2}}\Big).$$
The sign of the last coordinate  depends on parity of $a$. If $a$ is even then the sign is $+$, if not then it is a $-$. 

Respectively, we define $\tilde{\textbf{f}}_{\operatorname{type}_2}$. It consists of the descending operators $\tilde f_j$ with the same indexes as in the composition $\tilde{\textbf{e}}_{\operatorname{type}_2}$.  Clearly, $\tilde{\textbf{f}}_{\operatorname{type}_2}$ acts on $b_{\mu_{i}^{*}}$ so that it only changes the last $a+b$ coordinates of the weight to 
$$\Big(\overset{n-a-b + 1}{- \frac{1}{2}}, -\frac{1}{2}, \ldots, \mp\overset{n}{\frac{1}{2}}\Big).$$
The sign of the last coordinate depends on parity of $a$. If $a$ is even then the sign is $-$, if not then it is a $+$.  

As $\tilde{\textbf{e}}_{\operatorname{type}_2} b_{\mu_{i+1}^{*}} \neq 0$, Corollary \ref{cor_4.1} implies that $\tilde{\textbf{f}}_{\operatorname{type}_2}(b_\lambda \tens b_{\mu_{i}^{*}}) \neq 0$.  Clearly,  $\{n-a-b+1, \ldots, n \} \subset \operatorname{Fr}^{m}$, so $\lambda_{n-a-b+1} = \ldots = \lambda_n = 0$ and $$\varphi_j(b_\lambda) = \langle \lambda, \alpha_j^{\vee}\rangle =  0, \hspace{1mm} \forall j \in \{n-a-b+1, \ldots, n\}.$$
It is also clear all the indexes of crystal operators in compositions $\tilde{\textbf{e}}_{\operatorname{type}_2}$, $\tilde{\textbf{f}}_{\operatorname{type}_2}$ lie in the set $\{n-a-b+1, \ldots, n \} \subset \operatorname{Fr}^{m}$. Since $\tilde{\textbf{f}}_{\operatorname{type}_2}(b_\lambda \tens b_{\mu_{i}^{*}}) \neq 0$ we conclude  $$\tilde{\textbf{f}}_{\operatorname{type}_2}(b_\lambda \tens b_{\mu_{i}^{*}}) = b_\lambda \tens\tilde{\textbf{f}}_{\operatorname{type}_2} b_{\mu_{i}^{*}}$$
For brevity, we introduce the following notation.
\begin{notation}
We denote by $\mu_{i}^{\diamond \diamond}$, $\mu_{i+1}^{\diamond \diamond}$ the weights of the elements $ \tilde{\textbf{f}}_{\operatorname{type}_2}b_{\mu_{i}^{*}}$,  $\tilde{\textbf{e}}_{\operatorname{type}_2}b_{\mu_{i+1}^{*}}$ respectively. Equivalently, one can write
\begin{gather}
b_{\mu_{i+1}^{\diamond \diamond}}  := \tilde{\textbf{e}}_{\operatorname{type}_2}b_{\mu_{i+1}^{*}},  \\ b_{\mu_{i}^{\diamond \diamond}} := \tilde{\textbf{f}}_{\operatorname{type}_2}b_{\mu_{i}^{*}} .
\end{gather}
\end{notation}
Evidently, $\mu_{i}^{*} + \mu_{i+1}^{*} = \mu_{i}^{\diamond \diamond} + \mu_{i+1}^{\diamond \diamond}$. For the reference, we provide a concrete example of a type $2$ triplet and calculate the weights $\mu_i^{\diamond\diamond}, \mu_{i+1}^{\diamond \diamond}$ for it. 

\begin{example}
\label{ex_6.5}
 Let $(\lambda, \mu_{i},\mu_{i+1})$ be an admissible triplet, defined by the table below.
\begin{center}
\begin{tabular}{ c | c c c c c c c c c c c c }
 № & \textbf{\textcolor{red}{1}} & \textbf{\textcolor{red}{2}} & 3 & \textbf{\textcolor{blue}{4}} & \textbf{\textcolor{blue}{5}} & \textbf{\textcolor{blue}{6}} & 7 & 8 & \textbf{\textcolor{green}{9}} & \textbf{\textcolor{green}{10}} & \textbf{\textcolor{green}{11}} & \textbf{\textcolor{green}{12}} \\ 
 \hline
  $\mu_{i+1}$ & $+$  & $+$  &  $-$ & $-$ & $+$ & $+$ & $-$ & $-$ & $-$ & $-$ & $-$ & $+$ \\  
 \hline
 $\mu_{i}$ & $-$ & $-$ & $-$ & $+$ & $-$ & $-$ & $-$ & $-$ & $+$ & $+$ & $+$ & $-$ \\
\hline
 $\lambda$ & $2$ & $2$&$2$ & $1$ &$1$ &$1$ &$1$ &$1$ &$0$ &$0$ & $0$ &$0$
\end{tabular}
\end{center}
\vspace{2mm}
The respective weights  ${\mu}_{i}^{*}$, ${\mu}_{i+1}^{*}$ are given in the table below: 
\vspace{1mm} 
\begin{center}
 \begin{tabular}{ c | c c c c c c c c c c c c }
 № & \textbf{\textcolor{red}{1}} & \textbf{\textcolor{red}{2}} & 3 & \textbf{\textcolor{blue}{4}} & \textbf{\textcolor{blue}{5}} & \textbf{\textcolor{blue}{6}} & 7 & 8 & \textbf{\textcolor{green}{9}} & \textbf{\textcolor{green}{10}} & \textbf{\textcolor{green}{11}} & \textbf{\textcolor{green}{12}} \\ 
 \hline
  $\mu_{i+1}^{*}$ & $+$  & $+$  &  $-$ & $+$ & $-$ & $-$ & $-$ & $-$ & $+$ & $-$ & $-$ & $-$ \\  
 \hline
 $\mu_{i}^{*}$ & $-$ & $-$ & $-$ & $-$ & $+$ & $+$ & $-$ & $-$ & $-$ & $+$ & $+$ & $+$ \\
\hline
 $\lambda$ & $2$ & $2$&$2$ & $1$ &$1$ &$1$ &$1$ &$1$ &$0$ &$0$ & $0$ &$0$
\end{tabular}
\end{center}
\vspace{2mm}
It is evident that $(\mu_{i+1}^{*})_{12} + (\mu_{i+1}^{*})_{11} + \lambda_{12} + \lambda_{11}  =  -1 < 0.$
The triplet $(\lambda, \mu_i, \mu_{i+1})$ is of the type $2$, since $\operatorname{Fr}^{3}= \{9,10,11,12\}$. As was observed in Remark \ref{remark_6.1} $|\operatorname{Fr}^{3}_-|= 3 \geqslant 2$  and $|\operatorname{Fr}^{3}_+| = 1 \leqslant 1$. 

Then, the weights $\mu_{i}^{\diamond \diamond}, \mu_{i+1}^{\diamond\diamond}$ equal

\begin{center}
\begin{tabular}{ c | c c c c c c c c c c c c }
 № & \textbf{\textcolor{red}{1}} & \textbf{\textcolor{red}{2}} & 3 & \textbf{\textcolor{blue}{4}} & \textbf{\textcolor{blue}{5}} & \textbf{\textcolor{blue}{6}} & 7 & 8 & \textbf{\textcolor{green}{9}} & \textbf{\textcolor{green}{10}} & \textbf{\textcolor{green}{11}} & \textbf{\textcolor{green}{12}} \\ 
 \hline
  $\mu_{i+1}^{\diamond\diamond}$ & $+$  & $+$  &  $-$ & $+$ & $-$ & $-$ & $-$ & $-$ & $+$ & $+$ & $+$ & $-$ \\  
 \hline
 $\mu_{i}^{\diamond\diamond}$ & $-$ & $-$ & $-$ & $-$ & $+$ & $+$ & $-$ & $-$ & $-$ & $-$ & $-$ & $+$ \\
\hline
 $\lambda$ & $2$ & $2$&$2$ & $1$ &$1$ &$1$ &$1$ &$1$ &$0$ &$0$ & $0$ &$0$
\end{tabular}
\end{center}
\end{example}

Just like for type $1$ triplets we have the following proposition for type $2$ triplets. 

\begin{proposition}
\label{prop_6.6}
If $(\lambda, \mu_i, \mu_{i+1})$ is an admissible triplet of the type $2$, then $\lambda + \mu_{i+1}^{\diamond \diamond} \in P_+$.  Equivalently, as $\mu_{i} + \mu_{i+1} = \mu_{i}^{\diamond \diamond} + \mu_{i+1}^{\diamond \diamond}$,  if $(\lambda, \mu_{i}, \mu_{i+1})$ is an admissible triplet of the type $2$, then the triplet $(\lambda, \mu_{i+1}^{\diamond \diamond}, \mu_{i}^{\diamond \diamond})$ is admissible. 
\end{proposition}

\begin{corollary}
   \label{cor_6.4}
   Let $b_{\mu_1} \tens\ldots \tens b_{\mu_i} \tens b_{\mu_{i+1}} \tens \ldots \tens b_{\mu_N}$ be an element in $\operatorname{Sing}(\Cr{B}_S^{\tens N})$,  such that the (admissible) triplet  $(\lambda, \mu_i, \mu_{i+1})$, where $\lambda := \sum_{k=1}^{i-1} \mu_k$, is of the type $2$. Then 
    \begin{gather}
        t_i^{res}(b_{\mu_1} \tens\ldots \tens b_{\mu_i} \tens b_{\mu_{i+1}} \tens \ldots b_{\mu_N})= b_{\mu_1} \tens\ldots \tens b_{\mu_{i+1}^{\diamond \diamond}} \tens b_{\mu_{i}^{\diamond \diamond}} \tens \ldots b_{\mu_N}
    \end{gather}      
\end{corollary}
\begin{proof}
    The proof is identical to the proof of the Corollary \ref{cor_6.3} up to replacing type $1$ with type $2$ everywhere and Proposition \ref{prop_6.5} with Proposition \ref{prop_6.6}.
\end{proof}

\subsection{Action of the cactus group on the short semi-standard Young tables}
\label{final_final}

In the work of Michael Chmutov, Max Glick and Pablo Pylyavskii \cite{Chmutov} the action of  the cactus group on the set of semi-standard Young tableaux was defined. By definition of that action, the generators $\mathbf{t}_i, \hspace{1mm} i = 1, 2, \ldots, N-1$ of the cactus group $C_N$ act on semi-standard Young tableaux as Bender-Knuth involutions, denoted by $\tau_i^{bk}$. We recall how these involutions work. 

Let $Y$ be the semi-standard Young tableau filled with the numbers $1, 2, \ldots, N$. Denote by $Y_{[i,i+1]}$ the skew semi-standard Young tableau formed by all the cells with the numbers $i$, $i+1$ in them. Any row of $Y_{[i,i+1]}$ from left to right consists of 

\vspace{2mm}
\begin{enumerate}
    \item $a$ cells with the number $i$, with the cell with the number $i+1$ directly underneath,
    \item $b$ cells with the number $i$, with no cells underneath,
    \item $c$ cells with the number $i+1$, with no cells underneath,
    \item $d$ cells with the number $i+1$, with a cell with number $i$ above. 
\end{enumerate}
\vspace{2mm}

The cells of the second and third types we are going to call \textbf{free}. Let $Y^{'}_{[i,  i+1]}$ be the tableau obtained from $Y_{[i, i+1]}$ by switching the values $b,c$ in each row. 
By definition, $\tau_i^{bk}(Y)$ is the result of replacing $Y_{[i, i+1]}$ with $Y_{[i, i+1]}^{'}$ inside $Y$. Clearly, the result is still a semi-standard Young tableau. Also, it is evident that the Bender-Knuth involutions are local, meaning that $\tau_i^{bk}$ acts only on the cells with the values $i, i+1$ in them. We also provide an example of the action of a  Bender-Knuth involution on a semi-standard Young tableau, see Example \ref{ex_7.1}.

\begin{example}
\label{ex_7.1}
This example illustrates the action of Bender-Knuth involution $\tau_i^{bk}$ on the semi-standard Young tableau. Here in each row the free cells are colored in white and all the other cells are colored in gray:  

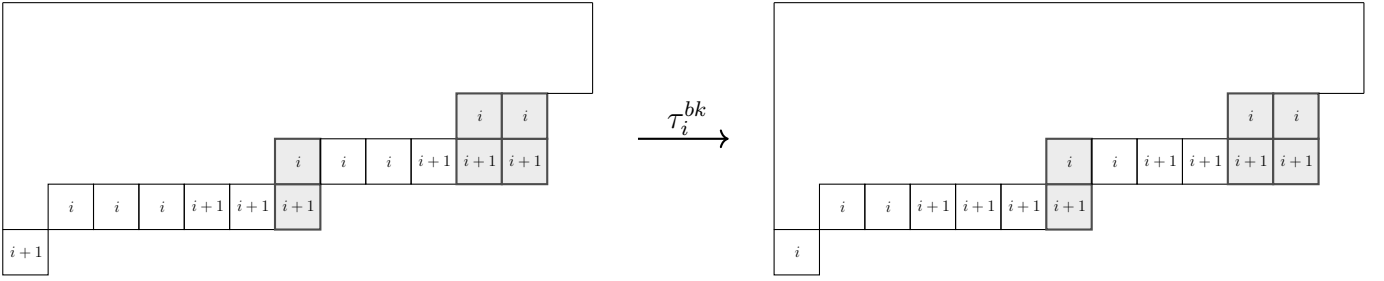
\begin{center}
\begin{tikzpicture}[scale = 0.6]
\draw (2,0) rectangle (3,1);
\draw (3, 0.5) circle (0pt)  node[scale = 0.5, anchor=east]{$i+1$};

\draw (3,1) rectangle (4,2);
\draw (3.7, 1.5) circle (0pt)  node[scale = 0.5, anchor=east]{$i$};
\draw (4,1) rectangle (5,2);
\draw (4.7, 1.5) circle (0pt)  node[scale = 0.5, anchor=east]{$i$};
\draw (5,1) rectangle (6,2);
\draw (5.7, 1.5) circle (0pt)  node[scale = 0.5, anchor=east]{$i$};
\draw (6,1) rectangle (7,2);
\draw (7, 1.5) circle (0pt)  node[scale = 0.5, anchor=east]{$i+1$};
\draw (7,1) rectangle (8,2);
\draw (8, 1.5) circle (0pt)  node[scale = 0.5, anchor=east]{$i+1$};
\filldraw[color=black!70, fill=gray!15,  thick] (8,1) rectangle (9,2);
\draw (9, 1.5) circle (0pt)  node[scale = 0.5, anchor=east]{$i+1$};
\filldraw[color=black!70, fill=gray!15,  thick] (8,2) rectangle (9,3);
\draw (8.7, 2.5) circle (0pt)  node[scale = 0.5, anchor=east]{$i$};
\draw (9,2) rectangle (10,3);
\draw (9.7, 2.5) circle (0pt)  node[scale = 0.5, anchor=east]{$i$};
\draw (10,2) rectangle (11,3);
\draw (10.7, 2.5) circle (0pt)  node[scale = 0.5, anchor=east]{$i$};
\draw (11,2) rectangle (12,3);
\draw (12, 2.5) circle (0pt)  node[scale = 0.5, anchor=east]{$i+1$};
\filldraw[color=black!70, fill=gray!15,  thick] (12,2) rectangle (13,3);
\draw (13, 2.5) circle (0pt)  node[scale = 0.5, anchor=east]{$i+1$};
\filldraw[color=black!70, fill=gray!15,  thick] (13,2) rectangle (14,3);
\draw (14, 2.5) circle (0pt)  node[scale = 0.5, anchor=east]{$i+1$};
\filldraw[color=black!70, fill=gray!15,  thick] (12,3) rectangle (13,4);
\draw (12.7, 3.5) circle (0pt)  node[scale = 0.5, anchor=east]{$i$};
\filldraw[color=black!70, fill=gray!15,  thick] (13,3) rectangle (14,4);
\draw (13.7, 3.5) circle (0pt)  node[scale = 0.5, anchor=east]{$i$};
\draw (2,1) -- (2,6);
\draw (2,6) -- (15,6);
\draw (15,6) -- (15,4);
\draw (15,4) -- (14,4);

\draw[thick,->] (16,3) -- (18,3);
\filldraw[black] (16.40,3.45) circle (0pt) node[anchor=west]{$\tau_i^{bk}$};


\draw (19,0) rectangle (20,1);
\draw (19.7, 0.5) circle (0pt)  node[scale = 0.5, anchor=east]{$i$};

\draw (20,1) rectangle (21,2);
\draw (20.7, 1.5) circle (0pt)  node[scale = 0.5, anchor=east]{$i$};
\draw (21,1) rectangle (22,2);
\draw (21.7, 1.5) circle (0pt)  node[scale = 0.5, anchor=east]{$i$};
\draw (22,1) rectangle (23,2);
\draw (23, 1.5) circle (0pt)  node[scale = 0.5, anchor=east]{$i+1$};
\draw (23,1) rectangle (24,2);
\draw (24, 1.5) circle (0pt)  node[scale = 0.5, anchor=east]{$i+1$};
\draw (24,1) rectangle (25,2);
\draw (25, 1.5) circle (0pt)  node[scale = 0.5, anchor=east]{$i+1$};
\filldraw[color=black!70, fill=gray!15,  thick] (25,1) rectangle (26,2);
\draw (26, 1.5) circle (0pt)  node[scale = 0.5, anchor=east]{$i+1$};
\filldraw[color=black!70, fill=gray!15,  thick] (25,2) rectangle (26,3);
\draw (25.7, 2.5) circle (0pt)  node[scale = 0.5, anchor=east]{$i$};
\draw (26,2) rectangle (27,3);
\draw (26.7, 2.5) circle (0pt)  node[scale = 0.5, anchor=east]{$i$};
\draw (27,2) rectangle (28,3);
\draw (28, 2.5) circle (0pt)  node[scale = 0.5, anchor=east]{$i+1$};
\draw (28,2) rectangle (29,3);
\draw (29, 2.5) circle (0pt)  node[scale = 0.5, anchor=east]{$i+1$};
\filldraw[color=black!70, fill=gray!15,  thick] (29,2) rectangle (30,3);
\draw (30, 2.5) circle (0pt)  node[scale = 0.5, anchor=east]{$i+1$};
\filldraw[color=black!70, fill=gray!15,  thick] (30,2) rectangle (31,3);
\draw (31, 2.5) circle (0pt)  node[scale = 0.5, anchor=east]{$i+1$};
\filldraw[color=black!70, fill=gray!15,  thick] (29,3) rectangle (30,4);
\draw (29.7, 3.5) circle (0pt)  node[scale = 0.5, anchor=east]{$i$};
\filldraw[color=black!70, fill=gray!15,  thick](30,3) rectangle (31,4);
\draw (30.7, 3.5) circle (0pt)  node[scale = 0.5, anchor=east]{$i$};
\draw (19,1) -- (19,6);
\draw (19,6) -- (32,6);
\draw (32,6) -- (32,4);
\draw (32,4) -- (31,4);
\end{tikzpicture}
\end{center}

\end{example}

In the previous chapter we defined an action of the cactus group $C_N$ on the set $\Sing(\Cr{B}_S^{\otimes N})$ and explicitly computed it for generators $\mathbf t_i \in C_N$. In chapter \ref{cell_tables_&_sssyt} we observed that for each $\lambda \in \Delta^{N}$ the set $\operatorname{Sing}_\lambda(\Cr{B}_S^{\tens N})$ is in bijection with the set of short semi-standard Young tableaux of the shape $\nu = \mathcal{F}_N(D_\lambda^{N}) \in \operatorname{SYD}(N,n)$, denoted by $\operatorname{SSSYT}(\nu, N)$. Clearly, $\Sing_\lambda(\Cr{B}_S^{\tens N})$ is an invariant subset of $\Sing(\Cr{B}_S^{\tens N})$ under the action of $C_N$. Therefore, for any shape $\nu \in \operatorname{SYD}(N,n)$ we can define the action of the cactus group $C_N$ on the set  $\operatorname{SSSYT}(\nu, N)$. Let us discuss it in detail. 
 
Denote the natural bijection between sets $\Sing(\Cr{B}_S^{\otimes N})$ and $T^N$, given by the formula below, by $\Cr{H}_{\lambda,N}: T^N_{\lambda} \rightarrow \Sing_{\lambda}(\Cr{B}_S^{\otimes N})$:
\begin{gather}
    (\mu_1, \mu_2, \ldots \mu_N) \mapsto b_{\mu_1} \tens b_{\mu_2} \tens \ldots \tens b_{\mu_N}.
\end{gather}
For any $\nu \in \operatorname{SYD}(N,n)$ we define the group homomorphism
$$\tilde\psi_{\nu}: C_N \rightarrow \Aut(\operatorname{SSSYT}(\nu, N))$$ 
by the following formula:
\begin{gather}
    \mathbf{t}_i \mapsto \Cr{Y}_{\lambda} \circ \Cr{I}_{\lambda} \circ\Cr{H}_{\lambda,N}^{-1} \circ t_i^{res}
    \circ\Cr{H}_{\lambda,N} \circ \Cr{I}_{\lambda}^{-1}\circ \Cr{Y}_{\lambda}^{-1},
\end{gather}
where $\lambda = \Cr{K}_N^{-1}(\Cr{F}_N^{-1}(\nu))$ and $\Cr{K}_N, \Cr{F}_N, \Cr{Y}_\lambda$, $\Cr{I}_\lambda$ are the maps defined in chapter \ref{cell_tables_&_sssyt}. Denote the image of the generator $\mathbf t_i$ under this map by $\tau_i := \tilde\psi_\nu(\mathbf{t}_i)$.

The homomorphism $\tilde \psi_\nu$ defines the action of the cactus group on the set $\operatorname{SSSYT}(\nu,N)$ of short semi-standard Young tableaux of the shape $\nu$. Clearly, $\operatorname{SSSYT}(\nu,N) \subset \operatorname{SSYT}(\nu, N)$, where $\operatorname{SSYT}(\nu, N)$ is a set of semi-standard Young tableaux of shape $\nu$, filled with the numbers $1, 2,\ldots, N$. 

Hence, it is interesting to see if the action of the cactus group $C_N$ given by $\tilde \psi_\nu$ on the subset $\operatorname{SSSYT}(\nu,N)$ correlates with the action of $C_N$ on the bigger set $\operatorname{SSYT}(\nu, N)$ defined by Chmutov in \cite{Chmutov}. Notice that the action of $C_N$ on $\operatorname{SSYT}(\nu, N)$ defined in \cite{Chmutov} cannot be restricted to the subset $\operatorname{SSSYT}(\nu,N)$, because Bender-Knuth involutions on the set of semi-standard Young tableaux may map a short tableau to a non-short one. So, the action of the cactus group on the subset $\operatorname{SSSYT}(\nu, N) \subset \operatorname{SSYT}(\nu, N)$ given by $\tilde \psi_\nu$ definitely differs from the one discussed in \cite{Chmutov}. Nevertheless, we will soon see that the map $\tau_i \in \Aut(\operatorname{SSSYT}(\nu, N))$ is similar to the Bender-Knuth involution $\tau_i^{bk} \in  \Aut(\operatorname{SSYT}(\nu, N))$.  

The following theorems explicitly describe the action of the Bender-Knuth generators $\mathbf{t}_i \in C_N$ on the set $\operatorname{SSSYT}(\nu, N)$ given by $\tilde \psi_\nu(\mathbf t_i) = \tau_i \in \Aut(\operatorname{SSSYT}(\nu, N))$. 

\begin{theorem}
\label{th_7.1}
Let $x = (\nu^{(1)}, \nu^{(2)}, \ldots, \nu^{(i-1)},\nu^{(i)}, \nu^{(i+1)}, \ldots \nu^{(N)}) \in \operatorname{SSSYT}(\nu, N)$. Consider $$\tau^{bk}_i(x) = (\nu^{(1)}, \nu^{(2)}, \ldots, \nu^{(i-1)},\tilde\nu^{(i)}, \nu^{(i+1)}, \ldots \nu^{(N)}).$$Suppose that $\tau^{bk}_i(x) \in \operatorname{SSSYT}(\nu) \Leftrightarrow \tilde\nu^{(i)} \in \operatorname{SYD}(i, n)$, then \begin{gather}\tau_i(x) = \tau^{bk}_i(x).
\end{gather} 
\end{theorem}
\begin{proof}
Assume that $n$ is even (the case of odd $n$ is similar).  Take $x \in \operatorname{SSSYT}(\nu, N)$. We say that a cell of the tableau $x = (\nu^{(1)}, \nu^{(2)}, \ldots, \nu^{(i-1)},\nu^{(i)}, \nu^{(i+1)}, \ldots \nu^{(N)}) $ is labeled $i$  iff it belongs to the skew diagram $\nu^{(i)} - \nu^{(i-1)}$.  

Let $\lambda = \Cr{K}_N^{-1}(\Cr{F}_N^{-1}(\nu))$. Consider $(\Cr{H}_{\lambda,N} \circ \Cr{I}_{\lambda}^{-1}\circ \Cr{Y}_{\lambda}^{-1})(x) = b_{\mu_1} \tens b_{\mu_2} \tens \ldots \tens b_{\mu_N} \in \operatorname{Sing}_{\lambda}(\Cr{B}_S^{\tens N})$. 
By definition of the maps $\Cr{I}_\lambda, \Cr{Y}_\lambda$, tableau $x$ has a cell labeled $i$ in the $m$-th column iff $(\mu_i)_{n-m} = -\frac{1}{2}$.   The skew tableau $\nu^{(i+1)} - \nu^{(i-1)}$ consists of cells filled with numbers of $x$ labeled $i$ or $i+1$ and is uniquely determined by the shape $\nu^{(i-1)}$ and weights $\mu_i$, $\mu_{i+1}$. Set $\gamma = \Cr{K}_{i-1}^{-1}(\Cr{F}_{i-1}^{-1}(\nu^{(i-1)})) = \sum_{k=1}^{i-1} \mu_k$. 

Free intervals, defined by an admissible triplet  $(\gamma, \mu_i, \mu_{i+1})$, correspond to maximal (with respect to inclusion) subsets of consecutive columns where the skew tableau $\nu^{(i+1)} - \nu^{(i-1)}$ contains exactly one cell (labeled either $i$ or $i+1$).  Consider an element $b_{\mu_1} \tens \ldots \tens b_{\mu_{i+1}^*} \tens b_{\mu_{i}^{*}} \tens \ldots \tens b_{\mu_N} \in \Cr{B}_S^{\tens N}$, where $b_{\mu_i^{*}}$ and $b_{\mu_{i+1}^{*}}$ are defined by formulas \ref{mu_i_star},  \ref{mu_i+1_star}. Even though the element $b_{\mu_1} \tens \ldots \tens b_{\mu_{i+1}^*} \tens b_{\mu_{i}^{*}} \tens \ldots \tens b_{\mu_N}$ might not lie in $\operatorname{Sing}_\lambda(\Cr{B}_S^{\tens N})$, one can construct a semi-standard Young tableau of the shape $\nu$ corresponding to this element using the same algorithm.  Namely, we draw, starting with $i=1$, a cell labeled $i$ in the $k$-th column iff $(\mu_i)_{n-k} = -\frac{1}{2}$. The fact that the result of these operations is a semi-standard Young tableau of the shape $\nu$ follows from Proposition \ref{prop_6.4} and the fact that $\mu_i^{*} + \mu_{i+1}^{*} = \mu_i + \mu_{i+1}$. Observe that the semi-standard tableau $y$ obtained from the element $b_{\mu_1} \tens \ldots \tens b_{\mu_{i+1}^*} \tens b_{\mu_{i}^{*}} \tens \ldots \tens b_{\mu_N}$ equals $\tau_i^{bk}(x)$. It follows from the definitions of $b_{\mu_i^{*}}$ and $b_{\mu_{i+1}^{*}}$, see \ref{mu_i_star},  \ref{mu_i+1_star}. 
The semi-standard Young tableau $y = \tau_i^{bk}(x)$ obtained from $b_{\mu_1} \tens \ldots \tens b_{\mu_{i+1}^*} \tens b_{\mu_{i}^{*}} \tens \ldots \tens b_{\mu_N}$ is short iff the triplet $(\gamma, \mu_{i+1}^{*}, \mu_{i}^{*})$ is admissible, i.e. the triplet $(\gamma, \mu_i, \mu_{i+1})$ is of type $0$. So, if $y =\tau_i^{bk}(x) \in \operatorname{SSSYT}(\nu, N)$, then the triplet $(\gamma, \mu_{i}, \mu_{i+1})$ is of the type $0$ and by Corollary \ref{cor_6.2} $$t_i^{res }(b_{\mu_1} \tens b_{\mu_2} \tens \ldots \tens b_{\mu_N}) = b_{\mu_1} \tens \ldots \tens b_{\mu_{i+1}^*} \tens b_{\mu_{i}^{*}} \tens \ldots \tens b_{\mu_N}.$$
Hence,  $\tau_i(x) = y = \tau_i^{bk}(x)$ and we are done.
\end{proof}

We have just observed that $\tau_i$ acts as the Bender-Knuth involution $\tau_i^{bk}$, whenever the Bender-Knuth involution  preserves the shortness of the tableau. The next step is to understand when the latter fails. The next proposition presents complete classification of such cases. 

\begin{proposition}
    \label{prop_7.1}
    Let $x = (\nu^{(1)}, \nu^{(2)}, \ldots, \nu^{(i-1)},\nu^{(i)}, \nu^{(i+1)}, \ldots, \nu^{(N)} = \nu) \in \operatorname{SSSYT}(\nu, N)$. If $\tau_i^{bk}(x) \notin \operatorname{SSSYT}(\nu, N)$, then there are $2$ possibilities for tableau $x$: 

\vspace{1mm}

\textbf{Possibility 1}. Tableau $x$ satisfies the following conditions: 
\begin{itemize}
    \item There is only one cell  in the last row of the skew Young tableau $\nu^{(i+1)} - \nu^{(i-1)}$. Moreover, this cell is free and has the number $i+1$ in it. 
    \item There is a  free cell with the number $i+1$ in it in the penultimate row of the skew tableau $\nu^{(i+1)} - \nu^{(i-1)}$.
    \item In the first two columns of the tableau $\nu^{(i+1)}$ there are exactly $i+1$ cells.
\end{itemize}
Tableau $x$ that satisfies the conditions above is said to be of  \textbf{type $1$}. 

\vspace{1mm}

\textbf{\textit{Possibility 2}}. Tableau $x$ satisfies the following conditions: 
\begin{itemize}
\item In the last row of the skew tableau $\nu^{(i+1)} - \nu^{(i-1)}$ there are at least two  free cells with the number $i+1$ and fewer than two cells with the number $i$.
\item In the first two columns of the tableau $\nu^{(i+1)}$ there are exactly $i+1$ cells.
\end{itemize}
Tableau $x$ that satisfies the conditions above is said to be of \textbf{type $2$}. 
\end{proposition}
\begin{proof}
Let $\lambda = \Cr{K}_N^{-1}(\Cr{F}_N^{-1}(\nu))$. Set $$(\Cr{I}_{\lambda}^{-1}\circ \Cr{Y}_{\lambda}^{-1})(x) = (\mu_1, \ldots \mu_i, \mu_{i+1}, \ldots, \mu_N) \in T_\lambda^{N}$$
and $\gamma  = \sum_{k=1}^{i-1} \mu_k$. Using definitions of the maps $\Cr{I}_{\lambda}, \Cr{Y}_\lambda$ one can show that $x$ is of type $1$ (resp. $2$) iff the admissible triplet $(\gamma, \mu_i, \mu_{i+1})$ is of type $1$ (resp. $2$), see Definition \ref{def_6.2} and Remark \ref{remark_6.1}. Also, it was observed in the proof of Theorem \ref{th_7.1}  that $\tau_i^{bk}(x) \in \operatorname{SSSYT}(\nu, N)$ iff the triplet $(\gamma, \mu_i, \mu_{i+1})$ is of the type $0$.   Hence, the statement of the proposition is equivalent to the statement that any admissible triplet $(\gamma, \mu_i, \mu_{i+1})$ which is not type $0$ is either type $1$ or type $2$. The latter was proved right after Definition \ref{def_6.2}.
\end{proof}

The last step is to  compute the map $\tau_i$ on the tableaux of types $1$ and $2$. Hence, the following two theorems.

\begin{theorem}
\label{th_7.2}
 If  $x \in \operatorname{SSSYT}(\nu, N)$ is a tableau of \textbf{type }$1$, see Proposition \ref{prop_7.1}, then $\tau_i$ operates on $x$ the following way: 
    \begin{itemize}
        \item All rows of the skew tableau $\nu^{(i+1)} - \nu^{(i-1)}$, aside from the last two rows convert in the same way they would under the Bender-Knuth involution.  
        \item The last row of  $\nu^{(i+1)} - \nu^{(i-1)}$ stays the same (it only has one cell with $i+1$ in it)
        \item The number inside the first free cell with the number $i+1$ in the penultimate row is being replaced with $i$ and then the usual Bender-Knuth involution is applied.   
    \end{itemize}
\end{theorem}
\begin{proof}
Set $\lambda = \Cr{K}_N^{-1}(\Cr{F}_N^{-1}(\nu))$. As was stated in the proof of Proposition \ref{prop_7.1}, the tableau $x \in \operatorname{SSSYT}(\nu, N)$ is of type $1$ iff the triplet $(\gamma, \mu_i, \mu_{i+1})$ is of type $1$, where $(\mu_1, \ldots, \mu_i, \mu_{i+1}, \ldots \mu_N) = (\Cr{I}_{\lambda}^{-1}\circ \Cr{Y}_{\lambda}^{-1}) (x)$ and $\gamma = \sum_{k=1}^{i-1} \mu_k$.

Then, the theorem is equivalent to Corollary \ref{cor_6.3} via the correspondence between elements of $\operatorname{Sing}_{\lambda}(\Cr{B}_S^{\tens N})$ and $\operatorname{SSSYT}(\nu, N)$ given by the bijection $\Cr{H}_{\lambda,N} \circ \Cr{I}_{\lambda}^{-1}\circ \Cr{Y}_{\lambda}^{-1} $.
\end{proof}

The next example illustrates the action of $\tau_i$ on the type $1$ tableaux described in Theorem \ref{th_7.2}

\begin{example}
\label{ex_7.2}
Suppose the first two columns of the tableau below (left) contain exactly $i+1$ cells. Then the automorphism ${\tau_i}$ maps it to the tableau on the right.
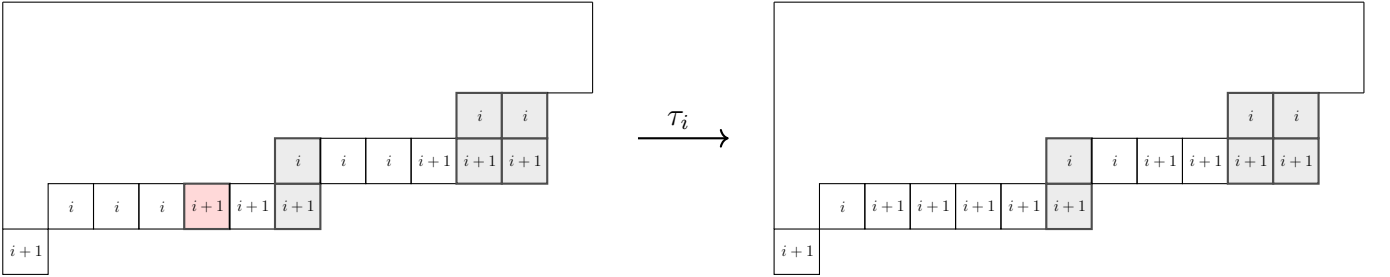
\vspace{2mm}

\begin{center}
\begin{tikzpicture}[scale = 0.6]
\draw (2,0) rectangle (3,1);
\draw (3, 0.5) circle (0pt)  node[scale = 0.5, anchor=east]{$i+1$};

\draw (3,1) rectangle (4,2);
\draw (3.7, 1.5) circle (0pt)  node[scale = 0.5, anchor=east]{$i$};
\draw (4,1) rectangle (5,2);
\draw (4.7, 1.5) circle (0pt)  node[scale = 0.5, anchor=east]{$i$};
\draw (5,1) rectangle (6,2);
\draw (5.7, 1.5) circle (0pt)  node[scale = 0.5, anchor=east]{$i$};
\filldraw[color=black!70, fill=red!15,  thick] (6,1) rectangle (7,2);
\draw (7, 1.5) circle (0pt)  node[scale = 0.5, anchor=east]{$i+1$};
\draw (7,1) rectangle (8,2);
\draw (8, 1.5) circle (0pt)  node[scale = 0.5, anchor=east]{$i+1$};
\filldraw[color=black!70, fill=gray!15,  thick] (8,1) rectangle (9,2);
\draw (9, 1.5) circle (0pt)  node[scale = 0.5, anchor=east]{$i+1$};
\filldraw[color=black!70, fill=gray!15,  thick] (8,2) rectangle (9,3);
\draw (8.7, 2.5) circle (0pt)  node[scale = 0.5, anchor=east]{$i$};
\draw (9,2) rectangle (10,3);
\draw (9.7, 2.5) circle (0pt)  node[scale = 0.5, anchor=east]{$i$};
\draw (10,2) rectangle (11,3);
\draw (10.7, 2.5) circle (0pt)  node[scale = 0.5, anchor=east]{$i$};
\draw (11,2) rectangle (12,3);
\draw (12, 2.5) circle (0pt)  node[scale = 0.5, anchor=east]{$i+1$};
\filldraw[color=black!70, fill=gray!15,  thick] (12,2) rectangle (13,3);
\draw (13, 2.5) circle (0pt)  node[scale = 0.5, anchor=east]{$i+1$};
\filldraw[color=black!70, fill=gray!15,  thick] (13,2) rectangle (14,3);
\draw (14, 2.5) circle (0pt)  node[scale = 0.5, anchor=east]{$i+1$};
\filldraw[color=black!70, fill=gray!15,  thick] (12,3) rectangle (13,4);
\draw (12.7, 3.5) circle (0pt)  node[scale = 0.5, anchor=east]{$i$};
\filldraw[color=black!70, fill=gray!15,  thick] (13,3) rectangle (14,4);
\draw (13.7, 3.5) circle (0pt)  node[scale = 0.5, anchor=east]{$i$};
\draw (2,1) -- (2,6);
\draw (2,6) -- (15,6);
\draw (15,6) -- (15,4);
\draw (15,4) -- (14,4);

\draw[thick,->] (16,3) -- (18,3);
\filldraw[black] (16.40,3.40) circle (0pt) node[anchor=west]{$\tau_i$};


\draw (19,0) rectangle (20,1);
\draw (20, 0.5) circle (0pt)  node[scale = 0.5, anchor=east]{$i+1$};

\draw (20,1) rectangle (21,2);
\draw (20.7, 1.5) circle (0pt)  node[scale = 0.5, anchor=east]{$i$};
\draw (21,1) rectangle (22,2);
\draw (22, 1.5) circle (0pt)  node[scale = 0.5, anchor=east]{$i+1$};
\draw (22,1) rectangle (23,2);
\draw (23, 1.5) circle (0pt)  node[scale = 0.5, anchor=east]{$i+1$};
\draw (23,1) rectangle (24,2);
\draw (24, 1.5) circle (0pt)  node[scale = 0.5, anchor=east]{$i+1$};
\draw (24,1) rectangle (25,2);
\draw (25, 1.5) circle (0pt)  node[scale = 0.5, anchor=east]{$i+1$};
\filldraw[color=black!70, fill=gray!15,  thick] (25,1) rectangle (26,2);
\draw (26, 1.5) circle (0pt)  node[scale = 0.5, anchor=east]{$i+1$};
\filldraw[color=black!70, fill=gray!15,  thick] (25,2) rectangle (26,3);
\draw (25.7, 2.5) circle (0pt)  node[scale = 0.5, anchor=east]{$i$};
\draw (26,2) rectangle (27,3);
\draw (26.7, 2.5) circle (0pt)  node[scale = 0.5, anchor=east]{$i$};
\draw (27,2) rectangle (28,3);
\draw (28, 2.5) circle (0pt)  node[scale = 0.5, anchor=east]{$i+1$};
\draw (28,2) rectangle (29,3);
\draw (29, 2.5) circle (0pt)  node[scale = 0.5, anchor=east]{$i+1$};
\filldraw[color=black!70, fill=gray!15,  thick] (29,2) rectangle (30,3);
\draw (30, 2.5) circle (0pt)  node[scale = 0.5, anchor=east]{$i+1$};
\filldraw[color=black!70, fill=gray!15,  thick] (30,2) rectangle (31,3);
\draw (31, 2.5) circle (0pt)  node[scale = 0.5, anchor=east]{$i+1$};
\filldraw[color=black!70, fill=gray!15,  thick] (29,3) rectangle (30,4);
\draw (29.7, 3.5) circle (0pt)  node[scale = 0.5, anchor=east]{$i$};
\filldraw[color=black!70, fill=gray!15,  thick](30,3) rectangle (31,4);
\draw (30.7, 3.5) circle (0pt)  node[scale = 0.5, anchor=east]{$i$};
\draw (19,1) -- (19,6);
\draw (19,6) -- (32,6);
\draw (32,6) -- (32,4);
\draw (32,4) -- (31,4);
\end{tikzpicture}
\end{center}

Non-free cells are highlighted in gray, and the first free cell in the penultimate row of the left skew Young tableau with the number $i+1$ inside is highlighted in red. This example illustrates the automorphism ${\tau_i}$ on tableaux satisfying the conditions of \textbf{Possibility} $1$ in Proposition \ref{prop_7.1}. 

\end{example}

\begin{theorem}
\label{th_7.3}
If  $x \in \operatorname{SSSYT}(\nu, N)$ is a tableau of \textbf{type} $2$,  see Proposition \ref{prop_7.1}, then $\tau_i$ operates on $x$ the following way: 
\begin{itemize}
    \item All rows of the skew tableau $\nu^{(i+1)} - \nu^{(i-1)}$, aside from the last one, convert in the same way they would under the Bender-Knuth involution.
    \item The first cell in the last row of the skew tableau $\nu^{(i+1)} - \nu^{(i-1)}$ (it is always free) changes its value from $i$ to $i+1$ and vice versa if the number of the  free cells in the last row of $\nu^{(i+1)} - \nu^{(i-1)}$ is odd.
\end{itemize}
\end{theorem}
\begin{proof}
Set $\lambda = \Cr{K}_N^{-1}(\Cr{F}_N^{-1}(\nu))$. As was stated in the proof of Proposition \ref{prop_7.1}, the tableau $x \in \operatorname{SSSYT}(\nu, N)$ is of type $2$ iff the triplet $(\gamma, \mu_i, \mu_{i+1})$ is of type $2$, where $(\mu_1, \ldots, \mu_i, \mu_{i+1}, \ldots \mu_N) = (\Cr{I}_{\lambda}^{-1}\circ \Cr{Y}_{\lambda}^{-1}) (x)$ and $\gamma = \sum_{k=1}^{i-1} \mu_k$.

Then, the theorem is equivalent to Corollary \ref{cor_6.4} via the correspondence between elements of $\operatorname{Sing}_{\lambda}(\Cr{B}_S^{\tens N})$ and $\operatorname{SSSYT}(\nu, N)$ given by the bijection $\Cr{H}_{\lambda,N} \circ \Cr{I}_{\lambda}^{-1}\circ \Cr{Y}_{\lambda}^{-1} $. 
\end{proof}

Finally, we present an example that illustrates the action of $\tau_i$ on the type $2$ tableaux described in Theorem \ref{th_7.3}

\begin{example}
\label{ex_7.3}
Suppose the first two columns of the tableau below (left) contain exactly $i+1$ cells. Then the automorphism ${\tau_i}$ maps it to the tableau on the right.
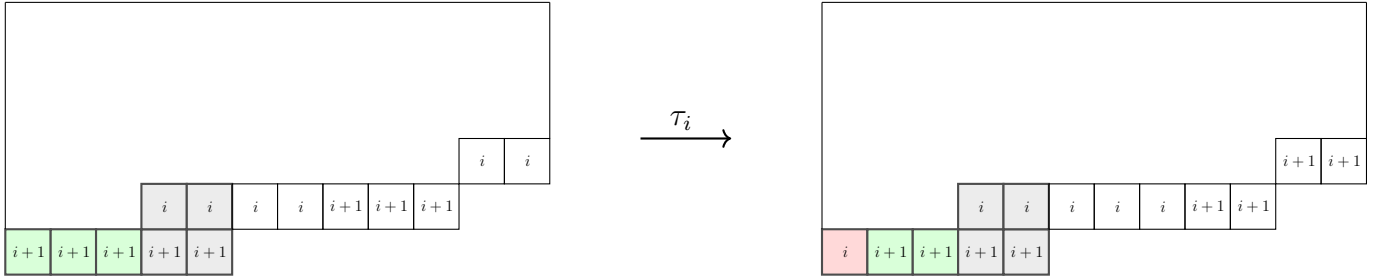
\vspace{2mm}  
\begin{center}
\begin{tikzpicture}[scale = 0.6]
\filldraw[color=black!70, fill=green!15,  thick] (2,0) rectangle (3,1);
\draw (3, 0.5) circle (0pt)  node[scale = 0.5, anchor=east]{$i+1$};
\filldraw[color=black!70, fill=green!15,  thick] (3,0) rectangle (4,1);
\draw (4, 0.5) circle (0pt)  node[scale = 0.5, anchor=east]{$i+1$};
\filldraw[color=black!70, fill=green!15,  thick] (4,0) rectangle (5,1);
\draw (5, 0.5) circle (0pt)  node[scale = 0.5, anchor=east]{$i+1$};
\filldraw[color=black!70, fill=gray!15,  thick] (5,0) rectangle (6,1);
\draw (6, 0.5) circle (0pt)  node[scale = 0.5, anchor=east]{$i+1$};
\filldraw[color=black!70, fill=gray!15,  thick](6,0) rectangle (7,1);
\draw (7, 0.5) circle (0pt)  node[scale = 0.5, anchor=east]{$i+1$};
\filldraw[color=black!70, fill=gray!15,  thick] (5,1) rectangle (6,2);
\draw (5.7, 1.5) circle (0pt)  node[scale = 0.5, anchor=east]{$i$};
\filldraw[color=black!70, fill=gray!15,  thick] (6,1) rectangle (7,2);
\draw (6.7, 1.5) circle (0pt)  node[scale = 0.5, anchor=east]{$i$};
\draw (7,1) rectangle (8,2);
\draw (7.7, 1.5) circle (0pt)  node[scale = 0.5, anchor=east]{$i$};
\draw (8,1) rectangle (9,2);
\draw (8.7, 1.5) circle (0pt)  node[scale = 0.5, anchor=east]{$i$};
\draw (9,1) rectangle (10,2);
\draw (10, 1.5) circle (0pt)  node[scale = 0.5, anchor=east]{$i+1$};
\draw (10,1) rectangle (11,2);
\draw (11, 1.5) circle (0pt)  node[scale = 0.5, anchor=east]{$i+1$};
\draw (11,1) rectangle (12,2);
\draw (12, 1.5) circle (0pt)  node[scale = 0.5, anchor=east]{$i+1$};

\draw (12,2) rectangle (13,3);
\draw (12.7, 2.5) circle (0pt)  node[scale = 0.5, anchor=east]{$i$};
\draw (13,2) rectangle (14,3);
\draw (13.7, 2.5) circle (0pt)  node[scale = 0.5, anchor=east]{$i$};

\draw (2,1) -- (2,6);
\draw (2,6) -- (14,6);
\draw (14,6) -- (14,3);

\draw[thick,->] (16,3) -- (18,3);
\filldraw[black] (16.40,3.40) circle (0pt) node[anchor=west]{$\tau_i$};

\filldraw[color=black!70, fill=red!15,  thick] (20,0) rectangle (21,1);
\draw (20.7, 0.5) circle (0pt)  node[scale = 0.5, anchor=east]{$i$};
\filldraw[color=black!70, fill=green!15,  thick] (21,0) rectangle (22,1);
\draw (22, 0.5) circle (0pt)  node[scale = 0.5, anchor=east]{$i+1$};
\filldraw[color=black!70, fill=green!15,  thick] (22,0) rectangle (23,1);
\draw (23, 0.5) circle (0pt)  node[scale = 0.5, anchor=east]{$i+1$};
\filldraw[color=black!70, fill=gray!15,  thick] (23,0) rectangle (24,1);
\draw (24, 0.5) circle (0pt)  node[scale = 0.5, anchor=east]{$i+1$};
\filldraw[color=black!70, fill=gray!15,  thick](24,0) rectangle (25,1);
\draw (25, 0.5) circle (0pt)  node[scale = 0.5, anchor=east]{$i+1$};
\filldraw[color=black!70, fill=gray!15,  thick] (23,1) rectangle (24,2);
\draw (23.7, 1.5) circle (0pt)  node[scale = 0.5, anchor=east]{$i$};
\filldraw[color=black!70, fill=gray!15,  thick] (24,1) rectangle (25,2);
\draw (24.7, 1.5) circle (0pt)  node[scale = 0.5, anchor=east]{$i$};
\draw (25,1) rectangle (26,2);
\draw (25.7, 1.5) circle (0pt)  node[scale = 0.5, anchor=east]{$i$};
\draw (26,1) rectangle (27,2);
\draw (26.7, 1.5) circle (0pt)  node[scale = 0.5, anchor=east]{$i$};
\draw (27,1) rectangle (28,2);
\draw (27.7, 1.5) circle (0pt)  node[scale = 0.5, anchor=east]{$i$};
\draw (28,1) rectangle (29,2);
\draw (29, 1.5) circle (0pt)  node[scale = 0.5, anchor=east]{$i+1$};
\draw (29,1) rectangle (30,2);
\draw (30, 1.5) circle (0pt)  node[scale = 0.5, anchor=east]{$i+1$};

\draw (30,2) rectangle (31,3);
\draw (31, 2.5) circle (0pt)  node[scale = 0.5, anchor=east]{$i+1$};
\draw (31,2) rectangle (32,3);
\draw (32, 2.5) circle (0pt)  node[scale = 0.5, anchor=east]{$i+1$};

\draw (20,1) -- (20,6);
\draw (20,6) -- (32,6);
\draw (32,6) -- (32,3);

\end{tikzpicture}
\end{center}

Non-free cells are highlighted in gray, free cells of the last row are highlighted in green. The first cell of the last row of the right tableau, where the number changed, is highlighted in red. This example illustrates the automorphism $\tau_i$ on tableaux satisfying the conditions of \textbf{Possibility} $2$ in Proposition \ref{prop_7.1}.
\end{example}

\renewcommand\refname{References}

\end{document}